\documentclass[nosumlimits,twoside]{amsart}
\usepackage{amsmath, amssymb}
\usepackage{amsfonts}
\usepackage[backref=page]{hyperref}
\usepackage{graphicx}
\usepackage{float}
\usepackage{srcltx}
\usepackage[all, v2, cmtip, 2cell]{xy}
\UseAllTwocells
\usepackage{cancel}
\usepackage{version}
\usepackage[T1]{fontenc}
\usepackage{xcolor}
\usepackage{enumerate}
\usepackage[inline]{enumitem}
\usepackage[normalem]{ulem}
\usepackage{latexsym}
\usepackage[2emode]{psfrag}
\usepackage{yhmath}
\usepackage{array}
\usepackage{dsfont}
\usepackage{nicefrac}
\usepackage[colorinlistoftodos,prependcaption,textsize=tiny]{todonotes}
\usepackage[capitalise,noabbrev]{cleveref}
\usepackage{mathtools}
\usepackage{etoolbox}
\usepackage{tikz-cd}
\usepackage{orcidlink}

\makeatletter
\patchcmd{\@setaddresses}{\indent}{\noindent}{}{}
\patchcmd{\@setaddresses}{\indent}{\noindent}{}{}
\patchcmd{\@setaddresses}{\indent}{\noindent}{}{}
\patchcmd{\@setaddresses}{\indent}{\noindent}{}{}
\makeatother

\newtheorem{theorem}{\sc Theorem}[section]
\newtheorem{proposition}[theorem]{\sc Proposition}

\newtheorem{lemma}[theorem]{\sc Lemma}
\newtheorem{corollary}[theorem]{\sc Corollary}
\theoremstyle{definition}
\newtheorem{definition}[theorem]{\sc Definition}

\newtheorem{example}[theorem]{\sc Example}

\theoremstyle{remark}
\newtheorem{remark}[theorem]{\sc Remark}

\allowdisplaybreaks
\newenvironment{invisible}{{\noindent\sc \colorbox{yellow}{Invisible:}\;}\color{gray}}{\medskip}
\excludeversion{invisible}

\def\pulb{\ar@{}[dr]|(0.2){\mbox{\Large{$\lrcorner$}}}}
\IfFileExists{tcilatex.tex}{\input{tcilatex}}{}

\newcommand{\id}{\mathrm{Id}}

\newcommand{\Aa}{\mathcal{A}}
\newcommand{\Bb}{\mathcal{B}}
\newcommand{\Cc}{\mathcal{C}}

\newcommand{\Ff}{\mathcal{F} }
\newcommand{\Gg}{\mathcal{G}}
\newcommand{\Hh}{\mathcal{H}}
\newcommand{\Mm}{\mathcal{M}}

\newcommand{\Qq}{\mathcal{Q}}
\newcommand{\Rr}{\mathcal{R} }
\newcommand{\Ss}{\mathcal{S} }
\newcommand{\Tt}{\mathcal{T}}
\newcommand{\Uu}{\mathcal{U}}
\newcommand{\Vv}{\mathcal{V}}
\newcommand{\Ww}{\mathcal{W}}

\newcommand{\ot}{\otimes}
\newcommand{\hc}{\circ_h} 
\newcommand{\vc}{\circ_v} 

\def\Bialg{{\sf Bialg}}
\def\Bialgcc{{\sf Bialg_{cc}}}

\def\TwBialg{{\sf Tw}}
\def\TrBialg{{\sf Tr}}
\def\TwTrBialg{{\sf TwTr}}

\def\Vec{{\sf Vec}}

\def\Cat{{\sf Cat}}

\newcommand{\op}{\mathrm{op}}
\newcommand{\cop}{\mathrm{cop}}

\begin{document}

\begin{abstract}
    It is well-known that the tensor product of two bialgebras constitutes the binary product in the category of cocommutative bialgebras and morphisms of bialgebras between them. In this paper, we extend this result to triangular bialgebras and twisted morphisms of triangular bialgebras. We do so by adopting the framework of 2-categories and the proper notion of binary product, as well as by employing a description of twists on the tensor product bialgebra, specifically developed for this purpose. We apply this extension to provide a new interpretation of the twisted tensor products of triangular bialgebras in terms of binary products. 
\end{abstract}

\title[\tiny The binary product in the 2-category of triangular bialgebras and twisted morphisms]{The binary product in the 2-category of triangular bialgebras and twisted morphisms}

\author[A.~Ardizzoni]{Alessandro Ardizzoni\, \orcidlink{0000-0001-7384-611X}}
\address{%
\parbox[b]{\linewidth}{University of Turin, Department of Mathematics ``G. Peano'', via
Carlo Alberto 10, I-10123 Torino, Italy}}
\email{alessandro.ardizzoni@unito.it}
\urladdr{\url{www.sites.google.com/site/aleardizzonihome}}

\author[A.~Sciandra]{Andrea Sciandra\, \orcidlink{0009-0008-0447-287X}}
\address{%
\parbox[b]{\linewidth} {Département de Mathématiques, Université Libre de Bruxelles, Boulevard du Triomphe, B-1050 Bruxelles, Belgium}}
\email{andrea.sciandra@ulb.be}
\urladdr{\url{www.andreasciandra.com}}

 \keywords{Triangular bialgebras, twists, 2-categories, binary products.}
\subjclass[2020]{Primary 16T10; Secondary 18A30, 18N10}

\maketitle

\tableofcontents

\section{Introduction}

Among general monoidal categories, the cartesian ones have several important properties. These are monoidal categories where the tensor product is provided by the binary product $\times$ and the unit object by the terminal object $\mathsf{1}$. An object $C$ in a cartesian category $\Cc$ automatically has morphisms $\Delta:C\to C\times C$ and $\varepsilon:C\to\mathsf{1}$ such that $(C,\Delta,\varepsilon)$ becomes a comonoid in $\Cc$. Therefore, bimonoid objects in $\Cc$ coincide with monoid objects and Hopf monoid objects with \textit{internal groups} in $\Cc$. Under suitable assumptions on the category $\Cc$, the category $\mathsf{Grp}(\Cc)$ of internal groups in $\Cc$ shares many beautiful properties with the category of groups. For instance, if $\Cc$ has finite limits, then $\mathsf{Grp}(\Cc)$ is protomodular \cite{Bourn}, i.e.\ the Split Short Five Lemma holds. Moreover, if $\mathsf{Grp}(\Cc)$ is semi-abelian \cite{JMT}, i.e.\ it is also exact in the sense of \cite{Barr} and has binary coproducts, then it is also action representable \cite{BJK} whenever $\Cc$ is cartesian closed.

This happens considering the cartesian monoidal category of cocommutative coalgebras, where the internal groups are given by cocommutative Hopf algebras that constitute a semi-abelian category \cite{GSV}. Since the tensor product of two cocommutative Hopf algebras coincides with their binary product, one can give an explicit description of the Huq commutator of two normal Hopf subalgebras of a given cocommutative Hopf algebra and also recover the abelian category of commutative and cocommutative Hopf algebras as the category of abelian objects inside the category of cocommutative Hopf algebras \cite{GSV}.

It is noteworthy that, even if we drop out the existence of an antipode, the resulting category of cocommutative bialgebras remains cartesian. We cannot make a further step by simply removing the assumption of cocommutativity, as in the category of all bialgebras the binary products do exist, but their construction is generally neither straightforward \cite{Agore} nor related to the tensor products.
\begin{invisible}
Per giustificare l'ultima frase bisogna dare un controesempio.
    Nell'articolo con Menini-Saracco daremo un esempio del caso duale. Se $A=\Bbbk[X]$ e $B=\Bbbk[Y]$ con $X,Y$ primitivi, siccome il funtore bialgebra tensoriale $T:\Vec\to \Bialg$ preserva i coprodotti (è un aggiunto sinistro di $P:\Bialg\to \Vec$), allora $A\coprod B=T(\Bbbk X)\coprodT(\Bbbk Y)\cong T(\Bbbk X\oplus \Bbbk Y)=\Bbbk\{X,Y\}$. Questa non è commutativa e quindi non può essere isomorfa a $A\ot B=\Bbbk[X]\otimes\Bbbk[Y]\cong \Bbbk[X,Y]$.
\end{invisible}
Instead, once observed that cocommutative bialgebras are particular instances of triangular bialgebras, we can ask ourselves whether a suitable notion of binary product exists for triangular bialgebras, and if it ultimately coincides with the corresponding tensor product. To this aim, we first observe that, given two (quasi)triangular bialgebras $(H_{1},\Rr_{1})$ and $(H_{2},\Rr_{2})$, the tensor product bialgebra $H_{1}\ot H_{2}$ becomes (quasi)triangular through 
\[
\widetilde{\Rr}:=(\id \ot\tau\ot\id )(\Rr_{1}\ot\Rr_{2}). 
\]
However, $(H_{1}\ot H_{2},\widetilde{\Rr})$ is not the binary product of $(H_{1},\Rr_{1})$ and $(H_{2},\Rr_{2})$ in the category of (quasi)triangular bialgebras and morphism of (quasi)triangular bialgebras. 
In fact, the main aim of this paper is to prove that $(H_{1}\ot H_{2},\widetilde{\Rr})$ can be structured as a binary product in the $2$-category of triangular bialgebras and twisted morphisms of triangular bialgebras whose definition is inspired by \cite{Davydov}. \medskip

More precisely, the content of the paper is the following. In Section \ref{sec:preliminaries}, we recall some preliminary notions concerning twists and quasitriangular structures that will be useful throughout the paper. It is known that, given a bialgebra $(H,m,u,\Delta,\varepsilon)$ and a twist $\Ff\in H\ot H$ as in the seminal paper of Drinfel'd \cite{Dr87}, one obtains a bialgebra $H_{\Ff}:=(H,m,u,\Delta_{\Ff},\varepsilon)$, where $\Delta_{\Ff}(\cdot):=\Ff\Delta(\cdot)\Ff^{-1}$ and, if $(H,\Rr)$ is (quasi)triangular, also $H_{\Ff}$ is (quasi)triangular with $\Rr_{\Ff}:=\Ff^{\mathrm{op}}\Rr\Ff^{-1}$. In Section \ref{sec:twisttensorproduct}, we provide a characterization of twists on the tensor product bialgebra, up to cohomology, that we consider of independent interest.  Explicitly, in Theorem \ref{thm:twistot} and Lemma \ref{lem:Drinfeldtwist}, we are able to prove that a twist on $H_{1}\ot H_{2}$ is cohomologous to a twist of the form
\[
(1\ot\Rr^{-1}\ot1)\big((\id \ot\tau\ot\id )(\Ff_{1}\ot\Ff_{2})\big),
\]
where $\Ff_{1}\in H_{1}\ot H_{1}$ and $\Ff_{2}\in H_{2}\ot H_{2}$ are twists on $H_{1}$ and $H_{2}$, respectively, and $\Rr\in H_{2}\ot H_{1}$ is a weak $\Rr$-matrix of $((H_2)_{\Ff_2},(H_1)_{\Ff_1})$ in the sense of \cite{Chen-quasi}. Remarkably, if, in particular, one considers a quasitriangular structure on $H_{1}\ot H_{2}$, then also $\Ff_{1}$ and $\Ff_{2}$ become quasitriangular, and the weak $\Rr$-matrix turns out to be central, see Proposition \ref{pro:Rot}. 

In Section \ref{sec:2-categories}, we adopt the 2-categorical language and in Section \ref{sec:binaryproduct} we extend the binary product construction of cocommutative bialgebras to triangular bialgebras by employing the proper notion of binary product in a 2-category. Given two objects $A,B$ in a 2-category $\Cc$, the binary product is an object $A\times B$ in $\Cc$ equipped with two projections $p:A\times B\to A$ and $q:A\times B\to B$ such that, for any object $X$ in $\Cc$, the induced functor $\Cc(X,A\times B)\to \Cc(X,A)\times \Cc(X,B)$ is an equivalence of categories, where $\Cc(X,Y)$ denotes the category of $1$-cells from $X$ to $Y$, for any object $Y$ in $\Cc$.  
Given a quasitriangular bialgebra $(H,\Rr)$ and bialgebra maps $f_{1}:H\to H_{1}$ and $f_{2}:H\to H_{2}$, we obtain a bialgebra map
\begin{equation}\label{diagmorph}
f:=(f_{1}\ot f_{2})\Delta_{H}:H\to(H_{1}\ot H_{2})_{\Ff},\  \text{where}\ \Ff:=1_{H_{1}}\ot(f_{2}\ot f_{1})(\Rr^{-1})\ot1_{H_{2}} 
\end{equation}
is a twist of $H_{1}\ot H_{2}$, i.e.\ $(f,\Ff):H\to(H_{1}\ot H_{2})_{\Ff}$ is a twisted morphism of bialgebras in the sense of Davydov \cite{Davydov}. In loc.\ cit.\ it is noticed that bialgebras and twisted morphisms of bialgebras form a 2-category $\TwBialg$, where, given twisted morphisms of bialgebras $(f',\Ff'),(f,\Ff):H\to H'$, a $2$-cell is a so-called gauge transformation $a:(f,\Ff)\Rightarrow(f',\Ff')$ which consists of an element $a\in H'$ with $\varepsilon'(a)=1$ such that the following equalities are satisfied
\[
(a\ot a)\Ff=\Ff'\Delta'(a),\qquad af(x)=f'(x)a.
\]
Taking the $2$-category $\TwBialg$ as a model, we introduce the 2-category $\TwTrBialg$ of triangular bialgebras and twisted morphisms: 0-cells are given by triangular bialgebras, 1-cells are pairs $(f,\Ff):(H,\Rr)\to(H',\Rr')$ such that $f:(H,\Rr)\to(H'_{\Ff},\Rr'_{\Ff})$ is a morphism of triangular bialgebras (and we call them twisted morphisms of triangular bialgebras) and 2-cells are given by gauge transformations, that are automatically compatible with the triangular structures involved. In Theorem \ref{thm:Tw2cart}, we prove that the triangular bialgebra $(H_{1}\ot H_{2},\widetilde{\Rr})$ is the binary product of the triangular bialgebras $(H_{1},\Rr_{1})$ and $(H_{2},\Rr_{2})$ in the 2-category $\TwTrBialg$, where the projections are given by $(\id \ot\varepsilon,1\ot1):(H_{1}\ot H_{2},\widetilde{\Rr})\to(H_{1},\Rr_{1})$ and $(\varepsilon\ot\id ,1\ot1):(H_{1}\ot H_{2},\widetilde{\Rr})\to(H_{2},\Rr_{2})$. Moreover, cf.\ Proposition \ref{pro:Tw2cart1}, given $(f_{1},\Ff_{1}):(H,\Rr)\to(H_{1},\Rr_{1})$ and $(f_{2},\Ff_{2}):(H,\Rr)\to(H_{2},\Rr_{2})$ in $\TwTrBialg$, the diagonal morphism $(f,\Ff):(H,\Rr)\to(H_{1}\ot H_{2},\widetilde{\Rr})$ is defined as in \eqref{diagmorph}.
Since $\TwTrBialg$ has terminal object given by $(\Bbbk,1\ot1)$ (see Lemma \ref{lem:termobj}), then it has all finite products and, moreover, it becomes a symmetric monoidal 2-category. The attached classifying category, obtained by identifying isomorphic $1$-cells, turns out to be a cartesian monoidal category.

As an application of our results, in Section \ref{sec:applications}, we employ the binary product construction to offer a new perspective on the twisted tensor products of triangular bialgebras, as recently introduced in \cite{Wu-Zhou}. \medskip

\noindent\textit{Notations and conventions}. All vector spaces are understood to be $\Bbbk$-vector spaces, where $\Bbbk$ is an arbitrary base field. By a linear map we mean a $\Bbbk$-linear map and the unadorned tensor product $\otimes$ is the one of vector spaces. We denote the canonical flip of vector spaces by $\tau$ and the set of multiplicative invertible elements in a ring $H$ by $H^\times$.

\section{Preliminaries}\label{sec:preliminaries}

Quasitriangular bialgebras were introduced in the seminal paper of Drinfel'd \cite{Dr87} and immediately became central in the theory of quantum groups. In this section, we recall some preliminary notions and results related to quasitriangular bialgebras and we mainly refer the reader to \cite{Kassel} and \cite{Majid-book} for more details about them. \medskip

From now on, $H$ will always denote a bialgebra over an arbitrary field $\Bbbk$.
In the following, given an element $T\in H\otimes H$, we adopt the short notation $T=\sum_{i}{T^{i}\otimes T_{i}}=T^{i}\otimes T_{i
}$, the summation being understood.
We set $T^\op=T_i\otimes T^i.$ 
If $T$ is invertible, we write $T^{-1}=\overline T=\overline{T}^{i}\otimes\overline{T}_{i}$. Moreover, we employ the leg notation $T_{12}=T\otimes1_{H}$, $T_{23}=1_{H}\otimes T$, $T_{13}=T^{i}\ot1_{H}\ot T_{i}$. 

\begin{definition}[\cite{Dr87}]
A bialgebra $H$ is said to be \emph{quasitriangular} if there is an invertible element $\Rr\in H\otimes H$, the \emph{universal $\Rr$-matrix} or \emph{quasitriangular structure}, such that $H$ is quasi-cocommutative, i.e.,
\begin{equation}\label{qtr1}
    \Delta^\mathrm{op}(\cdot)=\Rr\Delta(\cdot)\Rr^{-1},
\end{equation}
and the \emph{hexagon equations}
\begin{align}
    (\id _H\otimes\Delta)(\Rr)&=\Rr_{13}\Rr_{12},\label{qtr2}\\
    (\Delta\otimes\id _H)(\Rr)&=\Rr_{13}\Rr_{23},\label{qtr3}
\end{align}
are satisfied. If in addition 
\begin{equation}\label{tr}
\Rr ^{-1}=\Rr ^\op,\end{equation} 
then $(H,\Rr)$ is called \emph{triangular}. 

A morphism $f:(H,\Rr)\to(H',\Rr')$ between (quasi)triangular bialgebras $(H,\Rr)$ and $(H',\Rr')$ is a morphism of bialgebras $f:H\to H'$ such that $(f\ot f)(\Rr)=\Rr'$. 
In what follows, the category of triangular bialgebras and morphisms between triangular bialgebras will be denoted by $\TrBialg$.
\end{definition}

\begin{remark}
Recall, see e.g.\ \cite[Theorem VIII.2.4]{Kassel},  that a quasitriangular bialgebra $(H,\Rr)$ satisfies the \textit{quantum Yang-Baxter} equation
\begin{equation}\label{eq:QYB}
\Rr_{12}\Rr_{13}\Rr_{23}=\Rr_{23}\Rr_{13}\Rr_{12},
\end{equation}
and $(\varepsilon\otimes\id _{H})(\Rr^{\pm 1})=1_{H}=(\id _{H}\otimes\varepsilon)(\Rr^{\pm 1})$.
\end{remark}

We also recall the notion of twist, which is attributed to Drinfel'd, see \cite[page 909]{Dr87}, and that will be central in the following.

\begin{definition}\label{def:twist}
An invertible element $\Ff\in H\otimes H$ is said to be a \emph{twist} on $H$ if the $2$-cocycle condition and the normalization properties
\begin{align}
    (\Ff\otimes 1_{H})(\Delta\otimes\id _H)(\Ff)&=(1_{H}\otimes\Ff)(\id _H\otimes\Delta)(\Ff),\label{2-cocy}\\
    (\varepsilon\otimes\id _H)(\Ff)&=1_{H}=(\id _H\otimes\varepsilon)(\Ff),\label{norm}
\end{align}
are satisfied. 
\end{definition}

Given a twist $\Ff$ one can consider the linear map $\Delta_\Ff\colon H\rightarrow H\otimes H$ defined via
\begin{equation}
\label{eq:DeltaF}\Delta_\Ff(\cdot):=\Ff\Delta(\cdot)\Ff^{-1}.
\end{equation}
Then, $H_\Ff:=(H,m,u,\Delta_\Ff,\varepsilon)$ is a bialgebra. 
Moreover, if $(H,\Rr)$ is a (quasi)triangular bialgebra, so is $H_\Ff$ with universal $\Rr$-matrix $\Rr_\Ff:=\Ff^\op\Rr\Ff^{-1}$, cf.\ \cite[Theorem 2.3.4]{Majid-book}.

If $\Ff$ is a twist and $h\in H^\times$, we set 
\begin{equation}
\label{def:F^h}
\Ff^h\coloneqq\left( h\otimes h\right) \Ff 
\Delta \left(  h ^{-1}\right).    
\end{equation}

Two Drinfeld twists $\Ff'$, $\Ff$ on $H$ are \emph{cohomologous} if $(\Ff')^h=\Ff$ for some $h\in H^\times$, see \cite[Proposition 2.3.3]{Majid-book}. In this case,  the map $\hat{h}
:=h\left( -\right)  h
^{-1}:H _{\Ff '}\rightarrow H _{\Ff }$ turns out to be a bialgebra isomorphism, see the proof of \cite[Proposition 2.3.5]{Majid-book}, where the domain and codomain seem to be reversed. 

\begin{example}
\label{exa:Rtwist}
Observe that every quasitriangular structure $\Rr$ on a bialgebra $H$ is a twist. In fact, $\Rr$ is normalized and satisfies the 2-cocycle property since
\[
\Rr_{12}(\Delta\otimes\id _H)(\Rr)\overset{\eqref{qtr3}}=\Rr_{12}\Rr_{13}\Rr_{23}\overset{\eqref{eq:QYB}}{=}\Rr_{23}\Rr_{13}\Rr_{12}\overset{\eqref{qtr2}}=\Rr_{23}(\id _H\otimes\Delta)(\Rr).\]
Then, by \eqref{qtr1}, we have $\Delta_\Rr(\cdot)=\Rr\Delta(\cdot)\Rr^{-1}=\Delta^\op(\cdot)$ so that $H_\Rr=H^\cop$, see \cite[Example 2.3.6]{Majid-book}.
\end{example}

\section{Twists on the tensor product bialgebra}\label{sec:twisttensorproduct}

In this section, we provide a classification result, up to cohomology, for the twists on the tensor product bialgebra that will be useful in the following. This result may be known, but we have not found any reference in the literature.

We need the following weakening of the notion of universal $\Rr$-matrix.

\begin{definition}[{\cite[Definition 1.1]{Chen-quasi}}]
Let $A$ and $B$ be bialgebras. An invertible element $\Rr\in A\otimes B$, is called \emph{weak $\Rr$-matrix of $(A,B)$} if
the \emph{weak hexagon equations}
\begin{align}
  (\id _A\otimes\Delta_B)(\Rr)&=\Rr_{13}\Rr_{12},\label{wtr2}\\
    (\Delta_A\otimes\id _B)(\Rr)&=\Rr_{13}\Rr_{23},\label{wtr3}
\end{align}
are satisfied.
\end{definition}

We recall that $(\varepsilon_{A}\ot\id )(\Rr)=1_{B}$ and $(\id \ot\varepsilon_{B})(\Rr)=1_{A}$, see \cite[Lemma 1.2]{Chen-quasi}. 

\begin{lemma}\label{lem:weakop}
Let $A$ and $B$ be bialgebras. The following are equivalent for an invertible element $\Rr \in A\otimes B$. 
\begin{enumerate}
    \item[(i)] $\Rr $ is a (resp.\ central) weak $\Rr $-matrix of $\left(
A,B\right) .$ 

\item[(ii)]  $\Rr ^{\mathrm{op}}$ is a  (resp.\ central) weak $%
\Rr $-matrix of $\left( B^{\mathrm{op}},A^{\mathrm{op}}\right) $ (resp.\ $\left(B,A\right) $).

\item[(iii)]  $\Rr ^{-1}$ is a  (resp.\ central) weak $
\Rr $-matrix of $\left( A^{\mathrm{cop}},B^{\mathrm{cop}}\right) $  (resp.\ $\left(A,B\right) $).
\end{enumerate}

\end{lemma}

\begin{proof}
$(i)\Rightarrow (ii)$. From \eqref{wtr2} we get $\Rr ^{i}\otimes \Delta _{B}\left( \Rr_{i}\right) =\Rr ^{i}\Rr ^{j}\otimes \Rr _{j}\otimes
\Rr _{i}$ and so
\begin{eqnarray*}
\left( \Delta _{B}\otimes \id _{A}\right) \left( \Rr ^{\mathrm{%
op}}\right)  &=&\Delta_{B} \left( \Rr _{i}\right) \otimes \Rr ^{i}=%
\Rr _{j}\otimes \Rr _{i}\otimes \Rr ^{i}\Rr ^{j}
\\
&=&\Rr _{j}\otimes \Rr _{i}\otimes (\Rr ^{j}\cdot _{%
\mathrm{op}}\Rr ^{i})=\Rr _{13}^{\mathrm{op}}\cdot _{\mathrm{op}%
}\Rr _{23}^{\mathrm{op}}.
\end{eqnarray*}%
Similarly, from \eqref{wtr3} we get $\Delta _{A}\left( \Rr ^{i}\right)
\otimes \Rr _{i}=\Rr ^{i}\otimes \Rr ^{j}\otimes
\Rr _{i}\Rr _{j}$ and hence%
\begin{eqnarray*}
\left( \id _{B}\otimes \Delta _{A}\right) \left( \Rr ^{\mathrm{%
op}}\right)  &=&\Rr _{i}\otimes \Delta _{A}\left( \Rr %
^{i}\right) =\Rr _{i}\Rr _{j}\otimes \Rr ^{i}\otimes
\Rr ^{j} \\
&=&(\Rr _{j}\cdot _{\mathrm{op}}\Rr _{i})\otimes \Rr %
^{i}\otimes \Rr ^{j}=\Rr _{13}^{\mathrm{op}}\cdot _{\mathrm{op}%
}\Rr _{12}^{\mathrm{op}}.
\end{eqnarray*}
Note that, when $\Rr$ is central, we immediately have 
$\left( \Delta _{B}\otimes \id _{A}\right) \left( \Rr ^{\mathrm{%
op}}\right)  =
\Rr _{j}\otimes \Rr _{i}\otimes \Rr ^{i}\Rr ^{j}
=
\Rr _{j}\otimes \Rr _{i}\otimes \Rr ^{j}\Rr ^{i}
=\Rr_{13}^\op\Rr_{23}^\op$ and similarly 
$\left( \id _{B}\otimes \Delta _{A}\right) \left( \Rr ^{\mathrm{%
op}}\right) =\Rr _{13}^{\mathrm{op}}\Rr _{12}^{\mathrm{op}}$ so that $\Rr^\op$ is a central weak $\Rr$-matrix of $(B,A)$.

$(ii)\Rightarrow (i)$. Apply the converse implication to $\Rr^\op.$

$(i)\Rightarrow (iii)$.
As in \cite[Lemma 1.2]{Chen-quasi}, from \eqref{wtr2} we get $\left( \id %
_{A}\otimes \Delta _{B}\right) \left( \Rr ^{-1}\right) =\Rr %
_{12}^{^{-1}}\Rr _{13}^{^{-1}}$ whence $\left( \id %
_{A}\otimes \Delta _{B}^{\mathrm{op}}\right) \left( \Rr ^{-1}\right) =%
\Rr _{13}^{^{-1}}\Rr _{12}^{^{-1}}$. Similarly, from \eqref{wtr3} we get $\left( \Delta _{A}\otimes \id %
_{B}\right) \left( \Rr ^{-1}\right) =\Rr _{23}^{-1}\Rr %
_{13}^{-1}$ whence $\left( \Delta _{A}^{\mathrm{op}}\otimes \id %
_{B}\right) \left( \Rr ^{-1}\right) =\Rr _{13}^{-1}\Rr %
_{23}^{-1}$. Note that, if $\Rr$ is central, we immediately have $\left( \id %
_{A}\otimes \Delta _{B}\right) \left( \Rr ^{-1}\right)=\Rr %
_{12}^{^{-1}}\Rr _{13}^{^{-1}}=%
\Rr _{13}^{^{-1}}\Rr _{12}^{^{-1}}$ and $\left( \Delta _{A}\otimes \id %
_{B}\right) \left( \Rr ^{-1}\right) =\Rr _{23}^{-1}\Rr %
_{13}^{-1}=\Rr _{13}^{-1}\Rr %
_{23}^{-1}$ so that $\Rr^{-1}$ is a central weak $\Rr$-matrix of $(A,B)$.

$(iii)\Rightarrow (i)$. Apply the converse implication to $\Rr^{-1}.$
\end{proof}

\begin{example}
\label{exa:weakR}
Let $\alpha:H\to A$ and $\beta:H\to B$ be bialgebra maps and let $\Rr\in H\otimes H$ be an $\Rr$-matrix. Then $(\alpha\otimes \beta)(\Rr)$ is a weak $\Rr$-matrix of $(A,B)$. To see this, apply $\alpha\otimes\beta\otimes\beta $ to \eqref{qtr2} and apply $\alpha\otimes\alpha\otimes\beta $ to \eqref{qtr3} and observe that the inverse of $(\alpha\otimes \beta)(\Rr)$ is $(\alpha\otimes \beta)(\Rr^{-1})$.
\end{example}

\begin{lemma}
Let $\Rr \in A\otimes B$ be a weak $\Rr$-matrix of $(A,B).$ Then%
\begin{equation}
\label{eq:DDR}
(\Delta _{A}\otimes \Delta _{B})(\Rr )=\Rr ^{i}\Rr %
^{p}\otimes \Rr ^{j}\Rr ^{s}\otimes \Rr _{p}\Rr %
_{s}\otimes \Rr _{i}\Rr _{j}.
\end{equation}%
\end{lemma}
\begin{proof}
We compute
\begin{align*}
(\Delta _{A}\otimes \Delta _{B})(\Rr ) &= (\Delta _{A}\otimes \mathrm{%
Id}_{B\otimes B})(\id \otimes \Delta _{B})(\Rr ) \\
&\overset{\mathclap{\eqref{wtr2}}}{=}(\Delta _{A}\otimes \id _{B\otimes B})(%
\Rr _{13}\Rr _{12}) \\
&=(\Delta _{A}\otimes \id _{B\otimes B})(\Rr _{13})(\Delta
_{A}\otimes \id _{B\otimes B})(\Rr _{12}) \\
&\overset{\mathclap{\eqref{wtr3}}}{=}(\Rr ^{i}\otimes \Rr ^{j}\otimes
1\otimes \Rr _{i}\Rr _{j})(\Rr ^{p}\otimes \Rr %
^{s}\otimes \Rr _{p}\Rr _{s}\otimes 1) \\
&=\Rr ^{i}\Rr ^{p}\otimes \Rr ^{j}\Rr %
^{s}\otimes \Rr _{p}\Rr _{s}\otimes \Rr _{i}\Rr %
_{j}.\qedhere
\end{align*}
\end{proof}

\subsection{Twists up to cohomology}
In what follows, we provide a classification (up to cohomology) for the twists on the tensor product bialgebra. This result will be relevant to later prove Theorem \ref{thm:Tw2cart}. Given bialgebras $H_{1},H_{2}$ and $\Ff %
\in H_{1}\otimes H_{2}\otimes H_{1}\otimes H_{2}$, we set
\begin{eqnarray*}
\Ff _{1} &\coloneqq&\left( \id \otimes \varepsilon \otimes \id %
\otimes \varepsilon \right) \left( \Ff \right) =\left( \id %
\otimes \varepsilon \right) \left( \Ff ^{j}\right) \otimes \left(
\id \otimes \varepsilon \right) \left( \Ff _{j}\right) \in
H_{1}\otimes H_{1}, \\
\Ff _{2} &\coloneqq&\left( \varepsilon \otimes \id \otimes
\varepsilon \otimes \id \right) \left( \Ff \right) =\left(
\varepsilon \otimes \id \right) \left( \Ff ^{j}\right) \otimes
\left( \varepsilon \otimes \id \right) \left( \Ff _{j}\right)
\in H_{2}\otimes H_{2}, \\
\Gg  &\coloneqq&\left( \id \otimes \varepsilon \otimes \varepsilon
\otimes \id \right) \left( \Ff \right) =\left( \id %
\otimes \varepsilon \right) \left( \Ff ^{j}\right) \otimes \left(
\varepsilon \otimes \id \right) \left( \Ff _{j}\right) \in
H_{1}\otimes H_{2}, \\
\Hh  &\coloneqq&\left( \varepsilon \otimes \id \otimes \id %
\otimes \varepsilon \right) \left( \Ff \right) =\left( \varepsilon
\otimes \id \right) \left( \Ff ^{j}\right) \otimes \left(
\id \otimes \varepsilon \right) \left( \Ff _{j}\right) \in
H_{2}\otimes H_{1}.
\end{eqnarray*}
We denote the datum $(\Ff_1,\Ff_2,\Gg,\Hh)$  by $\Phi(\Ff)$.

\begin{theorem}
\label{thm:twistot}
Let $H_{1}$ and $H_{2}$ be bialgebras and let $\Ff %
\in H_{1}\otimes H_{2}\otimes H_{1}\otimes H_{2}$ be a twist on the tensor product bialgebra $%
H_{1}\otimes H_{2}$. Let $(\Ff_1,\Ff_2,\Gg,\Hh):=\Phi(\Ff)$.

Then $\Ff _{1}$ and $\Ff _{2}$ are twists of $H_{1}$ and $H_{2}$, respectively, while $%
\Rr\coloneqq\Gg ^{\mathrm{op}}\Hh ^{-1}\in H_{2}\otimes H_{1}$ is a weak $\Rr$-matrix of $((H_2)_{\Ff_2},(H_1)_{\Ff_1})$.
Moreover, $\Ff^\Gg$ defined as in \eqref{def:F^h}, turns out to be equal to
\begin{equation*}
\left( 1_{H_{1}}\otimes
\Rr ^{-1}\otimes
1_{H_{2}}\right) \big(\left( \id _{H_{1}}\otimes \tau \otimes \id %
_{H_{2}}\right) \left( \Ff _{1}\otimes \Ff _{2}\right)\big), 
\end{equation*}%
and consequently $\Ff $ is cohomologous to it. 
\end{theorem}

\begin{proof}
Define $p_{1}:=\id \otimes \varepsilon :H_{1}\otimes H_{2}\rightarrow
H_{1}$ and $p_{2}:=\varepsilon \otimes \id :H_{1}\otimes
H_{2}\rightarrow H_{2}$.

We have $\left( p_{1}\otimes p_{2}\right) \Delta _{H_{1}\otimes H_{2}}=\mathrm{%
Id}_{H_{1}\otimes H_{2}}$ and $\left( p_{2}\otimes p_{1}\right) \Delta
_{H_{1}\otimes H_{2}}=\tau_{H_{1},H_{2}}.
$ By using these equalities and the fact that $p_{1}$ and $p_{2}$ are
bialgebra maps, we can evaluate $p_{i}\otimes p_{j}\otimes p_{k}$ for $%
ijk=111,112,121,122,211,212,221,222$ on both sides of
\begin{equation}
\label{eq:000}
\left( \Ff \otimes 1_{H_{1}\otimes H_{2}}\right) \left( \Delta
_{H_{1}\otimes H_{2}}\otimes \id _{H_{1}\otimes H_{2}}\right) \left(
\Ff \right) =\left( 1_{H_{1}\otimes H_{2}}\otimes \Ff \right)
\left( \id _{H_{1}\otimes H_{2}}\otimes \Delta _{H_{1}\otimes
H_{2}}\right) \left( \Ff \right)
\end{equation}
to get respectively%
\begin{align}
\label{eq:111} &\left( \Ff _{1}\otimes 1_{H_{1}}\right) \left( \Delta
_{H_{1}}\otimes \id _{H_{1}}\right) \left( \Ff _{1}\right)
=\left( 1_{H_{1}}\otimes \Ff _{1}\right) \left( \id %
_{H_{1}}\otimes \Delta _{H_{1}}\right) \left( \Ff _{1}\right)  \\
\label{eq:112} &\left( \Ff _{1}\otimes 1_{H_{2}}\right) \left( \Delta
_{H_{1}}\otimes \id _{H_{2}}\right) \left( \Gg \right) =\left(
1_{H_{1}}\otimes \Gg \right) \left( p_{1}\otimes \id %
_{H_{1}\otimes H_{2}}\right) \left( \Ff \right)  \\
\label{eq:121} &\left( \Gg \otimes 1_{H_{1}}\right) \left( \id %
_{H_{1}\otimes H_{2}}\otimes p_{1}\right) \left( \Ff \right) =\left(
1_{H_{1}}\otimes \Hh \right) \left( p_{1}\otimes \tau \right) \left(
\Ff \right)  \\
\label{eq:122} &\left( \Gg \otimes 1_{H_{2}}\right) \left( \id %
_{H_{1}\otimes H_{2}}\otimes p_{2}\right) \left( \Ff \right) =\left(
1_{H_{1}}\otimes \Ff _{2}\right) \left( \id _{H_{1}}\otimes
\Delta _{H_{2}}\right) \left( \Gg \right)  \\
\label{eq:211} &\left( \Hh \otimes 1_{H_{1}}\right) \left( \tau \otimes
p_{1}\right) \left( \Ff \right) =\left( 1_{H_{2}}\otimes \Ff %
_{1}\right) \left( \id _{H_{2}}\otimes \Delta _{H_{1}}\right) \left(
\Hh \right)  \\
\label{eq:212} &\left( \Hh \otimes 1_{H_{2}}\right) \left( \tau \otimes
p_{2}\right) \left( \Ff \right) =\left( 1_{H_{2}}\otimes \Gg %
\right) \left( p_{2}\otimes \id _{H_{1}\otimes H_{2}}\right) \left(
\Ff \right)  \\
\label{eq:221}&\left( \Ff _{2}\otimes 1_{H_{1}}\right) \left( \Delta
_{H_{2}}\otimes \id _{H_{1}}\right) \left( \Hh \right) =\left(
1_{H_{2}}\otimes \Hh \right) \left( p_{2}\otimes \tau \right) \left(
\Ff \right)  \\
\label{eq:222} &\left( \Ff _{2}\otimes 1_{H_{2}}\right) \left( \Delta
_{H_{2}}\otimes \id _{H_{2}}\right) \left( \Ff _{2}\right)
=\left( 1_{H_{2}}\otimes \Ff _{2}\right) \left( \id %
_{H_{2}}\otimes \Delta _{H_{2}}\right) \left( \Ff _{2}\right) .
\end{align}%
From \eqref{eq:111} and \eqref{eq:222}, together with $p_{1}$ and $p_{2}$ applied to $\left(
\varepsilon _{H_{1}\otimes H_{2}}\otimes \id _{H_{1}\otimes
H_{2}}\right) \left( \Ff \right) =1_{H_{1}\otimes H_{2}}=\left(
\id _{H_{1}\otimes H_{2}}\otimes \varepsilon _{H_{1}\otimes
H_{2}}\right) \left( \Ff \right) $, we obtain that $\Ff _{1}$ and $%
\Ff _{2}$ are twists.

From \eqref{eq:112}, 
we get%
\begin{align}
\label{eq:p1IdF}
&\left( p_{1}\otimes \id _{H_{1}\otimes H_{2}}\right) \left( \Ff\right)  =\left( 1_{H_{1}}\otimes \Gg ^{-1}\right) \left(\Ff_{1}\otimes 1_{H_{2}}\right) \left( \Delta _{H_{1}}\otimes \id %
_{H_{2}}\right) \left( \Gg \right) .
\end{align}%
Then, from \eqref{eq:121}, 
we obtain%
\begin{align*}
\left( \id _{H_{1}\otimes H_{2}}\otimes p_{1}\right) &\left( \Ff\right)  =\left( \Gg ^{-1}\otimes 1_{H_{1}}\right) \left(
1_{H_{1}}\otimes \Hh \right) \left( p_{1}\otimes \tau \right) \left(
\Ff \right)  
\\
&\overset{\eqref{eq:p1IdF}}=\left( \Gg ^{-1}\otimes 1_{H_{1}}\right) \left( 1_{H_{1}}\otimes
\Hh \tau \left( \Gg \right) ^{-1}\right) \big(\left( \id %
_{H_{1}}\otimes \tau \right) \left( \Ff _{1}\otimes 1_{H_{2}}\right)\big)
\big(\left( \id _{H_{1}}\otimes \tau \right) \left( \Delta _{H_{1}}\otimes
\id _{H_{2}}\right) \left( \Gg \right)\big).
\end{align*}%
Therefore, we have
\begin{align}
\label{eq:Idp1F}  &\left( \id _{H_{1}\otimes H_{2}}\otimes p_{1}\right) \left( \Ff\right)  =\left( \Gg ^{-1}\otimes 1_{H_{1}}\right) \left( 1_{H_{1}}\otimes
\Hh  \left( \Gg ^\op\right) ^{-1}\right) \big(\left( \id %
_{H_{1}}\otimes \tau \right) \left( \Ff _{1}\otimes 1_{H_{2}}\right)\big)
\big(\left( \id _{H_{1}}\otimes \tau \right) \left( \Delta _{H_{1}}\otimes
\id _{H_{2}}\right) \left( \Gg \right)\big).
\end{align}
Finally, from \eqref{eq:211}, 
we get%
\begin{align*}
\left( \Hh ^{-1}\otimes 1_{H_{1}}\right) \left( 1_{H_{2}}\otimes
\Ff _{1}\right) &\left( \id _{H_{2}}\otimes \Delta
_{H_{1}}\right) \left( \Hh \right)  
=\left( \tau \otimes p_{1}\right) \left( \Ff \right)  \\
&\overset{\mathclap{\eqref{eq:Idp1F}}}=
\left(\left( \Gg ^\op\right) ^{-1}\otimes 1_{H_{1}}\right) 
\left(
\tau \otimes \id _{H_{1}}\right) \left( 1_{H_{1}}\otimes \Hh %
 \left( \Gg ^\op\right) ^{-1}\right)  \\
&\hspace{0.5cm}\left( \tau \otimes \id _{H_{1}}\right) \left( \id %
_{H_{1}}\otimes \tau \right) \left( \Ff _{1}\otimes 1_{H_{2}}\right)\\
&\hspace{0.5cm}
\left( \tau \otimes \id _{H_{1}}\right) \left( \id %
_{H_{1}}\otimes \tau \right) \left( \Delta _{H_{1}}\otimes \id %
_{H_{2}}\right) \left( \Gg \right)  \\
&=\left( \left( \Gg ^\op\right) ^{-1}\otimes 1_{H_{1}}\right)
\Hh _{13} \left( \Gg ^\op\right) _{13}^{-1}\left(
1_{H_{2}}\otimes \Ff _{1}\right) \left( \id _{H_{2}}\otimes
\Delta _{H_{1}}\right)  \left( \Gg ^\op\right).
\end{align*}%
Hence%
\begin{eqnarray*}
\left( 1_{H_{2}}\otimes \Ff _{1}\right) \left( \id %
_{H_{2}}\otimes \Delta _{H_{1}}\right) \left[ \Hh  \left(
\Gg ^\op\right) ^{-1}\right]  &=&\left[ \Hh  \left( \Gg ^\op%
\right) ^{-1}\right] _{12}\left[ \Hh  \left( \Gg ^\op\right)
^{-1}\right] _{13}\left( 1_{H_{2}}\otimes \Ff _{1}\right).  
\end{eqnarray*}%
If we rewrite the above equality in terms of $\Gg ^{\mathrm{op}}\Hh ^{-1}$, we get the
weak hexagon equation
\begin{eqnarray*}
\left( \id _{H_{2}}\otimes \left( \Delta _{H_{1}}\right) _{\Ff %
_{1}}\right) \left( \Gg ^{\mathrm{op}}\Hh ^{-1}\right)
&=&\left( \Gg ^{\mathrm{op}}\Hh ^{-1}\right) _{13}\left(
\Gg ^{\mathrm{op}}\Hh ^{-1}\right) _{12}. \end{eqnarray*}
The other weak hexagon equation, namely,
\begin{eqnarray*}
\left( \left( \Delta _{H_{2}}\right) _{\Ff _{2}}\otimes \id %
_{H_{1}}\right) \left( \Gg ^{\mathrm{op}}\Hh ^{-1}\right)
&=&\left( \Gg ^{\mathrm{op}}\Hh ^{-1}\right) _{13}\left(
\Gg ^{\mathrm{op}}\Hh ^{-1}\right) _{23},
\end{eqnarray*}
can be obtained in a similar way by employing  \eqref{eq:122}, \eqref{eq:212} and \eqref{eq:221}. Hence $\Rr:=\Gg^{\mathrm{op}}\Hh^{-1}$ is a weak $\Rr$-matrix of $((H_{2})_{\Ff_{2}},(H_{1})_{\Ff_{1}})$.
Note that the counterpart of \eqref{eq:Idp1F} is 
\begin{align*}
 \label{eq:p2IdF}   \left( p_{2}\otimes \id _{H_{1}\otimes H_{2}}\right) \left( \Ff\right)&=\left( 1_{H_{2}}\otimes \Gg ^{-1}\right) \left( \Hh 
 \left( \Gg ^\op\right) ^{-1}\otimes 1_{H_{2}}\right) \left( \tau \otimes
 \id _{H_{2}}\right) \left( 1_{H_{1}}\otimes \Ff _{2}\right)
 \left( \tau \otimes \id _{H_{2}}\right) \left( \id _{H_{1}}\otimes \Delta _{H_{2}}\right) \left( \Gg \right) 
 \end{align*} 
from which we get
\begin{eqnarray*}
  1_{H_{1}}\otimes \left( p_{2}\otimes \id _{H_{1}\otimes
H_{2}}\right) \left( \Ff \right) 
&=&
\left( 1_{H_{1}\otimes H_{2}}\otimes \Gg ^{-1}\right) \left(
1_{H_{1}}\otimes \Hh  \left( \Gg ^\op\right) ^{-1}\otimes
1_{H_{2}}\right)  \\
&&\left( \id _{H_{1}}\otimes \tau \otimes \id _{H_{2}}\right)
\left( 1_{H_{1}}\otimes 1_{H_{1}}\otimes \Ff _{2}\right)  \\
&&\left( \id _{H_{1}}\otimes \tau \otimes \id _{H_{2}}\right)
\left( 1_{H_{1}}\otimes \left( \id _{H_{1}}\otimes \Delta
_{H_{2}}\right) \left( \Gg \right) \right) .
\end{eqnarray*}%
We also have
\begin{align*}
\left( \id _{H_{1}}\otimes \tau \otimes \id _{H_{2}}\right)
&\left( \id _{H_{1}\otimes H_{1}}\otimes \Delta _{H_{2}}\right) \left(
p_{1}\otimes \id _{H_{1}\otimes H_{2}}\right) \left( \Ff %
\right)\\
&\overset{\mathclap{\eqref{eq:p1IdF}} }=
\left( \id _{H_{1}}\otimes \tau \otimes \id _{H_{2}}\right)
\left( \id _{H_{1}\otimes H_{1}}\otimes \Delta _{H_{2}}\right) \left(
1_{H_{1}}\otimes \Gg ^{-1}\right)  \\
&\hspace{.5cm}\left( \id _{H_{1}}\otimes \tau \otimes \id _{H_{2}}\right)
\left( \id _{H_{1}\otimes H_{1}}\otimes \Delta _{H_{2}}\right) \left(
\Ff _{1}\otimes 1_{H_{2}}\right)  \\
&\hspace{.5cm}\left( \id _{H_{1}}\otimes \tau \otimes \id _{H_{2}}\right)
\left( \id _{H_{1}\otimes H_{1}}\otimes \Delta _{H_{2}}\right) \left(
\Delta _{H_{1}}\otimes \id _{H_{2}}\right) \left( \Gg \right) \\
&=
\left( \id _{H_{1}}\otimes \tau \otimes \id _{H_{2}}\right)
\left(
1_{H_{1}}\otimes (\id _{H_{1}}\otimes \Delta _{H_{2}})(\Gg ^{-1})\right)  \\
&\hspace{.5cm}\big(\left( \id _{H_{1}}\otimes \tau \otimes \id _{H_{2}}\right)
\left( \Ff _{1}\otimes 1_{H_{2}}\otimes 1_{H_{2}}\right) \big)
\Delta _{H_{1}\otimes H_{2}}\left( \Gg \right). 
\end{align*}%
Now, if we evaluate $p_{1}\otimes p_{2}\otimes \id _{H_{1}\otimes
H_{2}}$ on \eqref{eq:000},
 we obtain%
\begin{equation*}
\left( \Gg \otimes 1_{H_{1}\otimes H_{2}}\right) \Ff =\left(
1_{H_{1}}\otimes \left( p_{2}\otimes \id _{H_{1}\otimes H_{2}}\right)
\left( \Ff \right) \right) \left( \id _{H_{1}}\otimes \tau
\otimes \id _{H_{2}}\right) \left( \id _{H_{1}\otimes
H_{1}}\otimes \Delta _{H_{2}}\right) \left( p_{1}\otimes \id %
_{H_{1}\otimes H_{2}}\right) \left( \Ff \right)
\end{equation*}%
from which, by employing the two equalities obtained above, we get
\begin{eqnarray*}
\Ff  &=&
\left( \Gg ^{-1}\otimes 1_{H_{1}\otimes H_{2}}\right) 
\left( 1_{H_{1}}\otimes \left( p_{2}\otimes \id _{H_{1}\otimes
H_{2}}\right) \left( \Ff \right) \right)  \\
&&\left( \id _{H_{1}}\otimes \tau \otimes \id _{H_{2}}\right)
\left( \id _{H_{1}\otimes H_{1}}\otimes \Delta _{H_{2}}\right) \left(
p_{1}\otimes \id _{H_{1}\otimes H_{2}}\right) \left( \Ff %
\right)\\
&=&
\left( \Gg ^{-1}\otimes 1_{H_{1}\otimes H_{2}}\right)  
\left( 1_{H_{1}\otimes H_{2}}\otimes \Gg ^{-1}\right) \left(
1_{H_{1}}\otimes \Hh  \left( \Gg ^\op\right) ^{-1}\otimes
1_{H_{2}}\right)  \\
&&\left( \id _{H_{1}}\otimes \tau \otimes \id _{H_{2}}\right)
\left( 1_{H_{1}}\otimes 1_{H_{1}}\otimes \Ff _{2}\right)  \\
&&\left( \id _{H_{1}}\otimes \tau \otimes \id _{H_{2}}\right)
\left( 1_{H_{1}}\otimes \left( \id _{H_{1}}\otimes \Delta
_{H_{2}}\right) \left( \Gg \right) \right)  \\
&&\left( \id _{H_{1}}\otimes \tau \otimes \id _{H_{2}}\right)
\left(
1_{H_{1}}\otimes (\id _{H_{1}}\otimes \Delta _{H_{2}})(\Gg ^{-1})\right) \\
&&\big(\left( \id _{H_{1}}\otimes \tau \otimes \id _{H_{2}}\right)
\left( \Ff _{1}\otimes 1_{H_{2}}\otimes 1_{H_{2}}\right)\big)  
\Delta _{H_{1}\otimes H_{2}}\left( \Gg \right)\\
&=&
\left( \Gg ^{-1}\otimes \Gg ^{-1}\right)  
\left( 1_{H_{1}}\otimes \Hh \left( \Gg ^\op\right)
^{-1}\otimes 1_{H_{2}}\right)  \\
&&\left( \id _{H_{1}}\otimes \tau \otimes \id _{H_{2}}\right)
\left( 1_{H_{1}}\otimes 1_{H_{1}}\otimes \Ff _{2}\right)  \\
&&\big(\left( \id _{H_{1}}\otimes \tau \otimes \id _{H_{2}}\right)
\left( \Ff _{1}\otimes 1_{H_{2}}\otimes 1_{H_{2}}\right)\big)  
\Delta _{H_{1}\otimes H_{2}}\left( \Gg \right)\\
&=&\left( \Gg ^{-1}\otimes \Gg ^{-1}\right) \left(
1_{H_{1}}\otimes \Hh  \left( \Gg ^\op\right) ^{-1}\otimes
1_{H_{2}}\right) \big(\left( \id _{H_{1}}\otimes \tau \otimes \id %
_{H_{2}}\right) \left( \Ff _{1}\otimes \Ff _{2}\right)\big) \Delta
_{H_{1}\otimes H_{2}}\left( \Gg \right).
\end{eqnarray*}%
Thus%
\begin{equation*}
\Ff =\left( \Gg ^{-1}\otimes \Gg ^{-1}\right) \left(
1_{H_{1}}\otimes \Hh  \left( \Gg ^\op\right) ^{-1}\otimes
1_{H_{2}}\right) \big(\left( \id _{H_{1}}\otimes \tau \otimes \id %
_{H_{2}}\right) \left( \Ff _{1}\otimes \Ff _{2}\right)\big) \Delta
_{H_{1}\otimes H_{2}}\left( \Gg \right),
\end{equation*}%
hence $\Ff $ is cohomologous to
\begin{equation*}
\Ff^\Gg=\left( \Gg \otimes \Gg \right) \Ff \Delta
_{H_{1}\otimes H_{2}}\left( \Gg ^{-1}\right) =\left( 1_{H_{1}}\otimes
\Rr^{-1}\otimes 1_{H_{2}}\right)
\big(\left( \id _{H_{1}}\otimes \tau \otimes \id _{H_{2}}\right)
\left( \Ff _{1}\otimes \Ff _{2}\right) \big).\qedhere
\end{equation*}
\end{proof}

We are also able to demonstrate a sort of converse of the previous statement.

\begin{lemma}\label{lem:Drinfeldtwist}
    Let $H_{1}$ and $H_{2}$ be bialgebras, $\Ff _{1}\in H_{1}\otimes H_{1}$ and $\Ff _{2}\in H_{2}\otimes H_{2}$ be twists on $H_{1}$ and $H_{2}$, respectively, and $\Rr \in H_{2}\otimes H_{1}$ be a weak $\Rr $-matrix of $((H_2)_{\Ff_2},(H_1)_{\Ff_1})$. Then, 
\[
\Ff :=\left( 1_{H_{1}}\otimes\Rr ^{-1}\otimes
1_{H_{2}}\right) \big(\left( \id _{H_{1}}\otimes \tau \otimes \id %
_{H_{2}}\right) \left( \Ff _{1}\otimes \Ff _{2}\right)\big)
\]
is a twist on $H_{1}\otimes H_{2}$. 
\end{lemma}

\begin{proof}
    We compute
\begin{align*}
   (1_{H_{1}\otimes H_{2}}&\otimes\Ff ) (\id _{H_{1}\otimes H_{2}}\otimes\Delta_{H_{1}\otimes H_{2}})(\Ff )\\
   &=(1_{H_{1}\otimes H_{2}}\otimes1_{H_{1}}\otimes\overline{\Rr}^{l}\ot\overline{\Rr}_{l}\otimes1_{H_{2}})(1_{H_{1}\otimes H_{2}}\otimes\Ff ^i_{1}\otimes\Ff _2^j\otimes\Ff _{1_i}\otimes\Ff _{2_j})\\
   &\hspace{.5cm}(1_{H_{1}}\otimes\overline{\Rr }^{j}\otimes\overline{\Rr }_{j_1}\otimes1_{H_{2}}\otimes\overline{\Rr }_{j_2}\otimes1_{H_{2}})(\Ff _1^s\otimes\Ff _2^t\otimes(\Ff _{1_{s}})_{1}\otimes(\Ff _{2_{t}})_{1}\otimes(\Ff _{1_{s}})_{2}\otimes(\Ff _{2_{t}})_{2})\\
   &\overset{\mathclap{\eqref{wtr2}}}{=}\Ff _1^s\otimes\overline{\Rr }^i\overline{\Rr }^j\Ff _2^t\otimes\Ff _1^i\overline{\Ff _1}^i\overline{\Rr }_i\Ff _1^{j}(\Ff _{1_{s}})_{1}\otimes\overline{\Rr }^l\Ff _2^j(\Ff _{2_{t}})_{1}\otimes\overline{\Rr }_l\Ff _{1_{i}}\overline{\Ff _1}_i\overline{\Rr }_j\Ff _{1_{j}}(\Ff _{1_s})_{2}\otimes\Ff _{2_{j}}(\Ff _{2_{t}})_{2}\\
   &=\Ff _1^s\otimes\overline{\Rr }^i\overline{\Rr }^j\Ff _2^t\otimes\overline{\Rr }_i\Ff _1^{j}(\Ff _{1_{s}})_{1}\otimes\overline{\Rr }^l\Ff _2^j(\Ff _{2_{t}})_{1}\otimes\overline{\Rr }_l\overline{\Rr }_j\Ff _{1_{j}}(\Ff _{1_s})_{2}\otimes\Ff _{2_{j}}(\Ff _{2_{t}})_{2}\\
   &=\Ff _1^s\otimes\overline{\Rr }^l\overline{\Rr }^j\Ff _2^t\otimes\overline{\Rr }_{l}\Ff _{1}^{i}(\Ff _{1_{s}})_{1}\otimes\overline{\Rr }^i\Ff _{2}^{j}(\Ff _{2_{t}})_{1}\otimes\overline{\Rr }_i\overline{\Rr }_j\Ff _{1_{i}}(\Ff _{1_{s}})_{2}\otimes\Ff _{2_{j}}(\Ff _{2_{t}})_{2}\\
   &\overset{\mathclap{\eqref{2-cocy}}}{=}\Ff _1^s\otimes\overline{\Rr }^l\overline{\Rr }^j\Ff _2^{l}(\Ff _2^t)_{1}\otimes\overline{\Rr }_{l}\Ff _{1}^{i}(\Ff _{1_{s}})_{1}\otimes\overline{\Rr }^i\Ff _{2_{l}}(\Ff _2^t)_{2}\otimes\overline{\Rr }_i\overline{\Rr }_j\Ff _{1_{i}}(\Ff _{1_{s}})_{2}\otimes\Ff _{2_{t}}\\
   &\overset{\mathclap{\eqref{2-cocy}}}{=}\Ff _1^i(\Ff _1^s)_{1}\otimes\overline{\Rr }^l\overline{\Rr }^j\Ff _2^{l}(\Ff _2^t)_{1}\otimes\overline{\Rr }_{l}\Ff _{1_{i}}(\Ff _1^s)_{2}\otimes\overline{\Rr }^i\Ff _{2_{l}}(\Ff _2^t)_{2}\otimes\overline{\Rr }_i\overline{\Rr }_j\Ff _{1_{s}}\otimes\Ff _{2_{t}}\\
   &=\Ff _1^i(\Ff _1^s)_{1}\otimes\overline{\Rr }^l\Ff _2^j\overline{\Ff _2}^j\overline{\Rr }^j\Ff _2^{l}(\Ff _2^t)_{1}\otimes\overline{\Rr }_{l}\Ff _{1_{i}}(\Ff _1^s)_{2}\otimes\Ff _{2_{j}}\overline{\Ff _2}_j\overline{\Rr }^i\Ff _{2_{l}}(\Ff _2^t)_{2}\otimes\overline{\Rr }_i\overline{\Rr }_j\Ff _{1_{s}}\otimes\Ff _{2_{t}}\\
   &\overset{\mathclap{\eqref{wtr3}}}{=}(1_{H_{1}}\otimes\Rr ^{-1}\otimes 1_{H_{2}}\otimes1_{H_{1}\otimes H_{2}})(\Ff _1^i\otimes\Ff _2^j\otimes\Ff _{1_{i}}\otimes\Ff _{2_{j}}\otimes1_{H_{1}\otimes H_{2}})\\
   &\hspace{.5cm}(1_{H_{1}}\otimes\overline{\Rr }_1^j\otimes1_{H_{1}}\otimes\overline{\Rr }^j_2\otimes\overline{\Rr }_j\otimes1_{H_{2}})((\Ff _1^s)_{1}\otimes(\Ff _2^t)_{1}\otimes(\Ff _1^s)_{2}\otimes(\Ff _2^t)_{2}\otimes\Ff _{1_{s}}\otimes\Ff _{2_{t}})\\
   &=(\Ff \otimes1_{H_{1}\otimes H_{2}})(\Delta_{H_{1}\otimes H_{2}}\otimes\id _{H_{1}\otimes H_{2}})(\Ff ).
\end{align*}
Moreover, we have    
\[
\begin{split}
(\id _{H_{1}\otimes H_{2}}&\otimes\varepsilon_{H_{1}\otimes H_{2}})(\Ff )=\\
&=(1_{H_{1}}\otimes(\id \otimes\varepsilon_{H_{1}})(\Rr ^{-1}))(\id _{H_{1}}\otimes\id _{H_{2}}\otimes\varepsilon_{H_{1}}\otimes\varepsilon_{H_{2}})(\id _{H_{1}}\otimes\tau\otimes\id _{H_{2}})(\Ff _{1}\otimes\Ff _{2})\\&=(\id _{H_{1}}\otimes\varepsilon_{H_{1}}\otimes\id _{H_{2}}\otimes\varepsilon_{H_{2}})(\Ff _{1}\otimes\Ff _{2})\overset{\eqref{norm}}{=}1_{H_{1}\otimes H_{2}}.
\end{split}
\]
Similarly, $(\varepsilon_{H_{1}\otimes H_{2}}\otimes\id _{H_{1}\otimes H_{2}})(\Ff )=1_{H_{1}\otimes H_{2}}$ and then $\Ff $ is a twist of $H_{1}\otimes H_{2}$.
\end{proof}

\begin{remark}
By Lemma \ref{lem:Drinfeldtwist}, if $H_{1}=H_{2}:=H$ and $\Ff_{2}=1_{H}\ot1_{H}=\Rr$, one has that $\Ff_{1}^{i}\ot1_{H}\ot\Ff_{1_{i}}\ot1_{H}$ is a twist on $H\ot H$. Similarly, if $\Ff_{1}=1_{H}\ot1_{H}=\Rr$, one obtains that $1_{H}\ot\Ff_{2}^{i}\ot 1_{H}\ot\Ff_{2_{i}}$ is a twist on $H\ot H$. We observe that in \cite[Lemma 1]{Lomp} it is proven that, if $\Ff^{i}\ot1_{H}\ot1_{H}\ot\Ff_{i}$ is a twist on $H\ot H$ for a given twist $\Ff$ on $H$, then $(e_{i}^{m}\ot e^{m}_{j})(\Ff)$ is a twist for $H^{\ot m}$ for any $1\leq i<j\leq m$, where $e^{m}_{i}:H\to H^{\ot m}$ denotes the embedding of $H$ into the $i$th tensorand of $H^{\ot m}$, so e.g.\ $e_2^4(x)=1_H\ot x\ot 1_H\ot 1_H$. These special twists will be mentioned again in Remark \ref{rmk:PanseraLomp}. We point out that in \cite{Lomp} the notion of twist is dual with respect to Definition \ref{def:twist}.
\end{remark}


As an immediate consequence of the previous results, we get the following characterization of weak $\Rr$-matrices.

\begin{corollary}
\label{cor:weakastwist}
Let $H_1$ and $H_2$ be bialgebras. Then $\Rr\in H_2\otimes H_1$ is a weak $\Rr$-matrix of $(H_2,H_1)$ if and only if $1_{H_1}\otimes\Rr^{-1}\otimes 1_{H_2}$ is a twist on the tensor product bialgebra $H_1\otimes H_2$.
\end{corollary}

\begin{proof}
If $\Rr$ is a weak $\Rr$-matrix, we can take $\Ff_1=1_{H_1}\otimes 1_{H_1}$ and  $\Ff_2=1_{H_2}\otimes 1_{H_2}$ and apply Lemma \ref{lem:Drinfeldtwist}. Conversely, if $\Ff\coloneqq 1_{H_1}\otimes\Rr^{-1}\otimes 1_{H_2}$ is a twist, we can use Theorem \ref{thm:twistot} to get that $\Ff_1=1_{H_1}\otimes1_{H_1}$, $\Ff_2=1_{H_2}\otimes 1_{H_2}$ $\Gg= 1_{H_1}\otimes1_{H_2}$, $\Hh=\Rr^{-1}$ and that $\Gg^\op\Hh^{-1}=\Rr$ is a weak $\Rr$-matrix.
\end{proof}

\subsection{Quasitriangular structures up to cohomology}
Since any quasitriangular structure is a special case of a twist, we can apply Theorem \ref{thm:twistot} to classify quasitriangular structures on the tensor product bialgebra up to cohomology. 
First, we show the following result.
\begin{lemma}
    Let $(H,\Rr)$ be a quasitriangular bialgebra. Given $h\in H^{\times}$, we set 
    \[\Rr'\coloneq (1\ot 1)^{h}=(h\ot h)\Delta(h^{-1})\in H\ot H.\] 
    Then, the following statements hold:
    \begin{itemize}
        \item[$1)$] Both $h$ and $\Rr'$ are central if and only if so is $\Delta(h).$
      \item[$2)$] If $h$ is central and $\Rr'$ is a central weak $\Rr$-matrix of $(H,H)$, then $\Rr^{h}=(h\ot h)\Rr\Delta(h^{-1})$ is a quasitriangular structure for $H$.
        \end{itemize}
\end{lemma}
\begin{proof}
1) If $h$ is central, then $\Rr'$ is central if and only if $\Delta(h)$ is central. Moreover, $\Delta(h)$ central implies $h$ central as $ha= h_1 a_1 \varepsilon(h_2 a_2) = a_1 h_1 \varepsilon(a_2 h_2) = ah$, for all $a\in H$.

2) Clearly $\Rr^h$ is invertible. Since $h$ is central, we obtain
\[
\Rr^{h}=(h\ot h)\Rr\Delta(h^{-1})=\Rr(h\ot h)\Delta(h^{-1})=\Rr\Rr'
\]
and also
\[
\Rr'\Delta(x)(\Rr')^{-1}
=(h\ot h)\Delta(h^{-1})\Delta(x)\Delta(h)(h^{-1}\ot h^{-1})
=\Delta(h^{-1}xh)=\Delta(x).
\]
Thus, we get
\[
\Rr^{h}\Delta(\cdot)(\Rr^{h})^{-1}=\Rr\Rr'\Delta(\cdot)(\Rr')^{-1}\Rr^{-1}=\Rr\Delta(\cdot)\Rr^{-1}=\Delta^{\mathrm{op}}(\cdot),
\]
hence \eqref{qtr1} is satisfied. 
Moreover, since $\Rr'$ is central  and $\Rr,\Rr'$ satisfy \eqref{qtr2} and\eqref{qtr3}, also $\Rr^{h}$ satisfies \eqref{qtr2} and \eqref{qtr3}.
\end{proof}

\begin{remark}
    Since $\Rr'\Delta(\cdot)(\Rr')^{-1}=\Delta(\cdot)$, we observe that $\Rr'$ is a quasitriangular structure for $H$ if and only if $H$ is cocommutative.
\end{remark}

We are now able to state the aforementioned classification result.

\begin{proposition}
\label{pro:Rot}
Let $H_{1}$ and $H_{2}$ be bialgebras. Then, any quasitriangular structure $\Ss$ on $H_{1}\ot H_{2}$ is cohomologous to a quasitriangular structure
\begin{equation}\label{eq:quasitriangtensprod}
\Rr_{\ot}:=(1_{H_{1}}\ot\Qq\ot1_{H_{2}})\big((\id _{H_{1}}\ot\tau\ot\id _{H_{2}})(\Rr _1\ot\Rr _2)\big),
\end{equation}
where $\Rr _1\in H_{1}\ot H_{1}$ and $\Rr _2\in H_{2}\ot H_{2}$ are quasitriangular structures of $H_{1}$ and $H_{2}$, respectively, and $\Qq\in H_{2}\ot H_{1}$ is a central weak $\Rr$-matrix of $(H_{2},H_{1})$. Explicitly, the triple $(\Rr_1,\Rr_2,\Qq)$ associated with $\Ss$ is defined by taking $(\Rr_1,\Rr_2,\Gg,\Hh):=\Phi(\Ss)$ and then $\Qq=\Hh(\Gg^\op)^{-1}$.
\end{proposition}

\begin{proof}
    Let $\Ss$ be a quasitriangular structure on $H_{1}\ot H_{2}$, hence a twist on $H_{1}\ot H_{2}$. By definition, the datum $(\Rr_1,\Rr_2,\Gg,\Hh):=\Phi(\Ss)$ is given by setting 
\begin{gather*}
    \Rr _{1}:=\left( \id \otimes \varepsilon \otimes \id %
\otimes \varepsilon \right) \left( \Ss\right)\in H_{1}\ot H_{1}, \qquad \Rr _{2}:=\left( \varepsilon \otimes \id \otimes
\varepsilon \otimes \id \right) \left( \Ss\right)\in H_{2}\ot H_{2} 
\\
\Gg :=\left( \id \otimes \varepsilon \otimes \varepsilon
\otimes \id \right) \left( \Ss\right)\in H_{1}\ot H_{2}, \qquad 
\Hh :=\left( \varepsilon \otimes \id \otimes \id %
\otimes \varepsilon \right) \left( \Ss\right)\in H_{2}\ot H_{1}. 
\end{gather*}
Then, by Theorem \ref{thm:twistot}, we know that $\Ss$ is cohomologous to 
\[
\Ss^\Gg=(1_{H_{1}}\ot\Rr^{-1}\ot1_{H_{2}})\big((\id _{H_{1}}\ot\tau\ot\id _{H_{2}})(\Rr _1\ot\Rr _2)\big),
\]
where $\Rr _{1}$ and $\Rr _{2}$ are twists on $H_{1}$ and $H_{2}$, respectively. Moreover, $\Rr=\Gg^{\mathrm{op}}\Hh^{-1}\in H_{2}\ot H_{1}$ is a weak $\Rr$-matrix of $((H_{2})_{\Rr_{2}},(H_{1})_{\Rr_{1}})$.
Since $\id \ot\varepsilon:H_{1}\ot H_{2}\to H_{1}$ and $\varepsilon\ot\id :H_{1}\ot H_{2}\to H_{2}$ are bialgebra maps, using Example \ref{exa:weakR} we obtain that $\Rr _1$ and $\Rr _2$ are weak $\Rr$-matrices of $(H_{1},H_{1})$ and $(H_{2},H_{2})$, respectively. Moreover, $\Ss$ satisfies \eqref{qtr1}, i.e.\ for all $x\in H_{1}$ and $a\in H_{2}$ we have
\begin{equation}\label{eqq}
    (x_{2}\ot a_{2}\ot x_{1}\ot a_{1})\Ss=\Ss(x_{1}\ot a_{1}\ot x_{2}\ot a_{2}).
\end{equation}
In particular, we have $(x_{2}\ot1_{H_{2}}\ot x_{1}\ot1_{H_{2}})\Ss=\Ss(x_{1}\ot1_{H_{2}}\ot x_{2}\ot1_{H_{2}})$ for all $x\in H_{1}$ and then, applying $\id \ot\varepsilon\ot\id \ot\varepsilon$, we get $\Delta^{\mathrm{op}}(x)\Rr_{1}=\Rr_{1}\Delta(x)$ for all $x\in H_{1}$. Thus, $\Rr_{1}$ is a quasitriangular structure on $H_{1}$. Similarly, we get that $\Rr_{2}$ is a quasitriangular structure on $H_{2}$. Therefore, $(H_{1})_{\Rr_{1}}=H_{1}^{\mathrm{cop}}$ and $(H_{2})_{\Rr_{2}}=H_{2}^{\mathrm{cop}}$, see Example \ref{exa:Rtwist}. Moreover, applying $\id \ot\varepsilon\ot\varepsilon\ot\id $ to \eqref{eqq}, we get $(x\ot a)\Gg=\Gg(x\ot a)$ for all $x\in H_{1}$ and $a\in H_{2}$, hence $\Gg$ is central while, applying $\varepsilon\ot\id \ot\id \ot\varepsilon$ to \eqref{eqq}, we get $(a\ot x)\Hh=\Hh(a\ot x)$ for all $x\in H_{1}$ and $a\in H_{2}$, hence also $\Hh$ is central. Thus, $\Rr$ is a central weak $\Rr$-matrix of $(H_{2}^{\mathrm{cop}},H_{1}^{\mathrm{cop}})$ that is equivalent, by Lemma \ref{lem:weakop}, to $\Qq\coloneq \Rr^{-1}\in H_{2}\ot H_{1}$ being a central weak $\Rr$-matrix of $(H_{2},H_{1})$.

We also recall that $\Ss$ satisfies \eqref{qtr2} and \eqref{qtr3}, i.e. 
\[
    (\id _{H_{1}\ot H_{2}}\ot\Delta_{H_{1}\ot H_{2}})(\Ss)=\Ss_{13}\Ss_{12},\qquad (\Delta_{H_{1}\ot H_{2}}\ot\id _{H_{1}\ot H_{2}})(\Ss)=\Ss_{13}\Ss_{23}.
\]
By applying $\id \ot\varepsilon\ot\varepsilon\ot\id \ot\varepsilon\ot\id $ to the first equality we get $(\id \ot\Delta)(\Gg)=\Gg_{13}\Gg_{12}$, while applying $\id \ot\varepsilon\ot\id \ot\varepsilon\ot\varepsilon\ot\id $ to the second we obtain $(\Delta\ot\id )(\Gg)=\Gg_{13}\Gg_{23}$. Therefore, $\Gg$ is a central weak $\Rr$-matrix of $(H_{1},H_{2})$. 
Moreover, we have
\begin{align*}
\Delta_{H_{1}\ot H_{2}}(\Gg)&=(\id \ot\tau\ot\id )(\Delta_{H_{1}}\ot\Delta_{H_{2}})(\Gg)\\
&\overset{\mathclap{\eqref{eq:DDR}}}=(\id \ot\tau\ot\id )(\Gg^{i}\Gg%
^{p}\otimes \Gg^{j}\Gg^{s}\otimes \Gg_{p}\Gg%
_{s}\otimes \Gg_{i}\Gg_{j})\\
&=\Gg^{i}\Gg%
^{p}\otimes  \Gg_{p}\Gg%
_{s}\otimes \Gg^{j}\Gg^{s}\otimes\Gg_{i}\Gg_{j}\\
&=(\Gg^{i}\ot\Gg_{s}\ot\Gg^{s}\ot\Gg_{i})(\Gg\ot\Gg).
\end{align*}
Set $\Ss'\coloneqq (\Gg\ot\Gg)
\Delta_{H_{1}\ot H_{2}}(\Gg^{-1})$. Therefore, $(\Ss')^{-1}= 
\Delta_{H_{1}\ot H_{2}}(\Gg)(\Gg^{-1}\ot\Gg^{-1})=\Gg^{i}\ot\Gg_{s}\ot\Gg^{s}\ot\Gg_{i}$.
Thus, $(\Ss')^{-1}$ is central invertible. Moreover, $(\Ss')^{-1}$ also satisfies \eqref{qtr2} and \eqref{qtr3} since $\Gg$ is a central weak $\Rr$-matrix. Hence $(\Ss')^{-1}$ is a central weak $\Rr$-matrix of $(H_{1}\ot H_{2},H_{1}\ot H_{2})$ and so, by Lemma \ref{lem:weakop}, also $\Ss'$ is a central weak $\Rr$-matrix of $(H_{1}\ot H_{2},H_{1}\ot H_{2})$. Therefore, by the previous lemma we know that $\Ss^{\Gg}$ is a quasitriangular structure on $H_{1}\ot H_{2}$.

In order to conclude, we show that, given a triple $(\Rr_1,\Rr_2,\Qq)$ where $\Rr_{1}\in H_{1}\ot H_{1}$ and $\Rr_{2}\in H_{2}\ot H_{2}$ are quasitriangular structures of $H_{1}$ and $H_{2}$, respectively, and $\Qq\in H_{2}\ot H_{1}$ is a central weak $\Rr$-matrix of $(H_{2},H_{1})$, then $\Rr_{\ot}$ defined as in \eqref{eq:quasitriangtensprod} is a quasitriangular structure on $H_{1}\ot H_{2}$. Clearly, $\Rr_{\ot}$ is invertible. We compute
\begin{align*}
    \Rr_{\ot}\Delta_{H_{1}\ot H_{2}}(x\ot a)&\overset{\mathclap{\eqref{eq:quasitriangtensprod}}}=(1\ot \Qq \ot1)(\Rr_{1}^ix_{1}\ot\Rr_2^ja_{1}\ot\Rr_{1_{i}}x_{2}\ot\Rr_{2_{j}}a_{2})\\&\overset{\mathclap{\eqref{qtr1}}}{=}(1\ot \Qq \ot1)(x_{2}\Rr_1^i\ot a_{2}\Rr_2^j\ot x_{1}\Rr_{1_{i}}\ot a_{1}\Rr_{2_{j}})\\&=(x_{2}\ot a_{2}\ot x_{1}\ot a_{1})(1\ot \Qq \ot1)(\Rr_1^i\ot \Rr_2^j\ot\Rr_{1_{i}}\ot\Rr_{2_{j}})\\&\overset{\mathclap{\eqref{eq:quasitriangtensprod}}}=\Delta^{\mathrm{op}}_{H_{1}\ot H_{2}}(x\ot a)\Rr_{\ot}
\end{align*}
and also
\begin{align*}
    (\id _{H_1\ot H_2}&\ot\Delta_{H_{1}\ot H_{2}})(\Rr_{\ot})\\
    &\overset{\mathclap{\eqref{eq:quasitriangtensprod}}}=(1\ot \Qq^i \ot( \Qq_i )_1\ot1\ot( \Qq_i )_2\ot1)(\Rr_1^i\ot\Rr_2^j\ot(\Rr_{1_i})_{1}\ot(\Rr_{2_{j}})_{1}\ot(\Rr_{1_{i}})_2\ot(\Rr_{2_{j}})_{2})\\&\overset{\mathclap{\eqref{qtr2},\eqref{wtr2}}}{=}\hspace{0.3cm}(1\ot \Qq^i\Qq^s \ot \Qq_s \ot1\ot \Qq_i \ot1)(\Rr_1^i\Rr_1^p\ot\Rr_2^j\Rr_2^s\ot\Rr_{1_{p}}\ot\Rr_{2_{s}}\ot\Rr_{1_{i}}\ot\Rr_{2_{j}})\\&=(1\ot \Qq^i \ot1\ot1\ot \Qq_i \ot1)(\Rr_1^i\ot\Rr_2^j\ot1\ot1\ot\Rr_{1_{i}}\ot\Rr_{2_{j}})\\&\hspace{.5cm} (1\ot \Qq^s \ot \Qq_s\ot1\ot1\ot1)(\Rr_1^p\ot\Rr_2^s\ot\Rr_{1_{p}}\ot\Rr_{2_{s}}\ot1\ot1)\\
    &\overset{\mathclap{\eqref{eq:quasitriangtensprod}}}=(\Rr_{\ot})_{13}(\Rr_{\ot})_{12},
\end{align*}
using that $\Qq$ is central. Similarly, we have $(\Delta_{H_{1}\ot H_{2}}\ot\id _{H_{1}\ot H_{2}})(\Rr_{\ot})=(\Rr_{\ot})_{13}(\Rr_{\ot})_{23}$ and then $\Rr_{\ot}$ is a quasitriangular structure on $H_{1}\ot H_{2}$. 
\end{proof}

\begin{remark}
    1) A direct computation, using the fact that the evaluation of the counit on any entry of $\Rr_1,\Rr_2$ and $\Qq$ gives $1$, shows that $\Phi(\Rr_\otimes)=(\Rr_1,\Rr_2,1_{H_1}\otimes 1_{H_2},\Qq)$, so the triple associated with $\Rr_\otimes$ is again $(\Rr_1,\Rr_2,\Qq)$.
    
    2) The assignment $(\Rr_1,\Rr_2,\Qq)\mapsto [\Rr_\otimes]$  defines a surjection from the set of triples $(\Rr_1, \Rr_2, \Qq)$, where $\Rr_{1}$ and $\Rr_{2}$ are $\Rr$-matrices of $H_{1}$ and $H_{2}$, respectively, and $\Qq$ is a central weak $\Rr$-matrix of $(H_{2},H_{1})$, onto the set of cohomology classes of $\Rr$-matrices of $H_1 \otimes H_2$. However, this map seems generally not injective. Otherwise it would be bijective and so its inverse would be well-defined, but this seems not the case.
    Indeed, given $h\in (H_1\otimes H_2)^\times$, set $h^1:=(\id \otimes \varepsilon) (h)$ and $h^2:=(\varepsilon \otimes
\id )(h)$. Then $[\Rr_\otimes^h]=[\Rr_\otimes]$ and a direct computation shows that 
\[\Phi(\Rr_\otimes^h)=(\Rr_1^{h^1},\Rr_2^{h^2},(h^1\otimes h^2)h^{-1},(h^2\otimes h^1)\Qq(h^\op)^{-1}).\] 
Thus, by using that $\Qq$ is central, we get that the triple associated with  $\Rr_\otimes^h$ is $(\Rr_1^{h^1},\Rr_2^{h^2},\Qq)$. We see no reason for it to coincide with $(\Rr_1, \Rr_2, \Qq)$.
    \begin{invisible}
        Volendo essere più rigorosi, notiamo che $\Rr_\otimes^g$ è cohomologo al $\Rr_\otimes'$ definito tramite la nuova tripla. Quindi $[\Rr_\otimes']=[\Rr_\otimes]$ ma l'inversa di cui sopra ha immagini diverse su queste due classi.
        \end{invisible}
\end{remark}

\begin{remark}
    Given bialgebras $H$ and $A$, a classification result for the quasitriangular structures on $H\ot A$ is given in \cite[Theorem 2.9]{Chen-quasi} (consider $R=1\ot 1$ so that $H\bowtie^{R}A=H\otimes A$ in loc.\ cit.). These have the form
\[
\Ff\coloneqq \big((\id\otimes \tau\otimes\id)(\Ss \otimes \Tt)\big)(\Uu^i\otimes \Vv^\op\otimes \Uu_i)=(\Uu^i\otimes \Vv^\op\otimes \Uu_i)\big((\id\otimes \tau\otimes\id)(\Ss \otimes \Tt)\big)
\]
where $\Ss \in H\ot H$ and $\Tt\in A\ot A$ are quasitriangular structures of $H$ and $A$, respectively, while $\Uu,\Vv\in H\ot A$ are central weak $\Rr$-matrices of $(H,A)$. Using the fact that the evaluation of the counit on any entry of $\Ss,\Tt,\Uu$ and $\Vv$ gives $1$, a direct computation shows that $\Phi(\Ff)=(\Ss,\Tt,\Uu,\Vv^\op)$ so that the triple associated with $\Ff$ is $(\Ss,\Tt,(\Vv\Uu^{-1})^\op)$. In view of Proposition \ref{pro:Rot}, 
$\Ff$ is cohomologous to 
\[(1\ot(\Vv\Uu^{-1})^\op\ot 1)\big((\id \ot\tau\ot\id )(\Ss \ot\Tt)\big).\] 
Note that, by Lemma \ref{lem:weakop}, since $\Uu$ and $\Vv$ are central, we coherently get that $(\Vv\Uu^{-1})^\op$ is a central weak $\Rr$-matrix of $(A,H)$.
\end{remark}


\section{The relevant \texorpdfstring{$2$}{2}-categories}\label{sec:2-categories}

The binary product of two objects $H_{1}$ and $H_{2}$ in the category $\Bialgcc$ of cocommutative bialgebras is given by the tensor product bialgebra $H_{1}\ot H_{2}$, where the projections are $p_{1}\coloneqq \id \ot\varepsilon_{H_{2}}:H_1\otimes H_2\to H_1$ and $p_{2}\coloneqq \varepsilon_{H_{1}}\ot\id :H_1\otimes H_2\to H_2$. Moreover, the base field $\Bbbk$ is a terminal object in $\Bialgcc$, hence $\Bialgcc$ is a cartesian monoidal category (see e.g.\ \cite[Remark 3.4]{Pfeiffer}). 
Given two morphisms $f_{1}:H\to H_{1}$ and $f_{2}:H\to H_{2}$ in $\Bialgcc$, the diagonal morphism $\langle f_{1},f_{2}\rangle:H\to H_{1}\ot H_{2}$ is given by $f:=(f_{1}\ot f_{2})\Delta_{H}$.
\begin{equation*}
\xymatrix{&H\ar[ld]_{f_1}\ar[rd]^{f_2}\ar@{.>}[d]|{\langle f_{1},f_{2}\rangle}\\ H_1&
H_1\otimes H_2\ar[l]_{p_1}\ar[r]^{p_2}& H_2}
\end{equation*}
We recall that, if $H$ is not cocommutative (and so $\Delta_{H}$ is not a coalgebra map), $f$ is not a coalgebra map, in general (take e.g.\ $f_1=f_2=\id$ so that $f=\Delta_H$). 

\begin{remark}
    We recall that products in $\Bialgcc$ are computed as in $\mathsf{Coalg_{cc}}$, see \cite[Remark 3.4]{Pfeiffer}, and the latter category is cartesian monoidal, as it is the category $\mathsf{Comon_{cc}}(\Mm)$ of cocommutative comonoids in any symmetric monoidal category $\Mm$, see \cite[Corollary 2.24]{KM}.
\end{remark}

Our goal is to extend the above binary product construction to the context of quasitriangular bialgebras.
If $(H,\Rr)$ is such a bialgebra, the morphism $f:H\to (H_1\ot H_2)_\Ff$ is a bialgebra map where $\Ff=1\ot (f_2\ot f_1)(\Rr^{-1})\ot 1.$ This result has been proven e.g.\ in \cite[Lemma 4.2]{Schneider} in presence of an antipode $S$ (where one has $\Rr^{-1}=(S\ot\id)(\Rr)$), but the result remains valid for $H$ bialgebra. We are so led to include twists in our treatment. Moreover, given quasitriangular bialgebras $(H_{1},\Rr_{1})$ and $(H_{2},\Rr_{2})$, the tensor product bialgebra $H_{1}\ot H_{2}$ is quasitriangular with $\widetilde{\Rr}:=(\id \ot\tau\ot\id )(\Rr_{1}\ot\Rr_{2})$, see e.g.\ \cite[Theorem 2.2]{Chen-quasi}. If $(H_{1},\Rr_{1})$ and $(H_{2},\Rr_{2})$ are further triangular, then also $\left( H_{1}\otimes H_{2},\widetilde{\Rr}\right)$ is triangular. In the following, we will adopt the following notation
\begin{equation}
\label{def:tensorproductquasi}
\left( H_{1},\Rr _{1}\right) \otimes \left( H_{2},\Rr %
_{2}\right) :=\left( H_{1}\otimes H_{2}, \widetilde{\Rr}\right).
\end{equation}

\begin{invisible}
We observe that the categories $\mathsf{QTrBialg}$ and $\TrBialg$ are monoidal in the following way. Given quasitriangular bialgebras $\left( H_{1},\Rr _{1}\right) $ and $\left( H_{2},\Rr %
_{2}\right) $, it is known that the tensor product bialgebra $H_{1}\ot H_{2}$ is quasitriangular with $(\id \ot\tau\ot\id )(\Rr_{1}\ot\Rr_{2})$, see e.g.\ \cite[Theorem 2.2]{Chen-quasi}. In case $(H_{1},\Rr_{1})$ and $(H_{2},\Rr_{2})$ are triangular, also $\left( H_{1}\otimes H_{2},\left( \id \otimes \tau
\otimes \id \right) \left( \Rr _{1}\otimes \Rr %
_{2}\right) \right)$ is triangular.
Indeed
\begin{eqnarray*}
\left[ \left( \id _{H_{1}}\otimes \tau \otimes \id %
_{H_{2}}\right) \left( \Rr _{1}\otimes \Rr _{2}\right) \right]
^{-1} &=&\left( \id _{H_{1}}\otimes \tau \otimes \id %
_{H_{2}}\right) \left( \Rr _{1}^{-1}\otimes \Rr %
_{2}^{-1}\right) \\
&=&\left( \id _{H_{1}}\otimes \tau \otimes \id _{H_{2}}\right)
\left( \tau \otimes \tau \right) \left( \Rr _{1}\otimes \Rr %
_{2}\right) \\
&=&\tau _{H_{1}\otimes H_{2},H_{1}\otimes H_{2}}\left( \id %
_{H_{1}}\otimes \tau \otimes \id _{H_{2}}\right) \left( \Rr %
_{1}\otimes \Rr _{2}\right) \\
&=&\left[\left( \id %
_{H_{1}}\otimes \tau \otimes \id _{H_{2}}\right) \left( \Rr %
_{1}\otimes \Rr _{2}\right)\right]^{\mathrm{op}}
\end{eqnarray*}
Thus, one can define the tensor product of $(H_{1},\Rr_{1})$ and $(H_{2},\Rr_{2})$ as
\begin{equation}
\left( H_{1},\Rr _{1}\right) \otimes \left( H_{2},\Rr %
_{2}\right) :=\left( H_{1}\otimes H_{2}, \widetilde{\Rr}\right)\quad \text{where}\quad \widetilde{\Rr}\coloneqq \left( \id \otimes \tau
\otimes \id \right) \left( \Rr _{1}\otimes \Rr %
_{2}\right).
\end{equation}%
Observe also that the base field $\Bbbk$ is (quasi)triangular with unique (quasi)triangular structure $\Rr=1\ot1$. The categories $(\mathsf{QTrBialg},\otimes,\Bbbk)$ and $(\TrBialg,\otimes,\Bbbk)$ are monoidal, defining the associativity and unit constraints as in $\mathsf{Bialg}$. Moreover, the projections $p_1$ and $p_2$ are clearly morphisms in $\mathsf{QTrBialg}$. However, $((H_{1},\Rr_{1})\ot(H_{2},\Rr_{2}),p_{1},p_{2})$ does not coincide with the binary product of $(H_{1},\Rr_{1})$ and $(H_{2},\Rr_{2})$ in $\mathsf{QTrBialg}$ (or $\mathsf{TrBialg}$) by a mere extension of the cocommutative case argument. In fact, consider another (quasi)triangular bialgebra $\left( H,\Ss \right) $ and two morphisms  of (quasi)triangular bialgebras $f_{1}:\left( H,\Ss \right) \rightarrow \left( H_{1},\Rr %
_{1}\right) $ and $f_{2}:\left( H,\Ss \right) \rightarrow \left(
H_{2},\Rr _2\right) $.
Mimicking the cocommutative case, define
$f=\left( f_{1}\otimes f_{2}\right) \Delta _{H}:H\to H_1\otimes H_2$ and let us see if it defines a morphism of triangular bialgebras
\begin{equation*}
f:\left( H,\Ss \right) \rightarrow \left( H_{1}\otimes H_{2},\left(
\id \otimes \tau \otimes \id \right) \left( \Rr %
_{1}\otimes \Rr _{2}\right) \right).
\end{equation*}%
Clearly, $f$ is
a counitary algebra map as it is composition of such maps. Moreover, we have%
\begin{eqnarray*}
\Delta _{H_{1}\otimes H_{2}}f\left( x\right)  &=&\Delta _{H_{1}\otimes
H_{2}}\left( f_{1}\left( x_{1}\right) \otimes f_{2}\left( x_{2}\right)
\right)  \\
&=&f_{1}\left( x_{1}\right) _{1}\otimes f_{2}\left( x_{2}\right) _{1}\otimes
f_{1}\left( x_{1}\right) _{2}\otimes f_{2}\left( x_{2}\right) _{2} \\
&=&f_{1}\left( x_{11}\right) \otimes f_{2}\left( x_{21}\right) \otimes
f_{1}\left( x_{12}\right) \otimes f_{2}\left( x_{22}\right)  \\
&=&f_{1}\left( x_{1}\right) \otimes f_{2}\left( x_{3}\right) \otimes
f_{1}\left( x_{2}\right) \otimes f_{2}\left( x_{4}\right)  \\
&\overset{\eqref{qtr1}}=&f_{1}\left( x_{1}\right) \otimes f_{2}(\Ss ^{i}x_{2}\overline{%
\Ss }^{j})\otimes f_{1}(\Ss _{i}x_{3}\overline{\Ss }%
_{j})\otimes f_{2}\left( x_{4}\right)  \\
&=&f_{1}\left( x_{1}\right) \otimes f_{2}\left( \Ss ^{i}\right)
f_{2}\left( x_{2}\right) f_{2}\left( \overline{\Ss }^{j}\right)
\otimes f_{1}\left( \Ss _{i}\right) f_{1}\left( x_{3}\right)
f_{1}\left( \overline{\Ss }_{j}\right) \otimes f_{2}\left(
x_{4}\right)  \\
&=&\left( 1\otimes f_{2}\left( \Ss ^{i}\right) \otimes f_{1}\left(
\Ss _{i}\right) \otimes 1\right) \left( f_{1}\left( x_{1}\right)
\otimes f_{2}\left( x_{2}\right) \otimes f_{1}\left( x_{3}\right) \otimes
f_{2}\left( x_{4}\right) \right) \left( 1\otimes f_{2}\left( \overline{%
\Ss }^{j}\right) \otimes f_{1}\left( \overline{\Ss }%
_{j}\right) \otimes 1\right)  \\
&=&\left( 1\otimes f_{2}\left( \Ss ^{i}\right) \otimes f_{1}\left(
\Ss _{i}\right) \otimes 1\right) \left( f\left( x_{1}\right) \otimes
f\left( x_{2}\right) \right) \left( 1\otimes f_{2}\left( \overline{\Ss}^{j}\right) \otimes f_{1}\left( \overline{\Ss }_{j}\right) \otimes
1\right)  \\
&=&\left( 1\otimes \left( f_{2}\otimes f_{1}\right) \left( \Ss %
\right) \otimes 1\right) \left( f\otimes f\right) \Delta _{H}\left( x\right)
\left( 1\otimes \left( f_{2}\otimes f_{1}\right) \left( \Ss %
^{-1}\right) \otimes 1\right), 
\end{eqnarray*}%
hence $f$ is not a morphism of coalgebras in general.

If we set $\Ff:=1\otimes \left( f_{2}\otimes f_{1}\right) \left( \Ss %
^{-1}\right) \otimes 1$, the previous equality rewrites as 
\[\Ff \Delta _{H_1\otimes H_2 }f\left( x\right) \Ff ^{-1}=\left(
f\otimes f\right) \Delta_H \left( x\right). \]
\end{invisible}

In order to extend the binary product construction, we consider the $2$-category of bialgebras and twisted morphisms of bialgebras introduced in \cite{Davydov}.

We refer to \cite[\S  7.1]{BorI94} or \cite[\S 2.3]{Johnson-Yau-book} for basic facts about (strict) $2$-categories. We just mention that a (locally small) $2$-category is precisely a category enriched over $\Cat$, see \cite[Proposition 2.3.9]{Johnson-Yau-book}. In particular, a $2$-category $\Cc$ is a category where, for all objects $X$ and $Y$, the hom-set $\Cc(X,Y)$ is a category itself, whose objects are called $1$-cells while the morphisms are called $2$-cells.

\subsection{The \texorpdfstring{$2$}{2}-category \texorpdfstring{$\TwBialg$}{TwBialg} of bialgebras and twisted morphisms}
Following \cite[pages 2 and 16]{Davydov}, we consider the category $\TwBialg$  whose objects are bialgebras and whose morphisms from a bialgebra $H$ to a bialgebra $H'$ are pairs $\left(f,\Ff \right)$ consisting of a twist $\Ff \in H^{\prime }\otimes H^{\prime }$ and a morphism of bialgebras $f:H\rightarrow H_{\Ff }^{\prime
}$. They are denoted by $\left(f,\Ff \right):H\to H'$ and called \emph{twisted morphisms of bialgebras}. 

In this category, the composition of twisted morphisms of  bialgebras $\left( f,\Ff \right) :H \rightarrow
H'$ and $\left( f^{\prime },%
\Ff ^{\prime }\right) :H' \rightarrow H'' $ is defined as
\begin{equation}
\label{eq:compo}
\left( f^{\prime },\Ff ^{\prime }\right) \circ \left( f,\Ff %
\right) :=\left( f^{\prime }f,\left( f^{\prime }\otimes f^{\prime }\right)
\left( \Ff \right) \Ff ^{\prime }\right) .
\end{equation}
The identity is $(\id,1\otimes 1):H\to H$. 

\begin{remark}
\label{rmk:PanseraLomp}
In \cite{Pansera} it is shown that, given a twisted morphism of bialgebras $(f,\Ff):H\to H$, then the bialgebra structure of $H$ can be extended to the skew polynomial ring $H[x;f]$, such that $\Delta(x)=\Ff(x\ot x)$ and $\varepsilon(x)=1$. We point out that in \cite{Pansera} the notion of twist (and, accordingly, that of twisted morphism of bialgebras) is dual with respect to Definition \ref{def:twist}. The result achieved in \cite{Pansera} has been extended in \cite{Lomp} to $H=B^{\ot m}$, for an arbitrary bialgebra $B$ and a central twist $\Ff$ on $B$ that satisfies the condition that $\Ff^{i}\ot1_{B}\ot1_{B}\ot\Ff_{i}$ is a twist on $B\ot B$. There, the skew polynomial ring is replaced by a skew monoid algebra $H\#M$ over a suitable free monoid $M$.
\end{remark}

Following \cite[page 5]{Davydov}, we also recall that $\TwBialg$ is in fact a (locally small) 2-category. 

In order to describe it explicitly, we need the following notion.

\begin{definition}[{cf.\ \cite[page 3]{Davydov}}]
Given two morphisms $(f',\Ff'),(f,\Ff):H\to H'$ in $\TwBialg$, a \emph{gauge transformation} $a:(f,\Ff)\Rightarrow (f',\Ff')$
consists of an element $a\in H'$ such that $\varepsilon'(a)=1$ and 
\begin{align}
  (a\ot a)\Ff&=\Ff'\Delta'(a); \label{def:gauge1}\\
  af(x)&=f'(x)a.
  \label{def:gauge2}
\end{align}
\end{definition}

Then, $\TwBialg$ is, in fact, a $2$-category  where: 
\begin{itemize}
    \item[i)] the $0$-cells are bialgebras $H$,
    \item[ii)] the $1$-cells are twisted morphisms of bialgebras $(f,\Ff):H\to H'$,
    \item[iii)] the $2$-cells are gauge transformations $a:(f,\Ff)\Rightarrow (f',\Ff')$. 
\end{itemize}
The vertical composition $\vc$ and the horizontal composition $\hc$ are given, respectively, by 
 \[\xymatrix@C=2cm{
H\ruppertwocell^{(f,\Ff)}{b}
\rlowertwocell_{(f'',\Ff'')}{a}
\ar[r]|(.3){(f',\Ff')} & H'}
\; =\;
\xymatrix@R=3cm@C=1cm{H\rtwocell^{(f,\Ff)}_{(f'',\Ff'')}{ab}&H'},
\qquad
\xymatrix{H\rtwocell^{(f,\Ff)}_{(f',\Ff')}{b} & H' \rtwocell^{(g,\Gg)}_{(g',\Gg')}{a}& H''}
\;= \; \xymatrix@C=1.3cm{H \rrtwocell^{(gf,(g\otimes g)(\Ff)\Gg)}_{(g'f',(g'\otimes g')(\Ff')\Gg')}{\mathrlap{ag(b)}}&  & H''}\]
i.e.\ $a\vc b=ab$ and $a\hc b=ag(b).$

\begin{remark}
The horizontal composition is not written explicitly in \cite{Davydov} but it can be deduced from \cite[page 4]{Davydov}. Indeed, by interchange law we have
\[a\hc b
=(a\vc 1)\hc (1\vc b)
=(a\hc 1)\vc (1\hc b)
=a\vc g(b)=ag(b).\]
\[\xymatrix{H\ar@/^{2pc}/[rr]|{(f,\Ff)}="a"
\ar[rr]|{(f',\Ff')}="b"
\ar@/^{-2pc}/[rr]|{(f',\Ff')}="c"
&&H'
\ar@{=>}"a";"b"^{b}
\ar@{=>}"b";"c"^{1}
\ar@/^{2pc}/[rr]|{(g,\Gg)}="a'"
\ar[rr]|{(g,\Gg)}="b'"
\ar@/^{-2pc}/[rr]|{(g',\Gg')}="c'"
\ar@{=>}"a'";"b'"^{1}
\ar@{=>}"b'";"c'"^{a}
&&H''}\]
The horizontal composition is well defined as 
\[
ag(b)gf(x)=ag(bf(x))=ag(f'(x)b)=ag(f'(x))g(b)=g'f'(x)ag(b)
\]
and 
\[
\begin{split}
(ag(b)\otimes ag(b))(g\otimes g)(\Ff)\Gg
&=(a\otimes a)\big((g\otimes g)((b\otimes b)\Ff)\big)\Gg
=(a\otimes a)\big((g\otimes g)(\Ff'\Delta'(b))\big)\Gg
\\&=(a\otimes a)(g\otimes g)(\Ff')(g\otimes g)\Delta'(b)\Gg
=(a\otimes a)(g\otimes g)(\Ff')\Gg\Delta'' g(b)\\&=(g'\ot g')(\Ff')(a\ot a)\Gg\Delta''g(b)=(g'\ot g')(\Ff')\Gg'\Delta''(a)\Delta''g(b)\\&=(g'\otimes g')(\Ff')\Gg'\Delta''(ag(b)).
\end{split}
\]
\end{remark}

\begin{remark}
\label{rmk:hata}
If $a\in H^{\times}$, from \eqref{def:gauge1}, we get that $\Ff^a =\Ff'$, where we use \eqref{def:F^h}, i.e.\ $\Ff'$ and $\Ff$ are cohomologous. Moreover, from \eqref{def:gauge2}, we have $\hat{a}\circ  f=f'$, where $\hat{a}
:=a\left( -\right)  a
^{-1}:H' _{\Ff}\rightarrow H' _{\Ff'}$.
\[\xymatrix@R=12pt{&H\ar[dl]_{f}\ar[dr]^{f'}\\H'_{\Ff}\ar[rr]^{\hat{a}}&&H'_{\Ff'} }\]

In \eqref{def:gauge1}, the roles of $\Ff$ and $\Ff'$ are  exchanged with respect to \cite{Davydov}.  This change is due to the fact that
\[
\begin{split}
\Delta'_{\Ff'}\hat{a}(x)&=\Ff'\Delta'(axa^{-1})(\Ff')^{-1}
=\Ff'\Delta'(a)\Delta'(x)\Delta'(a^{-1})(\Ff')^{-1}
=(a\otimes a) \Ff\Delta'(x)\Ff^{-1}(a^{-1}\otimes a^{-1})
\\&=(\hat{a}\otimes \hat{a})\Delta'_{\Ff}(x),
\end{split}
\]
so that $\hat{a}
:H' _{\Ff}\rightarrow H' _{\Ff'}$ is indeed a coalgebra map with domain and codomain in the given order which makes sense to the composition $\hat{a}\circ  f=f'$. Note that our \eqref{def:gauge1} agrees with \cite[\S1]{Davydov-TwDer}, where the particular gauge transformations considered have $a$ invertible.
\end{remark}

\begin{remark}
\label{rmk:gaugeps}
In \cite{Davydov}, the condition $\varepsilon'(a)=1$ is not part of the definition of gauge transformation but it is incorrectly observed that it follows from \eqref{def:gauge1} together with normalisation conditions for twists. What is true is that $\varepsilon'(a)=1$ if, and only if, $a\neq 0$. Of course, this holds when $a$ is invertible as in \cite{Davydov-TwDer}.
\end{remark}

\subsection{The \texorpdfstring{$2$}{2}-category \texorpdfstring{$\TwTrBialg$}{TwTrBialg} of triangular bialgebras and twisted morphisms}
We now want to define another $2$-category, whose 0-cells are triangular bialgebras, endowed with a $2$-functor (see \cite[\S 4.1]{Johnson-Yau-book}) into $\TwBialg$. From now on, the quasitriangular structures considered will be triangular.

\begin{definition}
Let $(H,\Rr)$ and $(H',\Rr')$ be triangular bialgebras. We say that $\left(f,\Ff \right):(H,\Rr)\to (H',\Rr')$ is a \emph{twisted morphism of triangular bialgebras} if $f:(H,\Rr)\to(H'_{\Ff},\Rr'_{\Ff})$ is a morphism in the category $\TrBialg$ of triangular bialgebras.
\end{definition}

\begin{remark}
Recall that $H_{\Ff }^{\prime }=\left( H^{\prime
},m^{\prime },u^{\prime },\Delta _{\Ff }^{\prime },\varepsilon
^{\prime }\right) $, where $\left( H^{\prime },m^{\prime },u^{\prime },\Delta
^{\prime },\varepsilon ^{\prime }\right) $ denotes the starting bialgebra
structure on $H^{\prime }$ while $\Delta _{\Ff }^{\prime }\left(
\cdot\right) =\Ff \Delta ^{\prime }\left( \cdot\right) \Ff ^{-1}.$ Therefore, $\left( f,\Ff \right) :\left( H,\Rr \right) \rightarrow
\left( H^{\prime },\Rr ^{\prime }\right) $ is a twisted
morphism of triangular bialgebras if $f:H\rightarrow H^{\prime }$ is a counitary algebra
map such that 
\begin{equation}\label{eq:deformedcoproduct}
\Ff \Delta ^{\prime }f\left( x\right) \Ff ^{-1}=\left(
f\otimes f\right) \Delta \left( x\right) 
\end{equation}%
and verifies%
\begin{equation}\label{eq:deformedR}
\left( f\otimes f\right) \left( \Rr \right) =\Ff ^{\mathrm{op}}%
\Rr ^{\prime }\Ff ^{-1},
\end{equation}
i.e.\ if $(f,\Ff):H\to H'$ is a morphism in $\TwBialg$ and \eqref{eq:deformedR} is satisfied.
\end{remark}

\begin{remark}
We observe that the composition of twisted morphisms of  triangular bialgebras $\left( f,\Ff \right) :\left( H,\Rr \right) \rightarrow
\left( H^{\prime },\Rr ^{\prime }\right) $ and $\left( f^{\prime },%
\Ff ^{\prime }\right) :\left( H^{\prime },\Rr ^{\prime
}\right) \rightarrow \left( H^{\prime \prime },\Rr ^{\prime \prime
}\right) $ is still a twisted morphism of triangular bialgebras: 
\begin{eqnarray*}
\left( f^{\prime }f\otimes f^{\prime }f\right) \left( \Rr \right) 
&=&\left( f^{\prime }\otimes f^{\prime }\right) \left( f\otimes f\right)
\left( \Rr \right) =\left( f^{\prime }\otimes f^{\prime }\right)
\left( \Ff ^{\mathrm{op}}\mathcal{R'F}^{-1}\right)  \\
&=&\left( f^{\prime }\otimes f^{\prime }\right) \left( \Ff \right) ^{%
\mathrm{op}}\left( f^{\prime }\otimes f^{\prime }\right) \left( \Rr %
^{\prime }\right) \left( f^{\prime }\otimes f^{\prime }\right) \left( 
\Ff \right) ^{-1} \\
&=&\left( f^{\prime }\otimes f^{\prime }\right) \left( \Ff \right) ^{%
\mathrm{op}}\Ff ^{\prime \mathrm{op}}\Rr ^{\prime \prime }%
\Ff ^{\prime -1}\left( f^{\prime }\otimes f^{\prime }\right) \left( 
\Ff \right) ^{-1} \\
&=&((f'\ot f')(\Ff)\Ff')^{\mathrm{op}}\Rr''((f'\ot f')(\Ff)\Ff')^{-1}.
\end{eqnarray*}
Moreover, $\id_{(H,\Rr)}\coloneqq\left( \id _H%
,1\otimes 1\right) :\left( H,\Rr \right) \rightarrow \left( H,%
\Rr \right) $ is clearly a twisted morphism of triangular bialgebras.
\end{remark}

\begin{remark}
    If $H$ and $H'$ are cocommutative bialgebras, one can consider the triangular structures $\Rr=1_{H}\ot1_{H}$ and $\Rr'=1_{H'}\ot_{H'}$. Therefore, a morphism $(f,\Ff):H\to H'$ in $\TwBialg$ is a twisted morphism of triangular bialgebras if and only if $\Ff=\Ff^{\mathrm{op}}$. In particular, $(f,1\ot1):(H,1\ot1)\to(H',1\ot1)$ is a twisted morphism of triangular bialgebras. 
\end{remark}

\begin{example}\label{ex:groupalgebra}
    Consider the abelian group $\Gamma=\langle x,y\ |\ xy=yx,\ x^{n}=1,\ y^{n}=1\rangle$, with $n>1$, and let $q\in\Bbbk$ be a primitive $n$-th root of unity. One can define a morphism $(f,\Ff):\Bbbk\Gamma\to\Bbbk\Gamma$ in $\TwBialg$ as in \cite[Example 3.2]{ShilinZhang}, see also \cite{Pansera}. We point out that in \cite{ShilinZhang} the notion of twist (and, accordingly, that of twisted morphism of bialgebras) is dual with respect to Definition \ref{def:twist}. Thus, here we consider the inverse of the twist considered there. We define $f:\Bbbk\Gamma\to\Bbbk\Gamma$ on $x$ and $y$ as $f(x)=y$ and $f(y)=x$ and we extend it to an algebra automorphism of $\Bbbk\Gamma$. Moreover, defining $\Ff:=\frac{1}{n}\sum_{i,j=0}^{n-1}{q^{-ij}x^{i}\ot y^{-j}}$, we obtain a morphism $(f,\Ff):\Bbbk\Gamma\to\Bbbk\Gamma$ in $\TwBialg$. The bialgebra $\Bbbk\Gamma$ is cocommutative, but we can choose triangular structures $\Rr$  different from $1\ot1$, see e.g.\ \cite[Example 2.1.17]{Majid-book}. If we define $\Rr':=(f\ot f)(\Rr)_{\Ff^{-1}}$ we obtain that $(f,\Ff):(\Bbbk\Gamma,\Rr)\to(\Bbbk\Gamma,\Rr')$ is a (bijective) twisted morphism of triangular bialgebras. 
\end{example}

\begin{example}
\label{exa:Sw1}
Let $\Bbbk$ be a field of characteristic different from 2. The Sweedler Hopf algebra $H$ is the $\Bbbk$-algebra given by generators $g$ and $x$, and relations
\[
g^{2}=1,\quad x^{2}=0,\quad  xg=-gx. 
\]    
This becomes a Hopf algebra with comultiplication, counit and antipode determined by 
\[
\Delta(g)=g\ot g,\ \Delta(x)=x\ot 1+g\ot x,\ \varepsilon(g)=1,\ \varepsilon(x)=0,\ S(g)=g,\ S(x)=-gx.
\]
There is an exhaustive 1-parameter family of triangular structures on $H$ given by 
\[
\Rr_{\lambda}:=\frac{1}{2}(1\ot1+g\ot1+1\ot g-g\ot g)+\frac{\lambda}{2}(x\ot x-xg\ot x+x\ot xg+xg\ot xg),
\]
see e.g.\ \cite[Exercise 2.1.7]{Majid-book}. One can define a morphism $(f,\Ff):H\to H$ in $\TwBialg$ as in \cite[Example 3.1]{ShilinZhang}. Again we point out that in \cite{ShilinZhang} the notion of twist (and, accordingly, that of twisted morphism of bialgebras) is dual with respect to Definition \ref{def:twist}. Thus, here we consider the inverse of the twist considered there. We define $f_{s}:H\to H$ on $g$ and $x$ as $f_{s}(g)=g$ and $f_{s}(x)=sx$, for $s\in\Bbbk$ with $s\not=0$, and we extend it into an algebra automorphism of $H$. Moreover, defining $\Ff_{d}:=1\ot1+d(xg\ot x)$, for $d\in\Bbbk$, we obtain a morphism $(f_{s},\Ff_{d}):H\to H$ in $\TwBialg$. We observe that $(\Rr_{\lambda})_{\Ff}=\Rr_{\lambda+2d}$: 
\[
\begin{split}
    (\Rr_{\lambda})_{\Ff}&=(1\ot1+d(x\ot xg))\Rr_{\lambda}(1\ot 1-d(xg\ot x))\\&=\Rr_{\lambda}-\frac{d}{2}(1\ot1+g\ot 1+1\ot g-g\ot g)(xg\ot x)+\frac{d}{2}(x\ot xg)(1\ot 1+g\ot 1+1\ot g-g\ot g)\\&=\Rr_{\lambda}+d(x\ot xg+x\ot x+xg\ot xg-xg\ot x)=\Rr_{\lambda+2d}.
\end{split}
\]
Therefore, $(f_{s},\Ff_{d}):(H,\Rr_{\lambda})\to(H,\Rr_{\gamma})$ is a twisted morphism of triangular bialgebras if and only if 
$\Rr_{\lambda s^{2}}=(f\ot f)(\Rr_{\lambda})=(\Rr_{\gamma})_{\Ff}=\Rr_{\gamma+2d}$,
i.e.\ if and only if $\lambda s^{2}=\gamma+2d$. 

Hence $(f_{s},\Ff_{d}):(H,\Rr_{\lambda})\to(H,\Rr_{\lambda s^{2}-2d})$ is a (bijective)  twisted morphism of triangular bialgebras. In particular, $(f_{1}=\id ,\Ff_{d}):(H,\Rr_{\lambda})\to (H,\Rr_{\lambda-2d})$ is a (bijective) twisted morphism of triangular bialgebras.
\end{example}

\begin{definition}
We define the 2-category $\TwTrBialg$ where:
\begin{itemize}
    \item[i)] the $0$-cells are triangular bialgebras $(H,\Rr)$,
    \item[ii)] the $1$-cells are twisted morphisms of triangular bialgebras $(f,\Ff):(H,\Rr)\to (H',\Rr')$,
    \item[iii)] the $2$-cells are the gauge transformations $a:(f,\Ff)\Rightarrow (f',\Ff')$ between the underlying twisted morphisms of bialgebras. 
\end{itemize}


\end{definition}
Consider the assignments
\begin{equation}
\label{def:funcF}
F:\TwTrBialg\to\TwBialg,\;(H,\Rr)\mapsto H,\;(f,\Ff )\mapsto (f,\Ff ),\;a\mapsto a.    \end{equation}
Since the identity $1$-cells, the identity $2$-cells, the composition of $2$-cells, and horizontal compositions of $1$-cells and $2$-cells are defined in the same way in the $2$-categories $\TwTrBialg$ and $\TwBialg$, it is clear that $F$ preserves them. Thus it is a $2$-functor, see also \cite[Explanation 4.1.9]{Johnson-Yau-book}.

\begin{invisible}
 Following \cite[page 105]{Johnson-Yau-book}, a $2$-functor is a strict functor between two 2-categories. One the previous line in loc.\ cit.\ a strict functor is defined as a lax functor $(F,F^2,F^0):\Aa\to \Bb$ in which $F^2$ and $F^0$ are identity natural transformations. Here $F:Ob(\Aa)\to Ob(\Bb)$ is a function on objects. For each pair of objects there is a \emph{local functor} \[F_{X,Y}:\Aa(X,Y)\to \Bb(F(X),F(Y)),\;[f:X\to Y]\mapsto [Ff:FX\to FY],\;[\alpha:f\Rightarrow g]\mapsto [F\alpha:Ff\Rightarrow Fg].\] The fact that $F^2$ is the identity natural
transformation just means that $F$ preserves the composition of $1$-cells and the horizontal composition of $2$-cells, while the fact that $F^0$ is the identity natural
transformation just means that $F$ preserves the identity $1$-cell and the identity $2$-cell so that it is a usual functor. 

In our case $F:Ob(\TwTrBialg)\to Ob(\TwBialg),(H,\Rr)\mapsto H$ while the local functor is 
\begin{align*}
   F_{(H,\Rr),(H',\Rr')}:\TwTrBialg((H,\Rr),(H',\Rr'))&\to \TwBialg(H,H'),\\
   [(f,\Ff):(H,\Rr)\to(H',\Rr')]&\mapsto [(f,\Ff):H\to H']\\
   [a:(f,\Ff)\Rightarrow (f',\Ff)']&\mapsto [a:(f,\Ff)\Rightarrow (f',\Ff)'].
\end{align*}

\end{invisible}

\begin{remark}
We observe that, given two morphisms $(f',\Ff'),(f,\Ff):(H,\Rr)\to (H',\Rr')$ in $\TwTrBialg$ and a gauge transformation $a:(f,\Ff)\Rightarrow (f',\Ff')$, we have
\begin{align*}
    (a\otimes a)\Rr'_{\Ff}
&=(a\otimes a)\Ff^\op\Rr'\Ff^{-1}=((a\ot a)\Ff)^{\mathrm{op}}\Rr'\Ff^{-1}=(\Ff'\Delta'(a))^{\mathrm{op}}\Rr'\Ff^{-1}\\=&\,(\Ff')^\op(\Delta')^\op(a)\Rr'\Ff^{-1}
\overset{\eqref{qtr1}}=(\Ff')^\op\Rr'\Delta'(a)\Ff^{-1}=(\Ff')^\op\Rr'(\Ff')^{-1}(a\ot a)=\Rr'_{\Ff'}(a\ot a).
\end{align*}
Hence, the gauge transformation is automatically compatible with the triangular structures involved. Moreover,
if $a\in H^{\times}$, then the bialgebra map $\hat{a}
:=a\left( -\right)  a
^{-1}:H' _{\Ff}\rightarrow H' _{\Ff'}$ is a morphism of triangular bialgebras. 
\end{remark}

\begin{invisible}We keep here the computations we did before discovering \cite{Davydov}.

$i)$\ Let $\left( f,\Ff \right) :\left( H,\Rr \right)
\rightarrow \left( H^{\prime },\Rr ^{\prime }\right) $ and $\left(
f^{\prime },\Ff ^{\prime }\right) :\left( H^{\prime },\Rr %
^{\prime }\right) \rightarrow \left( H^{\prime \prime },\Rr ^{\prime
\prime }\right) $ be twisted morphism of triangular bialgebras. Then $\Ff$ is a twist on $H^{\prime }$, $f:H\rightarrow H_{\Ff %
}^{\prime }$ is a bialgebra map, $\Ff ^{\prime }$ is a Drinfel'd
twist on $H^{\prime \prime }$ and $f^{\prime }:H^{\prime }\rightarrow H_{%
\Ff }^{\prime \prime }$ is a bialgebra map. Let us check that $%
\Ff ^{\prime \prime }:=\left( f^{\prime }\otimes f^{\prime }\right)
\left( \Ff \right) \Ff ^{\prime }$ is a twist on $%
H^{\prime \prime }.$%
\begin{eqnarray*}
\left( \Ff ^{\prime \prime }\otimes 1\right) \left( \Delta ^{\prime
\prime }\otimes \id \right) \left( \Ff ^{\prime \prime
}\right)  &=&\left( \left( f^{\prime }\otimes f^{\prime }\right) \left( 
\Ff \right) \Ff ^{\prime }\otimes 1\right) \left( \Delta
^{\prime \prime }\otimes \id \right) \left( \left( f^{\prime }\otimes
f^{\prime }\right) \left( \Ff \right) \Ff ^{\prime }\right)  \\
&=&\left( \left( f^{\prime }\otimes f^{\prime }\right) \left( \Ff %
\right) \otimes 1\right) \underleftrightarrow{\left( \Ff ^{\prime
}\otimes 1\right) \left( \Delta ^{\prime \prime }\otimes \id \right)
\left( f^{\prime }\otimes f^{\prime }\right) \left( \Ff \right) }%
\left( \Delta ^{\prime \prime }\otimes \id \right) \left( \Ff %
^{\prime }\right)  \\
&=&\left( \left( f^{\prime }\otimes f^{\prime }\right) \left( \Ff %
\right) \otimes 1\right) \left( f^{\prime }\otimes f^{\prime }\otimes
f^{\prime }\right) \left( \Delta ^{\prime }\otimes \id \right) \left( 
\Ff \right) \left( \Ff ^{\prime }\otimes 1\right) \left(
\Delta ^{\prime \prime }\otimes \id \right) \left( \Ff %
^{\prime }\right)  \\
&=&\left( f^{\prime }\otimes f^{\prime }\otimes f^{\prime }\right) \left[
\left( \Ff \otimes 1\right) \left( \Delta ^{\prime }\otimes \mathrm{Id%
}\right) \left( \Ff \right) \right] \left( \Ff ^{\prime
}\otimes 1\right) \left( \Delta ^{\prime \prime }\otimes \id \right)
\left( \Ff ^{\prime }\right)  \\
&=&\left( f^{\prime }\otimes f^{\prime }\otimes f^{\prime }\right) \left[
\left( 1\otimes \Ff \right) \left( \id \otimes \Delta ^{\prime
}\right) \left( \Ff \right) \right] \left( 1\otimes \Ff %
^{\prime }\right) \left( \id \otimes \Delta ^{\prime \prime }\right)
\left( \Ff ^{\prime }\right)  \\
&=&\left( 1\otimes \left( f^{\prime }\otimes f^{\prime }\right) \left( 
\Ff \right) \right) \underleftrightarrow{\left( f^{\prime }\otimes
f^{\prime }\otimes f^{\prime }\right) \left( \id \otimes \Delta
^{\prime }\right) \left( \Ff \right) \left( 1\otimes \Ff %
^{\prime }\right) }\left( \id \otimes \Delta ^{\prime \prime }\right)
\left( \Ff ^{\prime }\right)  \\
&=&\left( 1\otimes \left( f^{\prime }\otimes f^{\prime }\right) \left( 
\Ff \right) \right) \left( 1\otimes \Ff ^{\prime }\right)
\left( \id \otimes \Delta ^{\prime \prime }\right) \left( f^{\prime
}\otimes f^{\prime }\right) \left( \Ff \right) \left( \id %
\otimes \Delta ^{\prime \prime }\right) \left( \Ff ^{\prime }\right) 
\\
&=&\left( 1\otimes \Ff ^{\prime \prime }\right) \left( \id %
\otimes \Delta ^{\prime \prime }\right) \left( \Ff ^{\prime \prime
}\right) 
\end{eqnarray*}%
We also have%
\begin{eqnarray*}
\left( \varepsilon ^{\prime \prime }\otimes \id \right) \left( 
\Ff ^{\prime \prime }\right)  &=&\left( \varepsilon ^{\prime \prime
}\otimes \id \right) \left( \left( f^{\prime }\otimes f^{\prime
}\right) \left( \Ff \right) \Ff ^{\prime }\right)  \\
&=&\left( \varepsilon ^{\prime \prime }\otimes \id \right) \left(
f^{\prime }\otimes f^{\prime }\right) \left( \Ff \right) \left(
\varepsilon ^{\prime \prime }\otimes \id \right) \left( \Ff %
^{\prime }\right)  \\
&=&f^{\prime }\left( \varepsilon ^{\prime }\otimes \id \right) \left( 
\Ff \right) 1^{\prime }=f^{\prime }\left( 1\right) 1^{\prime
}=1^{\prime }
\end{eqnarray*}%
and%
\begin{eqnarray*}
\left( \id \otimes \varepsilon ^{\prime \prime }\right) \left( 
\Ff ^{\prime \prime }\right)  &=&\left( \id \otimes
\varepsilon ^{\prime \prime }\right) \left( \left( f^{\prime }\otimes
f^{\prime }\right) \left( \Ff \right) \Ff ^{\prime }\right)  \\
&=&\left( \id \otimes \varepsilon ^{\prime \prime }\right) \left(
f^{\prime }\otimes f^{\prime }\right) \left( \Ff \right) \left( 
\id \otimes \varepsilon ^{\prime \prime }\right) \Ff ^{\prime }
\\
&=&f^{\prime }\left( \id \otimes \varepsilon ^{\prime }\right) \left( 
\Ff \right) 1^{\prime }=f^{\prime }\left( 1\right) 1^{\prime
}=1^{\prime }.
\end{eqnarray*}%
Thus $\Ff ^{\prime }$ is a twist. Clearly $f^{\prime
}f:H\rightarrow H^{\prime \prime }$ is a counitary algebra map. Moreover,
we have 
\begin{eqnarray*}
\left( f^{\prime }f\otimes f^{\prime }f\right) \Delta \left( x\right) 
&=&\left( f^{\prime }\otimes f^{\prime }\right) \left( f\otimes f\right)
\Delta \left( x\right) =\left( f^{\prime }\otimes f^{\prime }\right) \left( 
\Ff \Delta ^{\prime }f\left( x\right) \Ff ^{-1}\right)  \\
&=&\left( f^{\prime }\otimes f^{\prime }\right) \left( \Ff \right)
\left( f^{\prime }\otimes f^{\prime }\right) \left( \Delta ^{\prime }f\left(
x\right) \right) \left( f^{\prime }\otimes f^{\prime }\right) \left( 
\Ff ^{-1}\right)  \\
&=&\left( f^{\prime }\otimes f^{\prime }\right) \left( \Ff \right) 
\Ff ^{\prime }\left( \Delta ^{\prime \prime }f^{\prime }f\left(
x\right) \right) \left( \Ff ^{\prime }\right) ^{-1}\left( f^{\prime
}\otimes f^{\prime }\right) \left( \Ff ^{-1}\right) 
\end{eqnarray*}%
so that $\left( f^{\prime }f,\left( f^{\prime }\otimes f^{\prime }\right)
\left( \Ff \right) \Ff ^{\prime }\right) :\left( H,\Rr %
\right) \rightarrow \left( H^{\prime \prime },\Rr ^{\prime \prime
}\right) $ is a twisted  morphism of triangular bialgebras once proved the
compatibility with the quasitiangular structure. Indeed we have%
\begin{eqnarray*}
\left( f^{\prime }f\otimes f^{\prime }f\right) \left( \Rr \right) 
&=&\left( f^{\prime }\otimes f^{\prime }\right) \left( f\otimes f\right)
\left( \Rr \right) =\left( f^{\prime }\otimes f^{\prime }\right)
\left( \Ff ^{\mathrm{op}}\mathcal{R'F}^{-1}\right)  \\
&=&\left( f^{\prime }\otimes f^{\prime }\right) \left( \Ff \right) ^{%
\mathrm{op}}\left( f^{\prime }\otimes f^{\prime }\right) \left( \Rr %
^{\prime }\right) \left( f^{\prime }\otimes f^{\prime }\right) \left( 
\Ff \right) ^{-1} \\
&=&\left( f^{\prime }\otimes f^{\prime }\right) \left( \Ff \right) ^{%
\mathrm{op}}\Ff ^{\prime \mathrm{op}}\Rr ^{\prime \prime }%
\Ff ^{\prime -1}\left( f^{\prime }\otimes f^{\prime }\right) \left( 
\Ff \right) ^{-1} \\
&=&\Ff ^{\prime \prime \mathrm{op}}\Rr ^{\prime \prime }%
\Ff ^{\prime \prime -1}.
\end{eqnarray*}

$ii)$ Let us check associativity. We compute%
\begin{eqnarray*}
\left( f^{\prime \prime },\Ff ^{\prime \prime }\right) \circ \left(
\left( f^{\prime },\Ff ^{\prime }\right) \circ \left( f,\Ff %
\right) \right)  &=&\left( f^{\prime \prime },\Ff ^{\prime \prime
}\right) \circ \left( f^{\prime }f,\left( f^{\prime }\otimes f^{\prime
}\right) \left( \Ff \right) \Ff ^{\prime }\right)  \\
&=&\left( f^{\prime \prime }f^{\prime }f,\left( f^{\prime \prime }\otimes
f^{\prime \prime }\right) \left[ \left( f^{\prime }\otimes f^{\prime
}\right) \left( \Ff \right) \Ff ^{\prime }\right] \Ff %
^{\prime \prime }\right)  \\
&=&\left( f^{\prime \prime }f^{\prime }f,\left( f^{\prime \prime }f^{\prime
}\otimes f^{\prime \prime }f^{\prime }\right) \left( \Ff \right)
\left( f^{\prime \prime }\otimes f^{\prime \prime }\right) \left( \Ff %
^{\prime }\right) \Ff ^{\prime \prime }\right)  \\
&=&\left( f^{\prime \prime }f^{\prime },\left( f^{\prime \prime }\otimes
f^{\prime \prime }\right) \left( \Ff ^{\prime }\right) \Ff %
^{\prime \prime }\right) \circ \left( f,\Ff \right)  \\
&=&\left( \left( f^{\prime \prime },\Ff ^{\prime \prime }\right)
\circ \left( f^{\prime },\Ff ^{\prime }\right) \right) \circ \left( f,%
\Ff \right) .
\end{eqnarray*}%
The rest of the proof is trivial.
\end{invisible}

\section{Binary products in the \texorpdfstring{$2$}{2}-category \texorpdfstring{$\TwTrBialg$}{TwTr}}\label{sec:binaryproduct}
We are now going to show that in the 2-category $\TwTrBialg$, the tensor product \eqref{def:tensorproductquasi} constitutes the binary product. Naturally, this requires us to consider the concept of binary product within the context of $2$-categories, see e.g.\ \cite[\S 2.2]{CKWW2}.

\begin{definition}
A \emph{product of two objects} $A,B$ in a $2$-category $\Cc$ is an
object $A\times B$ together with projections $p:A\times B\rightarrow A$ and $%
q:A\times B\rightarrow B$ such that, for any object $X$ in $\Cc$,
the functor
\[\Cc(X,A\times B)\to\Cc(X,A)\times \Cc(X,B),\,h\mapsto (ph,qh),\,\left[\xymatrix{X\rtwocell^{h}_{k}{\gamma}&A\times B}\right]\mapsto \left(\xymatrix{X\rtwocell^{ph}_{pk}{p\gamma}&A},\xymatrix{X\rtwocell^{qh}_{qk}{q\gamma}&B}\right),\] 
is an equivalence of categories (essentially surjective on objects and fully faithful).\\ Here $p\gamma\coloneqq \id_p\hc \gamma$, where we use the horizontal composition.

Explicitly, this means that:
\begin{itemize}
    \item[1)] for any $f:X\rightarrow A$ and $g:X\rightarrow B$, there exists a morphism $h:X\rightarrow A\times B$ and
isomorphisms $ph\cong f$ and $qh\cong g$; \medskip
\begin{equation*}
\xymatrix{&X\ar[ld]_{f}\ar[rd]^{g}\ar@{.>}[d]|{h}\\ A&
A\times B\ar[l]_{p}\ar[r]^{q}& B}
\end{equation*}%
    \item[2)] for any $h,k:X\rightarrow
A\times B$ and $2$-cells $\alpha:ph\Rightarrow pk$ and $\beta
:qh\Rightarrow qk$, there exists a unique $\gamma :h\Rightarrow k$ such
that $p\gamma =\alpha $ and $q\gamma=\beta $.
\[\xymatrix{ & X\ar[dl]_{ph} \ar[d]^-{k} \ar[dr]^{qh} & \\
A  & A\times B \ar[l]^{p}\ar[r]_{q}  \ar@{}[ur]|(.19){}="1"|(.42){}="2" \ar@{=>}"2";"1"^(.3){\beta}
\ar@{}[ul]|(.19){}="1"|(.42){}="2" \ar@{=>}"2";"1"^(.3){\alpha}& B
}\]
\end{itemize}
\end{definition}

Next aim is to prove that $\TwTrBialg$ has binary products. This will be done in successive steps. Note that in the context of $\TwTrBialg$ the functor above is even surjective on objects.

\begin{proposition}
\label{pro:Tw2cart1} 
Given $\left( f_{1},\Ff _{1}\right) :\left( H,\Ss \right)
\rightarrow \left( H_{1},\Rr _{1}\right) $ and $\left( f_{2},\Ff_{2}\right) :\left( H,\Ss \right) \rightarrow \left( H_{2},\Rr_{2}\right) $ any morphisms in $\TwTrBialg$, there is a morphism $\langle(f_1,\Ff_1),(f_2,\Ff_2)\rangle:\left( H,\Ss \right) \rightarrow
\left( H_{1},\Rr _{1}\right) \otimes \left( H_{2},\Rr %
_{2}\right) $ in $\TwTrBialg$ (the \emph{diagonal morphism}) such that the diagram
\begin{equation*}
\xymatrixcolsep{2cm}\xymatrixrowsep{1.3cm}\xymatrix{&(H,\Ss)\ar[ld]_{(f_1,%
\Ff_1)}\ar[rd]^{(f_2,\Ff_2)}\ar@{.>}[d]|{\langle(f_1,\Ff_1),(f_2,\Ff_2)\rangle}\\ (H_1,\Rr_1)&
(H_1,\Rr_1)\otimes (H_2,\Rr_2)\ar[l]_{(\id\otimes\varepsilon,1\otimes
1)}\ar[r]^{(\varepsilon\otimes\id,1\otimes 1)}& (H_2,\Rr_2)}
\end{equation*}%
commutes, namely 
\[\langle(f_1,\Ff_1),(f_2,\Ff_2)\rangle\coloneqq \big(\left( f_{1}\otimes f_{2}\right) \Delta _{H}\;,\;\left( 1_{H_{1}}\otimes \left( f_{2}\otimes
f_{1}\right) \left( \Ss ^{-1}\right) \otimes 1_{H_{2}}\right) \left(
\id \otimes \tau \otimes \id \right) \left( \Ff %
_{1}\otimes \Ff _{2}\right) \big).\]
\end{proposition}

\begin{proof}
    Set $p_{1}:=\id \otimes \varepsilon $ and $p_{2}:=\varepsilon \otimes
\id $. Observe that $p_{1}$ and $p_{2}$ are counitary algebra maps. Moreover, we
have
\begin{equation*}
(p_{1}\otimes p_{1})\Delta _{H_{1}\otimes H_{2}}(a\otimes
b)=a_{1}\varepsilon (b_{1})\otimes a_{2}\varepsilon (b_{2})=a_{1}\otimes
a_{2}\varepsilon (b)=(1\otimes 1)\Delta (a\varepsilon (b))(1\otimes 1)^{-1}
\end{equation*}%
and also
\begin{equation*}
(p_{1}\otimes p_{1})(\id \otimes \tau \otimes \id )(\Rr %
_{1}\otimes \Rr _{2})=(\id \otimes \id \otimes
\varepsilon \otimes \varepsilon )(\Rr _{1}\otimes \Rr _{2})=%
\Rr _{1}
\end{equation*}%
so that $(p_{1},1\otimes 1)$ is a morphism in $\TwTrBialg$. Similarly, $%
(p_{2},1\otimes 1)$ is a morphism in $\TwTrBialg$.

Note that $(f_2\otimes f_1)(\Ss)\in
H_2\otimes H_1$ is a weak $\Rr$-matrix of $((H_2)_{\Ff_2},(H_1)_{\Ff_1})$ by Example \ref{exa:weakR} applied to the bialgebra maps $f_1:H\to (H_1)_{\Ff_1}$ and $%
f_2:H\to (H_2)_{\Ff_2}$. Hence, using Lemma \ref{lem:Drinfeldtwist}, we
conclude that $\Ff \coloneqq \left( 1_{H_{1}}\otimes \left( f_{2}\otimes
f_{1}\right) \left( \Ss ^{-1}\right) \otimes 1_{H_{2}}\right) \left(
\id \otimes \tau \otimes \id \right) \left( \Ff %
_{1}\otimes \Ff _{2}\right) $ is a twist on $H_{1}\otimes H_{2}$. We prove that $(f,\Ff )$ is a twisted morphism of triangular
bialgebras for $f:=(f_{1}\otimes f_{2})\Delta_{H}$. Clearly $f$ is a counitary
algebra map since $f_{1}$, $f_{2}$ and $\Delta_{H}$ are so. Moreover, we compute
\begin{align*}
\Delta& _{H_1\otimes H_2}f\left( x\right) =\Delta _{H_1\otimes
H_2}\left( f_{1}\left( x_{1}\right) \otimes f_{2}\left( x_{2}\right)
\right) \\
&=\left( f_{1}\left( x_{1}\right) _{1}\otimes f_{2}\left( x_{2}\right)
_{1}\otimes f_{1}\left( x_{1}\right) _{2}\otimes f_{2}\left( x_{2}\right)
_{2}\right) \\
&=\overline{\Ff _{1}}^{i}f_{1}\left( x_{11}\right) \left( \Ff %
_{1}\right) ^{i}\otimes \overline{\Ff _{2}}^{i}f_{2}\left(
x_{21}\right) \left( \Ff _{2}\right) ^{i}\otimes \overline{\Ff %
_{1}}_{i}f_{1}\left( x_{12}\right) \left( \Ff _{1}\right) _{i}\otimes
\overline{\Ff _{2}}_{i}f_{2}\left( x_{22}\right) \left( \Ff %
_{2}\right) _{i} \\
&=\left( \overline{\Ff _{1}}^{i}\otimes \overline{\Ff _{2}}%
^{i}\otimes \overline{\Ff _{1}}_{i}\otimes \overline{\Ff _{2}}%
_{i}\right) \left( f_{1}\left( x_{1}\right) \otimes f_{2}\left( x_{3}\right)
\otimes f_{1}\left( x_{2}\right) \otimes f_{2}\left( x_{4}\right) \right)
\left( \left( \Ff _{1}\right) ^{i}\otimes \left( \Ff %
_{2}\right) ^{i}\otimes \left( \Ff _{1}\right) _{i}\otimes \left(
\Ff _{2}\right) _{i}\right) \\
&=
\left( \id \otimes \tau \otimes \id \right) \left( \overline{%
\Ff _{1}}^{i}\otimes \overline{\Ff _{1}}_{i}\otimes \overline{%
\Ff _{2}}^{i}\otimes \overline{\Ff _{2}}_{i}\right) \\
&\hspace{.5cm}\left( f_{1}\left( x_{1}\right) \otimes f_{2}\left( \Ss ^{i}x_{2}%
\overline{\Ss }^{i}\right) \otimes f_{1}\left( \Ss _{i}x_{3}%
\overline{\Ss }_{i}\right) \otimes f_{2}\left( x_{4}\right) \right)
\\
&\hspace{.5cm}\left( \id \otimes \tau \otimes \id \right) \left( \left(
\Ff _{1}\right) ^{i}\otimes \left( \Ff _{1}\right) _{i}\otimes
\left( \Ff _{2}\right) ^{i}\otimes \left( \Ff _{2}\right)
_{i}\right)\\
&=
\left( \id \otimes \tau \otimes \id \right) \left( \Ff %
_{1}^{-1}\otimes \Ff _{2}^{-1}\right) \\
&\hspace{.5cm}\big(f_{1}\left( x_{1}\right) \otimes f_{2}\left( \Ss ^{i}\right)
f_{2}\left( x_{2}\right) f_{2}\left( \overline{\Ss }^{i}\right)
\otimes f_{1}\left( \Ss _{i}\right) f_{1}\left( x_{3}\right)
f_{1}\left( \overline{\Ss }_{i}\right) \otimes f_{2}\left(
x_{4}\right)\big) \\
&\hspace{.5cm} \left( \id \otimes \tau \otimes \id \right) \left( \Ff %
_{1}\otimes \Ff _{2}\right)\\
&=
\big(\left( \id \otimes \tau \otimes \id \right) \left( \Ff %
_{1}^{-1}\otimes \Ff _{2}^{-1}\right)\big) \left( 1\otimes f_{2}\left(
\Ss ^{i}\right) \otimes f_{1}\left( \Ss _{i}\right) \otimes
1\right) \\
&\hspace{.5cm}\left( f_{1}\left( x_{1}\right) \otimes f_{2}\left( x_{2}\right) \otimes
f_{1}\left( x_{3}\right) \otimes f_{2}\left( x_{4}\right) \right) \\
&\hspace{.5cm}\left( 1\otimes f_{2}\left( \overline{\Ss }^{i}\right) \otimes
f_{1}\left( \overline{\Ss }_{i}\right) \otimes 1\right) \big(\left(
\id \otimes \tau \otimes \id \right) \left( \Ff %
_{1}\otimes \Ff _{2}\right)\big)\\
&=
\big(\left( \id \otimes \tau \otimes \id \right) \left( \Ff %
_{1}^{-1}\otimes \Ff _{2}^{-1}\right)\big) \left( 1\otimes f_{2}\left(
\Ss ^{i}\right) \otimes f_{1}\left( \Ss _{i}\right) \otimes
1\right)\left( f\left( x_{1}\right) \otimes f\left( x_{2}\right) \right) \\
&\hspace{.5cm}\left( 1\otimes f_{2}\left( \overline{\Ss }^{i}\right) \otimes
f_{1}\left( \overline{\Ss }_{i}\right) \otimes 1\right) \big(\left(
\id \otimes \tau \otimes \id \right) \left( \Ff %
_{1}\otimes \Ff _{2}\right)\big)\\
&=
\big(\left( \id \otimes \tau \otimes \id \right) \left( \Ff %
_{1}^{-1}\otimes \Ff _{2}^{-1}\right)\big) \left( 1\otimes \left(
f_{2}\otimes f_{1}\right) \left( \Ss \right) \otimes 1\right) \left( f\otimes f\right) \Delta \left( x\right) \\
&\hspace{.5cm}  \left( 1\otimes \left( f_{2}\otimes f_{1}\right) \left( \Ss %
^{-1}\right) \otimes 1\right) \big(\left( \id \otimes \tau \otimes \id\right) \left( \Ff _{1}\otimes \Ff _{2}\right)\big)\\
&=\Ff ^{-1}\left( f\otimes f\right) \Delta \left( x\right)\Ff 
\end{align*}
so $\Delta _{H\otimes H^{\prime }}f\left( x\right) =\Ff %
^{-1}\left( f\otimes f\right) \Delta \left( x\right) \Ff $ i.e.\ $%
\left( \Delta _{H\otimes H^{\prime }}\right) _{\Ff }f\left( x\right)
=\left( f\otimes f\right) \Delta \left( x\right) $.

Here is the first place where we will use that $\left( H,\Ss \right) $
is triangular and not just quasitriangular. We compute
\begin{eqnarray*}
\left( f\otimes f\right) \left( \Ss \right) &=&\left( f_{1}\otimes
f_{2}\otimes f_{1}\otimes f_{2}\right) \left( \Delta _{H}\otimes \Delta
_{H}\right) \left( \Ss \right) \\
&\overset{\eqref{eq:DDR}}=&
f_1(\Ss^{i}\Ss
^{k})\otimes f_2(\Ss^{j}\Ss^{h})\otimes f_1(\Ss_{k}\Ss_{h})\otimes f_2(\Ss_{i}\Ss_{j})\\
&=&f_{1}\left( \Ss ^{i}\right) \underleftrightarrow{f_{1}\left(
\Ss ^{k}\right) }\otimes \underbrace{f_{2}\left( \Ss %
^{j}\right) }f_{2}\left( \Ss ^{h}\right) \otimes \underleftrightarrow{%
f_{1}\left( \Ss _{k}\right) }f_{1}\left( \Ss _{h}\right)
\otimes f_{2}\left( \Ss _{i}\right) \underbrace{f_{2}\left( \Ss_{j}\right) } \\
&\overset{\eqref{eq:deformedR}}=&f_{1}\left( \Ss ^{i}\right) \left( \Ff _{1}\right)
_{i}\left( \Rr _{1}\right) ^{k}\overline{\left( \Ff %
_{1}\right) }^{i}\otimes \left( \Ff _{2}\right) _{j}\left( \Rr %
_{2}\right) ^{k}\overline{\left( \Ff _{2}\right) }^{j}f_{2}\left(
\Ss ^{h}\right) \\&\otimes& \left( \Ff _{1}\right) ^{i}\left(
\Rr _{1}\right) _{k}\overline{\left( \Ff _{1}\right) }%
_{i}f_{1}\left( \Ss _{h}\right) \otimes f_{2}\left( \Ss %
_{i}\right) \left( \Ff _{2}\right) ^{j}\left( \Rr _{2}\right)
_{k}\overline{\left( \Ff _{2}\right) }_{j} \\
&=&
\left( f_{1}\left( \Ss ^{i}\right) \otimes 1_{H_{2}}\otimes
1_{H_{1}}\otimes f_{2}\left( \Ss _{i}\right) \right) \left( \left(
\Ff _{1}\right) _{i}\otimes \left( \Ff _{2}\right) _{j}\otimes
\left( \Ff _{1}\right) ^{i}\otimes \left( \Ff _{2}\right)
^{j}\right) \\
&&\left( \left( \Rr _{1}\right) ^{k}\otimes \left( \Rr %
_{2}\right) ^{k}\otimes \left( \Rr _{1}\right) _{k}\otimes \left(
\Rr _{2}\right) _{k}\right) \\
&&\left( \overline{\left( \Ff \right) _{1}}^{i}\otimes \overline{\left(
\Ff _{2}\right) }^{j}\otimes \overline{\left( \Ff _{1}\right) }%
_{i}\otimes \overline{\left( \Ff _{2}\right) }_{j}\right) \left(
1_{H_{1}}\otimes f_{2}\left( \Ss ^{h}\right) \otimes f_{1}\left(
\Ss _{h}\right) \otimes 1_{H_{2}}\right)\\
&\overset{(*)}{=}&
\left( 1_{H_{1}}\otimes f_{2}\left( \Ss _{i}\right) \otimes
f_{1}\left( \Ss ^{i}\right) \otimes 1_{H_{2}}\right)^{\mathrm{op}}
\left( \left( \Ff _{1}\right) ^{i}\otimes \left( \Ff %
_{2}\right) ^{j}\otimes \left( \Ff _{1}\right) _{i}\otimes \left(
\Ff _{2}\right) _{j}\right) ^{\mathrm{op}} 
\\
&&\left( \left( \Rr _{1}\right) ^{k}\otimes \left( \Rr %
_{2}\right) ^{k}\otimes \left( \Rr _{1}\right) _{k}\otimes \left(
\Rr _{2}\right) _{k}\right) \\
&&\left( \overline{\left( \Ff \right) _{1}}^{i}\otimes \overline{\left(
\Ff _{2}\right) }^{j}\otimes \overline{\left( \Ff _{1}\right) }%
_{i}\otimes \overline{\left( \Ff _{2}\right) }_{j}\right) \left(
1_{H_{1}}\otimes f_{2}\left( \Ss ^{h}\right) \otimes f_{1}\left(
\Ss _{h}\right) \otimes 1_{H_{2}}\right)\\
&=&
\left( 1_{H_{1}}\otimes \left( f_{2}\otimes f_{1}\right) \left( \Ss ^{%
\mathrm{op}}\right) \otimes 1_{H_{2}}\right) ^{\mathrm{op}}\big(\left( \id \otimes \tau \otimes \id \right) \left( \Ff _{1}\otimes
\Ff _{2}\right) ^{\mathrm{op}}\big) \\
&&\left( \id _{H_{1}}\otimes \tau \otimes \id _{H_{2}}\right)
\left( \Rr _{1}\otimes \Rr _{2}\right) \\
 &&\big(\left( \id \otimes \tau \otimes \id \right) \left(
\Ff _{1}^{-1}\otimes \Ff _{2}^{-1}\right)\big) \left(
1_{H_{1}}\otimes \left( f_{2}\otimes f_{1}\right) \left( \Ss \right)
\otimes 1_{H_{2}}\right)\\
&\overset{\eqref{tr}}{=}&
\left( 1_{H_{1}}\otimes \left( f_{2}\otimes f_{1}\right) \left( \Ss %
^{-1}\right) \otimes 1_{H_{2}}\right) ^{\mathrm{op}}\big(\left( \id %
\otimes \tau \otimes \id \right) \left( \Ff _{1}\otimes
\Ff _{2}\right) ^{\mathrm{op}}\big) \\
&&\left( \id _{H_{1}}\otimes \tau \otimes \id _{H_{2}}\right)
\left( \Rr _{1}\otimes \Rr _{2}\right) \\
&&\big(\left( \id \otimes \tau \otimes \id \right) \left(
\Ff _{1}^{-1}\otimes \Ff _{2}^{-1}\right)\big) \left(
1_{H_{1}}\otimes \left( f_{2}\otimes f_{1}\right) \left( \Ss \right)
\otimes 1_{H_{2}}\right)\\
&=&
\left( \left( 1_{H_{1}}\otimes \left( f_{2}\otimes f_{1}\right) \left(
\Ss ^{-1}\right) \otimes 1_{H_{2}}\right) \left( \id \otimes
\tau \otimes \id \right) \left( \Ff _{1}\otimes \Ff %
_{2}\right) \right) ^{\mathrm{op}} \\
&&\left( \id _{H_{1}}\otimes \tau \otimes \id _{H_{2}}\right)
\left( \Rr _{1}\otimes \Rr _{2}\right) \\
&&\left( \left( 1_{H_{1}}\otimes \left( f_{2}\otimes f_{1}\right) \left(
\Ss ^{-1}\right) \otimes 1_{H_{2}}\right) \big(\left( \id \otimes
\tau \otimes \id \right) \left( \Ff _{1}\otimes \Ff %
_{2}\right)\big) \right) ^{-1}\\
&=&\Ff ^{\mathrm{op}}\big(\left( \id _{H_{1}}\otimes \tau \otimes
\id _{H_{2}}\right) \left( \Rr _{1}\otimes \Rr %
_{2}\right)\big) \Ff ^{-1},
\end{eqnarray*}
where the notation ``$\op$''
 in $(*)$ refers to the tensor product of the bialgebra $H_{1}\ot H_{2}$, hence, for instance, $\left( 1_{H_{1}}\otimes f_{2}\left( \Ss _{i}\right) \otimes
f_{1}\left( \Ss ^{i}\right) \otimes 1_{H_{2}}\right)^{\mathrm{op}}=\tau_{H_{1}\ot H_{2},H_{1}\ot H_{2}}\left( 1_{H_{1}}\otimes f_{2}\left( \Ss _{i}\right) \otimes
f_{1}\left( \Ss ^{i}\right) \otimes 1_{H_{2}}\right)$.
Thus, we have that $(f,\Ff)$ is a twisted morphism of triangular
bialgebras.
We compute
\begin{equation*}
\begin{split}
(p_{1},1\otimes 1)& \circ (f,\Ff )=\Big(p_{1}(f_{1}\otimes
f_{2})\Delta ,(p_{1}\otimes p_{1})(\Ff )(1\otimes 1)\Big) \\
& =\Big(f_{1}p_{1}\Delta ,\big((p_{1}\otimes p_{1})\left( 1\otimes \left(
f_{2}\otimes f_{1}\right) \left( \Ss ^{-1}\right) \otimes 1\right)\big)
\big((p_{1}\otimes p_{1})\left( \id \otimes \tau \otimes \id %
\right) \left( \Ff _{1}\otimes \Ff _{2}\right)\big) \Big) \\
& =\Big(f_{1},(1\otimes (\varepsilon \otimes f_{1})(\Ss ^{-1}))\big((%
\id \otimes \id \otimes \varepsilon \otimes \varepsilon )(%
\Ff _{1}\otimes \Ff _{2})\big)\Big) \\
& =\big(f_{1},(1\otimes f_{1}(1))\Ff _{1}\big) 
=\big(f_{1},\Ff _{1}\big).
\end{split}%
\end{equation*}%
Analogously, one can prove that $(p_{2},1\otimes
1)\circ (f,\Ff )=(f_{2},\Ff _{2})$. Thus, the diagram in the
statement commutes.
\end{proof}

\begin{remark}
    As we said before, condition \eqref{tr} is used to obtain $(f\ot f)(\Ss)=\Ff^{\mathrm{op}}\widetilde{\Rr}\Ff^{-1}$. We observe that the latter equality is satisfied if and only if $(f_{2}\ot f_{1})(\Ss^{\mathrm{op}})=(f_{2}\ot f_{1})(\Ss^{-1})$. Since we need this condition for arbitrary $f_1$ and $f_2$ , it should hold, in particular, when $f_1=\id=f_2$, which forces $\Ss^{\op}=\Ss^{-1}$.
\end{remark}

\begin{proposition}
\label{pro:Tw2cart2} 
Let $(f,\Ff),\left( f^{\prime },\Ff ^{\prime }\right):(H,\Ss) \to (H_1,\Rr_1)\ot (H_2,\Rr_2)$ be any morphisms in $\TwTrBialg$
together with $2$-cells
\begin{align*}
 g_{1} &:\left(
\id\otimes\varepsilon,1\otimes 1\right) \circ \left( f,\Ff \right)  \Rightarrow \left(
\id\otimes\varepsilon,1\otimes 1\right) \circ \left( f^{\prime },\Ff ^{\prime
}\right)   \\
g_{2} &:\left(
\varepsilon\otimes\id,1\otimes 1\right) \circ \left( f,\Ff \right) \Rightarrow \left(
\varepsilon\otimes\id,1\otimes 1\right) \circ \left( f^{\prime },\Ff ^{\prime
}\right).
\end{align*}
Then, there is a unique $2$-cell $g:\left( f,\Ff \right)
\Rightarrow \left( f^{\prime },\Ff ^{\prime }\right) $ such that 
\[%
\left( \id\otimes\varepsilon,1\otimes 1\right) g=g_{1}\qquad \text{and} \qquad\left( \varepsilon\otimes\id,1\otimes 1\right)
g=g_{2},\] explicitly given by $g=(\id\ot\varepsilon\ot\varepsilon\ot\id)((\Ff')^{-1})(g_1\otimes g_2)(\id\ot\varepsilon\ot\varepsilon\ot\id)(\Ff)$.     
\end{proposition}

\begin{proof}
As above, we set $p_{1}:=\id \otimes \varepsilon $ and $p_{2}:=\varepsilon \otimes
\id $. 
We are looking for a unique $g:\left( f,\Ff \right) \Rightarrow
\left( f^{\prime },\Ff ^{\prime }\right) $ such that $\left(
p_{1},1\otimes 1\right) g=g_{1}$ and $\left( p_{2},1\otimes 1\right) g=g_{2}.
$

Here $\left( p_{1},1\otimes 1\right) g=\id _{\left( p_{1},1\otimes
1\right) }\hc g=p_{1}\left( g\right) $ and similarly $\left(
p_{2},1\otimes 1\right) g=p_{2}\left( g\right) .$

A posteriori, we must have
$\left( g\otimes g\right) \Ff =\Ff ^{\prime }\Delta_{H_{1}\ot H_{2}} \left(
g\right)$ and if we apply $p_{1}\otimes p_{2}$ to this equality, we get%
\begin{equation*}
\left( p_{1}\left( g\right) \otimes p_{2}\left( g\right) \right) \left(
p_{1}\otimes p_{2}\right) \left( \Ff \right) =\left( p_{1}\otimes
p_{2}\right) \left( \Ff ^{\prime }\right) \big(\left( p_{1}\otimes
p_{2}\right) \Delta_{H_{1}\ot H_{2}} \left( g\right)\big)
\end{equation*}%
i.e.\ 
$(g_{1}\otimes g_{2})\Gg=\Gg ^{\prime }g$, where $\Gg \coloneqq (p_{1}\ot p_{2})(\Ff)$ and $\Gg ^{\prime }\coloneqq(p_{1}\ot p_{2})(\Ff')$. Therefore, we must have $g=\left( \Gg ^{\prime }\right) ^{-1}\left( g_{1}\otimes
g_{2}\right) \Gg$ which only depends on $\Ff,\Ff^{\prime },g_{1}$ and $g_{2}
$. Thus, this $g$ is the unique candidate. Let us check it really
defines a $2$-cell $g:\left( f,\Ff \right) \Rightarrow \left(
f^{\prime },\Ff ^{\prime }\right) .$ 
To this aim it is enough to notice that it can be obtained as the vertical composition of the following $2$-cells
\[\xymatrix{(f,\Ff)\ar@{=>}[r]^-{\Gg}&(\widehat{\Gg}\circ f,\Ff^\Gg)\ar@{=>}[r]^-{g_1\otimes g_2}&(\widehat{\Gg'}\circ f',(\Ff')^{\Gg'})\ar@{=>}[r]^-{(\Gg')^{-1}}&(f',\Ff')}.\]
The $2$-cells on the two sides above are clearly well-defined. It remains to discuss the one in the middle.

We have $\left( p_{1},1\otimes 1\right) \circ \left( f,\Ff %
\right) =\left( p_{1}f,\left( p_{1}\otimes p_{1}\right)
(\Ff )\right) =\left( p_{1}f,\Ff %
_{1}\right) $ and $\left( p_{2},1\otimes 1\right) \circ
\left( f,\Ff \right) =\left( p_{2}f,%
\Ff _{2}\right)$, where $\Ff_{1}=(p_{1}\ot p_{1})(\Ff)$ and $\Ff_{2}=(p_{2}\ot p_{2})(\Ff)$ and, similarly with $(f',\Ff')$. Therefore, by definition of the $2$-cells $%
g_{1}:\left(
p_{1}f,\Ff _{1}\right) \Rightarrow \left(
p_{1}f^{\prime },\Ff _{1}^{\prime }\right) $ and $g_{2}:\left( p_{2}f,\Ff %
_{2}\right)\Rightarrow \left( p_{2}f^{\prime },\Ff %
_{2}^{\prime }\right) $, we have $\varepsilon_1(g_1)=1,\varepsilon_2(g_2)=1$ and%
\begin{align*}
\left( g_{1}\otimes g_{1}\right) \Ff _{1}  =\Ff _{1}^{\prime }\Delta _{H_{1}}\left(
g_{1}\right),&\qquad g_{1}
p_1f(-)=p_{1}f^{\prime }(-)g_{1}, \\
\left( g_{2}\otimes g_{2}\right) \Ff _{2} =\Ff _{2}^{\prime }\Delta _{H_{2}}\left(
g_{2}\right) ,&\qquad g_{2}
p_{2}f(-)=p_{2}f^{\prime }(-)g_{2}.
\end{align*}
Clearly $\varepsilon_{H_1\otimes H_2}(g_1\ot g_2)=\varepsilon_1(g_1)\varepsilon_2(g_2)= 1$. By Theorem  \ref{thm:twistot} we know that
\begin{align}
\label{eqcohom0}
\Ff^\Gg& =\left( 1\otimes \left(  \Gg ^{\mathrm{op}} \Hh 
^{-1}\right) ^{-1}\otimes 1\right) \big(\left( \id _{H_{1}}\otimes
\tau \otimes \id _{H_{2}}\right) \left( \Ff _{1}\otimes \Ff _{2}\right)\big),\\
\label{eqcohom}
(\Ff')^{\Gg'}& =\left( 1\otimes \left( \left( \Gg %
^{\prime }\right) ^{\mathrm{op}}\left( \Hh ^{\prime }\right)
^{-1}\right) ^{-1}\otimes 1\right) \big(\left( \id _{H_{1}}\otimes
\tau \otimes \id _{H_{2}}\right) \left( \Ff _{1}^{\prime
}\otimes \Ff _{2}^{\prime }\right)\big),
\end{align}
where $\Hh\coloneqq (p_{2}\ot p_{1})(\Ff)$ and $\Hh'\coloneqq (p_{2}\ot p_{1})(\Ff')$. Since $f,f^{\prime }$ are morphisms of triangular bialgebras, we have%
\begin{align*}
\left( f\otimes f\right) \left( \Ss \right)
&=\Ff  ^{\mathrm{op}}\big(\left( \id %
_{H_{1}}\otimes \tau \otimes \id _{H_{2}}\right) \left( \Rr %
_{1}\otimes \Rr _{2}\right)\big) \left( \Ff \right) ^{-1},\\
\left( f^{\prime }\otimes f^{\prime }\right) \left( \Ss \right)
&=\left( \Ff ^{\prime }\right) ^{\mathrm{op}}\big(\left( \id %
_{H_{1}}\otimes \tau \otimes \id _{H_{2}}\right) \left( \Rr _{1}\otimes \Rr _{2}\right)\big) \left( \Ff ^{\prime }\right) ^{-1}.
\end{align*}%
If we apply $p_{2}\otimes p_{1}$ to the equalities above, we get%
\begin{align}
\label{eq:p2f'0}
\left( p_{2}f\otimes p_{1}f\right) \left( \Ss \right) &=\left( \left( p_{1}\otimes p_{2}\right) \left( \Ff\right) \right) ^{\mathrm{op}}\left( \left( p_{2}\otimes
p_{1}\right) \left( \Ff \right) \right) ^{-1}=
\Gg  ^{\mathrm{op}} \Hh  ^{-1}\\
\label{eq:p2f'}
\left( p_{2}f^{\prime }\otimes p_{1}f^{\prime }\right) \left( \Ss \right) &=\left( \left( p_{1}\otimes p_{2}\right) \left( \Ff^{\prime }\right) \right) ^{\mathrm{op}}\left( \left( p_{2}\otimes
p_{1}\right) \left( \Ff ^{\prime }\right) \right) ^{-1}=\left(
\Gg ^{\prime }\right) ^{\mathrm{op}}\left( \Hh ^{\prime
}\right) ^{-1}
\end{align}%
so that we obtain
\begin{align*}
(\Ff')^{\Gg'}&\Delta _{H_{1}\otimes H_{2}}(g_1\ot g_2)=\\
&\overset{\mathclap{\eqref{eqcohom},\eqref{eq:p2f'}}}{=}\hspace{0.3cm}\left( 1_{H_{1}}\otimes (p_{2}f'\ot p_{1}f')(\Ss^{-1}) \otimes 1_{H_{2}}\right) \big(\left(
\id _{H_{1}}\otimes \tau \otimes \id _{H_{2}}\right)\left( \Ff'_{1}\otimes \Ff'_{2}\right)\big)\Delta_{H_{1}\ot H_{2}}(g_{1}\ot g_{2}) \\
&=\left( 1_{H_{1}}\otimes (p_{2}f'\ot p_{1}f')(\Ss^{-1}) \otimes 1_{H_{2}}\right) \left(
\id _{H_{1}}\otimes \tau \otimes \id _{H_{2}}\right)\left( \Ff'_{1}\Delta_{H_{1}}(g_{1})\otimes \Ff'_{2}\Delta_{H_{2}}(g_{2})\right)  \\
&=\left( 1_{H_{1}}\otimes (p_{2}f'\ot p_{1}f')(\Ss^{-1}) \otimes 1_{H_{2}}\right) \left(
\id _{H_{1}}\otimes \tau \otimes \id _{H_{2}}\right) \left(
(g_{1}\ot g_{1})\Ff _{1}\otimes (g_{2}\ot g_{2})\Ff _{2}\right)  \\
&= \left( 1_{H_{1}}\otimes  \left(
p_{2}f'\otimes p_{1}f'\right) \left( \Ss ^{-1}\right)  \otimes 1_{H_{2}}\right) \left(
g_{1}\otimes g_{2}\ot g_{1}\otimes g_{2}\right) \left( \id %
_{H_{1}}\otimes \tau \otimes \id _{H_{2}}\right) \left( \Ff %
_{1}\otimes \Ff _{2}\right) \\
&=\left( 1_{H_{1}}\otimes \left(
g_{2}p_{2}f\otimes g_{1}p_{1}f\right) \left( \Ss ^{-1}\right)  \otimes 1_{H_{2}}\right) \left(
g_{1}\otimes 1\ot1\otimes g_{2}\right) \left( \id %
_{H_{1}}\otimes \tau \otimes \id _{H_{2}}\right) \left( \Ff %
_{1}\otimes \Ff _{2}\right)  \\
&=\left( g_{1}\otimes g_{2} \otimes g_{1}\otimes
g_{2}\right)  \left( 1_{H_{1}}\otimes \left( p_{2}f\otimes
p_{1}f\right) \left( \Ss ^{-1}\right) \otimes 1_{H_{2}}\right) \left(
\id _{H_{1}}\otimes \tau \otimes \id _{H_{2}}\right) \left(
\Ff _{1}\otimes \Ff _{2}\right) \\
&\overset{\mathclap{\eqref{eqcohom0},\eqref{eq:p2f'0}}}{=}\hspace{0.3cm}\left( \left( g_{1}\otimes g_{2}\right) \otimes \left( g_{1}\otimes
g_{2}\right) \right) \Ff^\Gg.
\end{align*}
Moreover%
\begin{align*}
\widehat{\Gg'}f^{\prime }(h)(g_1\ot g_2) &=\Gg'f'(h)(\Gg')^{-1}(g_1\ot g_2)\\
&=\left( p_{1}\otimes p_{2}\right)(\Ff')\big(\left( p_{1}\otimes p_{2}\right) \Delta _{H_{1}\otimes
H_{2}}f^{\prime }(h)\big)\big((p_{1}\ot p_{2})((\Ff')^{-1})\big)(g_{1}\ot g_{2})\\
&=\big((p_{1}\ot p_{2})(\Ff'\Delta_{H_{1}\ot H_{2}}f'(h)(\Ff')^{-1})\big)(g_{1}\ot g_{2})\\
&=\big((p_{1}\ot p_{2})((f'\ot f')\Delta(h))\big)(g_{1}\ot g_{2})\\
&=\big((p_{1}f'\ot p_{2}f')\Delta(h)\big)(g_{1}\ot g_{2})\\
&=(g_{1}\ot g_{2})\big((p_{1}f\ot p_{2}f)\Delta(h)\big)\\
&=(g_{1}\ot g_{2})\big((p_{1}\ot p_{2})((f\ot f)\Delta(h))\big)\\
&=(g_{1}\ot g_{2})\big((p_{1}\ot p_{2})(\Ff\Delta_{H_1\ot H_2} f(h)\Ff^{-1})\big)
\\
&=(g_{1}\ot g_{2})\Gg f(h)\Gg^{-1}=(g_{1}\ot g_{2})\widehat{\Gg} f(h).
\end{align*}%
Thus, we get the desired $2$-cell $g_1\ot g_2:\left( \widehat{\Gg}\circ f,\Ff ^\Gg\right) \Rightarrow \left(
 \widehat{\Gg'}\circ f^{\prime },(\Ff ^{\prime })^{\Gg'}\right) .$ Finally, we have
\begin{equation*}
\left( p_{1},1\otimes 1\right) g=p_{1}\left( g\right) =p_{1}\left( \left(
\Gg ^{\prime }\right) ^{-1}\left( g_{1}\otimes g_{2}\right)\Gg \right)
=p_{1}\left( \Gg ^{\prime }\right) ^{-1}g_{1}\varepsilon \left(
g_{2}\right) p_1(\Gg)=g_{1}
\end{equation*}%
as $p_{1}\left( \Gg \right) =p_{1}\left( p_{1}\otimes
p_{2}\right) \left( \Ff \right) =\left( \id \otimes
\varepsilon \otimes \varepsilon \otimes \varepsilon \right) \left( \Ff\right) =1$ and similarly $p_{1}\left( \Gg ^{\prime }\right) =1.$ Analogously,
\begin{equation*}
\left( p_{2},1\otimes 1\right) g=p_{2}\left( g\right) =p_{2}\left( \left(
\Gg ^{\prime }\right) ^{-1}\left( g_{1}\otimes g_{2}\right)\Gg \right)
=p_{2}\left( \Gg ^{\prime }\right) ^{-1}\varepsilon \left(
g_{1}\right) g_{2}p_2(\Gg)=g_{2}
\end{equation*}%
as  $p_{2}\left( \Gg \right) =p_{2}\left( p_{1}\otimes
p_{2}\right) \left( \Ff \right) =\left( \varepsilon \otimes
\varepsilon \otimes \varepsilon \otimes \id \right) \left( \Ff %
\right) =1$ and $p_{2}\left( \Gg ^{\prime }\right) =1.$    
\end{proof}

Putting together Proposition \ref{pro:Tw2cart1} and Proposition \ref{pro:Tw2cart2}, we obtain the following result.

\begin{theorem}
\label{thm:Tw2cart}
The $2$-category $\TwTrBialg$ has binary products that, for objects $(H_{1},\Rr_{1})$ and $(H_{2},\Rr_{2})$ in $\TwTrBialg$, is given by
\begin{equation}
\label{def:binaryproduct}
\big((H_{1},\Rr_{1})\ot(H_{2},\Rr_{2}),(\id \otimes\varepsilon,1\otimes1),(\varepsilon\otimes\id ,1\otimes1)\big).
\end{equation}
\end{theorem}

\begin{remark}
\label{rmk:tensorproductmorphism}
    Given morphisms $(f_{1},\Ff_{1})
    $ and $(f_{2},\Ff_{2})
    $ in $\TwTrBialg$, as in the diagram \begin{equation*}
\xymatrixcolsep{2cm}\xymatrixrowsep{1.3cm}\xymatrix{(H,\Rr)\ar[d]_{(f_1,\Ff_1)}&(H,\Rr)\ot(H',\Rr')\ar[r]^{(\varepsilon\otimes\id,1\otimes 1)}\ar[l]_{(\id\otimes\varepsilon,1\otimes
1)}
\ar@{.>}[d]|{(f_1,\Ff_1)\ot(f_2,\Ff_2)}&(H',\Rr')\ar[d]^{(f_2,\Ff_2)}\\ (H_1,\Rr_1)&
(H_1,\Rr_1)\otimes (H_2,\Rr_2)\ar[l]_{(\id\otimes\varepsilon,1\otimes
1)}\ar[r]^{(\varepsilon\otimes\id,1\otimes 1)}& (H_2,\Rr_2)}
\end{equation*}
we calculate the compositions
\begin{align*}
(f_{1},\Ff_{1})(\id \ot\varepsilon,1\ot1)&=(f_{1}(\id \ot\varepsilon),\Ff_{1})
,\qquad (f_{2},\Ff_{2})(\varepsilon\ot\id ,1\ot1)=(f_{2}(\varepsilon\ot\id ),\Ff_{2})
.    \end{align*} 
As usual, we refer to their diagonal morphism in $\TwTrBialg$ as \emph{the product of the two morphisms} $(f_1,\Ff_1)$ and $(f_2,\Ff_2)$, denoted by $(f_1,\Ff_1)\ot(f_2,\Ff_2)$. 
By Proposition \ref{pro:Tw2cart1}, it is the pair  
\((f,\Ff)\), where 
\begin{align*}
f&=(f_{1}(\id \ot\varepsilon)\ot f_{2}(\varepsilon\ot\id ))\Delta_{H\ot H'}=f_{1}\ot f_{2} \\
\Ff&=\big(1\ot(f_{2}(\varepsilon\ot\id )\ot f_{1}(\id \ot\varepsilon))(\id \ot\tau\ot\id )(\Rr^{-1}\ot\Rr'^{-1})\ot1\big)(\id \ot\tau\ot\id )(\Ff_{1}\ot\Ff_{2})\\&=(\id \ot\tau\ot\id )(\Ff_{1}\ot\Ff_{2}).
\end{align*}
Summing up, we have the formula 
\begin{equation}
\label{eq:tns1cell}
(f_1,\Ff_1)\ot(f_2,\Ff_2)=(f_{1}\ot f_{2},(\id \ot\tau\ot\id )(\Ff_{1}\ot\Ff_{2})).  \end{equation} 
\end{remark}

We now look for a terminal object in $\TwTrBialg$. We first recall the definition in the $2$-categorical context, see again \cite[\S 2.2]{CKWW2}. 

\begin{definition}
A \emph{terminal object} in a $2$-category $\Cc$ is a 0-cell $\mathsf{1}$ such that the functor 
\[\Cc(X,\mathsf{1})\to* \] into the terminal category is an equivalence of categories. 
Explicitly, this means that for any 0-cell $X$ there is a 1-cell $X\rightarrow \mathsf{1}$ and for any two 1-cells $f,g:X\rightarrow \mathsf{1}$ there is a unique $2$-cell $f\Rightarrow g$. Note that this $2$-cell is necessarily  an isomorphism. 
\begin{invisible}
There are unique $\alpha:f\Rightarrow g$ and $\beta:g\Rightarrow f$ and also their compositions are unique and have than to be the identities. 
\end{invisible}
\end{definition}

Note that in $\TwTrBialg$ the functor above is even bijective on objects, whence a category isomorphism, see \cite[Definition 1.5.1]{BorI94}.

\begin{lemma}
\label{lem:termobj}
The triangular bialgebra $(\Bbbk,1\otimes 1)$ is a terminal object in the $2$-category $\TwTrBialg$.
\end{lemma}

\begin{proof}
Let $(H,\Rr )$ be a triangular bialgebra. Since $\Ff =1\otimes1$ is the unique twist (and triangular structure) on $\Bbbk$ and $\varepsilon:H\to\Bbbk$ is the unique morphism in $\Bialg$ from $H$ to $\Bbbk$, we get that $(\varepsilon,1\ot1):(H,\Rr)\to(\Bbbk,1\ot1)$ is the unique morphism in $\TwTrBialg$ from $(H,\Rr)$ to $(\Bbbk,1\ot1)$. Given a $2$%
-cell $k:\left( \varepsilon ,1\otimes 1\right) \Rightarrow \left(
\varepsilon ,1\otimes 1\right) ,$ we have that $(k\otimes k)(1\ot 1)=(1\otimes 1)\Delta_{\Bbbk} (k)$, so that $k^2=k$. Since $k\neq 0$  (cf.\ Remark \ref{rmk:gaugeps}), we get $k=1$, so that the corresponding $2$-cell is necessarily the identity $2$-cell. Therefore, $(\Bbbk,1\ot1)$ is a terminal object in the 2-category $\TwTrBialg$.
\end{proof}

In view of \cite[Proposition 2.4]{CKWW2}, by Theorem \ref{thm:Tw2cart} and Lemma \ref{lem:termobj}, we get the following result.

\begin{corollary}
The $2$-category $\TwTrBialg$ has finite products.
\end{corollary}

Moreover, in view of \cite[Theorem 2.15]{CKWW2}, we also obtain the following result. Here, by a symmetric monoidal $2$-category we mean $2$-category that is symmetric monoidal as a bicategory.

\begin{corollary}
$\TwTrBialg$ is a symmetric monoidal 2-category with the binary product, defined on objects as in \eqref{def:binaryproduct} and on 1-cells as in \eqref{eq:tns1cell}, as its tensor product and the terminal object $(\Bbbk,1\ot1)$ as its unit object.
\end{corollary}

We end this subsection by proving a result that will be useful in the following.

\begin{lemma}
\label{lem:isoTwTr}
A $1$-cell $\left( f,\Ff \right) :\left( H,\Rr \right) \rightarrow \left(
H^{\prime },\Rr ^{\prime }\right) $  in $\TwTrBialg$ is invertible if, and
only if, $f:\left( H,\Rr \right) \rightarrow \left( H_{\Ff %
}^{\prime },\Rr _{\Ff }^{\prime }\right) $ is an isomorphism in $\TrBialg$.

Thus, $(H,\Rr)$ and $(H',\Rr')$ are isomorphic in $\TwTrBialg$ if, and only if, they are twist equivalent up to isomorphism.
\end{lemma}

\begin{proof}
Let $\left( f,\Ff \right) :\left( H,\Rr \right) \rightarrow
\left( H^{\prime },\Rr ^{\prime }\right) $ be an isomorphism in $\TwTrBialg$. Then
there is $\left( g,\Gg \right) :\left( H^{\prime },\Rr %
^{\prime }\right) \rightarrow \left( H,\Rr \right) $ in $\TwTrBialg$ such that $%
\left( f,\Ff \right) \circ \left( g,\Gg \right) =\left(
\id _{H^{\prime }},1_{H^{\prime }}\otimes 1_{H^{\prime }}\right) $
and $\left( g,\Gg \right) \circ \left( f,\Ff \right) =\left(
\id _{H},1_{H}\otimes 1_{H}\right) .$ In particular, $f g=\mathrm{%
Id}_{H^{\prime }}$ and $g f=\id _{H}$ so that $f:\left( H,%
\Rr \right) \rightarrow \left( H_{\Ff }^{\prime },\Rr _{%
\Ff }^{\prime }\right) $ is an isomorphism in $\TrBialg$.

Conversely, if $f:\left( H,\Rr \right) \rightarrow \left( H_{\Ff}^{\prime },\Rr _{\Ff }^{\prime }\right) $ is an isomorphism in $\TrBialg$, we can set $g:=f^{-1}$ and $\Gg :=\left(
f^{-1}\otimes f^{-1}\right) \left( \Ff ^{-1}\right) \in H\otimes H.$
We check that $\Gg $ is a twist on $H.$ Since $%
f:H\rightarrow H^{\prime }$ is an algebra map ($H^{\prime }$ and $H_{%
\Ff }^{\prime }$ have the same underlying algebra), we get
\begin{align*}
\left( f\otimes f\otimes f\right) \left( \left( \Gg \otimes
1\right) \left( \Delta \otimes \id \right) \left( \Gg \right)
\right)  
&=\left( \left( f\otimes f\right) \left( \Gg \right) \otimes
1\right) \left( \left( f\otimes f\right) \Delta \otimes f\right) \left(
\Gg \right)  \\
&=\left( \Ff ^{-1}\otimes 1\right) \left( \Delta _{\Ff %
}^{\prime }f\otimes f\right) \left( \Gg \right)  \\
&=\left( \Delta ^{\prime }f\otimes f\right) \left( \Gg \right)
\left( \Ff ^{-1}\otimes 1\right)  \\
&=\left( \Delta ^{\prime }\otimes \id \right) \left( \Ff %
^{-1}\right) \left( \Ff ^{-1}\otimes 1\right)  \\
&=\left( \id \otimes \Delta ^{\prime }\right) \left( \Ff %
^{-1}\right) \left( 1\otimes \Ff ^{-1}\right)  \\
&=\left( f\otimes \Delta ^{\prime }f\right) \left( \Gg \right)
\left( 1\otimes \Ff ^{-1}\right)  \\
&=\left( 1\otimes \Ff ^{-1}\right) \left( f\otimes \Delta _{\Ff}^{\prime }f\right) \left( \Gg \right)  \\
&=\left( 1\otimes \left( f\otimes f\right) \left( \Gg \right)
\right) \left( f\otimes \left( f\otimes f\right) \Delta \right) \left(
\Gg \right)  \\
&=\left( f\otimes f\otimes f\right) \left( \left( 1\otimes \Gg %
\right) \left( \id \otimes \Delta \right) \left( \Gg \right)
\right)
\end{align*}%
so that $\left( \Gg \otimes 1\right) \left( \Delta \otimes \id %
\right) \left( \Gg \right) =\left( 1\otimes \Gg \right) \left(
\id \otimes \Delta \right) \left( \Gg \right) .$ Moreover%
\begin{equation*}
\left( \varepsilon \otimes \id \right) \left( \Gg \right)
=\left( \varepsilon \otimes \id \right) \left( f^{-1}\otimes
f^{-1}\right) \left( \Ff ^{-1}\right) =\left( \varepsilon \otimes
f^{-1}\right) \left( \Ff ^{-1}\right) =1
\end{equation*}%
and similarly $\left( \id \otimes \varepsilon \right) \left( \Gg\right) =1.$ Now, we verify that $g:H^{\prime }\rightarrow H_{\Gg }$
is a coalgebra map, whence a bialgebra map. We compute%
\begin{eqnarray*}
\left( f\otimes f\right) \Delta _{\Gg }\left( h\right)  &=&\left(
f\otimes f\right) \left( \Gg \Delta \left( h\right) \Gg %
^{-1}\right)  \\
&=&\left( f\otimes f\right) \left( \Gg \right) \left( f\otimes
f\right) \Delta \left( h\right) \left( f\otimes f\right) \left( \Gg %
^{-1}\right)  \\
&=&\Ff ^{-1}\left( f\otimes f\right) \Delta \left( h\right) \Ff=\Delta ^{\prime }f\left( h\right) 
\end{eqnarray*}%
so that $\left( f\otimes f\right) \Delta _{\Gg }=\Delta ^{\prime }f$
and hence $\Delta _{\Gg }g=\left( g\otimes g\right) \Delta ^{\prime }.
$ Moreover
\begin{eqnarray*}
\left( g\otimes g\right) \left( \Rr ^{\prime }\right)  &=&\left(
g\otimes g\right) \left( \left( \Ff ^{\mathrm{op}}\right) ^{-1}%
\Rr _{\Ff }^{\prime }\Ff \right)  \\
&=&\left( \left( g\otimes g\right) \left( \Ff ^{-1}\right) \right) ^{%
\mathrm{op}}\left( g\otimes g\right) \left( \Rr _{\Ff %
}^{\prime }\right) \left( g\otimes g\right) \left( \Ff \right)  \\
&=&\Gg ^{\mathrm{op}}\left( g\otimes g\right) \left( f\otimes
f\right) \left( \Rr \right) \Gg ^{-1}=\Gg ^{\mathrm{op}}%
\mathcal{RG}^{-1}=\Rr _{\Gg }.
\end{eqnarray*}%
Thus, $g:\left( H^{\prime },\Rr ^{\prime }\right) \rightarrow \left(
H_{\Gg },\Rr _{\Gg }\right) $ is a morphism in $\TrBialg$, hence $\left( g,\Gg \right)
:\left( H^{\prime },\Rr ^{\prime }\right) \rightarrow \left( H,%
\Rr \right)$ is a morphism in $\TwTrBialg$. By definition of composition, we have
\begin{eqnarray*}
\left( g,\Gg \right) \circ \left( f,\Ff \right)  &=&\left(
gf,\left( g\otimes g\right) \left( \Ff \right) \Gg \right)
=\left( \id _{H},1_{H}\otimes 1_{H}\right) , \\
\left( f,\Ff \right) \circ \left( g,\Gg \right)  &=&\left(
fg,\left( f\otimes f\right) \left( \Gg \right) \Ff \right)
=\left( \id _{H^{\prime }},1_{H^{\prime }}\otimes 1_{H^{\prime
}}\right)
\end{eqnarray*}%
and hence $\left( f,\Ff \right) $ is an isomorphism in $\TwTrBialg$.
\end{proof}

\subsection{Comparison with cocommutative bialgebras}
Here we are going to see how cocommutative bialgebras sit inside twisted triangular bialgebras.
Given a cocommutative bialgebra $H$, we know that $(H,1\otimes1)$ is an object in $\TwTrBialg$. If $H'$ is another cocommutative bialgebra, then a twisted morphisms of triangular bialgebras $(f,1\ot 1):(H,1\otimes1)\to (H',1\otimes1)$ is just a bialgebra map $f:H\to H'.$ A gauge transformation $a:(f,1\ot 1)\Rightarrow (f',1\ot 1)$ is just a grouplike element $a\in G(H')$ such that \eqref{def:gauge2} holds true.

In light of the above discussion, we denote by $\mathsf{Bialg}^{2}_{\mathsf{cc}}$ the $2$-category where: 
\begin{itemize}
    \item[i)] $0$-cells are cocommutative bialgebras,
    \item[ii)] $1$-cells are bialgebra maps $f:H\to H'$, 
    \item[iii)] $2$-cells $a:f\Rightarrow f'$ are grouplike elements $a\in H'$ such that \eqref{def:gauge2} holds true. 
\end{itemize}
 The vertical and horizontal compositions are defined as follows:
 \[\xymatrix@C=2cm{
H\ruppertwocell^{f}{b}
\rlowertwocell_{f''}{a}
\ar[r]_(.35){f'} & H'}
 =
\xymatrix{H\rtwocell^{f}_{f''}{ab}&H'},\qquad
\xymatrix{H\rtwocell^{f}_{f'}{b} & H' \rtwocell^{g}_{g'}{a}& H''}
=  \xymatrix@C=1cm{H \rrtwocell^{gf}_{g'f'}{\mathrlap{ag(b)}}&  & H''}.
\] 
\begin{invisible}
The vertical composition of $2$-cells is clearly associative and unital. \\
If $a:\alpha\Rightarrow \alpha'$, $b:\beta\Rightarrow \beta'$ and $c:\gamma\Rightarrow \gamma'$, then 
\[(a\hc b)\hc c=(a\alpha(b))\hc c=a\alpha(b)\alpha\beta(c)
=a\alpha(b\beta(c))
=a\alpha(b\hc c)=a\hc(b\hc c).\] 
Moreover $1\hc a=a\alpha(1)=a$ and $a\hc 1=1\id(a)=a$ so that the horizontal composition is associative and unital. Note that $\id_g\hc \id f=1\hc 1=1=\id_{gf}.$ Moreover, for $2$-cells as in the 
\[
\xymatrix@C=2cm{
H\ruppertwocell^{g}{b'}
\rlowertwocell_{g''}{a'}
\ar[r]_(.35){g'} & H'\ruppertwocell^{f}{b}
\rlowertwocell_{f''}{a}
\ar[r]_(.35){f'}&H''}
\]
we have the interchange law
$(a\vc b)\hc (a'\vc b')
=(ab)\hc (a'b')
=abf(a'b')
=abf(a')f(b')
=af'(a')bf(b')
=(af'(a'))\vc (bf(b'))
=(a\hc a')\vc (b\hc b')$.
Thus, in view of \cite[Proposition 2.3.4]{Johnson-Yau-book}, $\Bialgcc$ is a $2$-category.
\end{invisible}

Since, the assignments 
\[U:\mathsf{Bialg}^{2}_{\mathsf{cc}}\to \TwTrBialg,\; H\mapsto (H,1\otimes 1),\; f\mapsto (f,1\ot 1),\; a\mapsto a,\]
preserve the identity $1$-cells, the identity $2$-cells, the composition of $2$-cells, and horizontal compositions of $1$-cells and $2$-cells, then $U$ defines a $2$-functor, see also \cite[Explanation 4.1.9]{Johnson-Yau-book}. Clearly, this $2$-functor is locally fully faithful.\medskip

In some $2$-categories the product of two objects can exist without being their product in the $2$-categorical sense, see \cite[exercise 7.10.4]{BorI94}. Still, for $\mathsf{Bialg}^{2}_{\mathsf{cc}}$ we have the following result.

\begin{lemma}
The product of two objects in the $1$-category $\Bialgcc$ is also their product in the $2$-category $\mathsf{Bialg}^{2}_{\mathsf{cc}}$. The same goes for the terminal object.
\end{lemma}

\begin{proof}
We mentioned that the binary product of two objects $H_1$ and $H_2$ in the $1$-category $\Bialgcc$ is $(H_1\otimes H_2,\id\ot\varepsilon,\varepsilon\ot\id)$ and that, given morphisms $f_1:H\to H_1$ and $f_2:H\to H_2$ in $\Bialgcc$, the diagonal morphism is $\langle f_1,f_2\rangle\coloneq (f_1\ot f_2)\Delta_H:H\to H_1\ot H_2$. Now, given morphisms $f,f':H\to H_1\ot H_2$ and $2$-cells $g_1:(\id\ot\varepsilon)f\Rightarrow (\id\ot\varepsilon) f'$ and $g_2:(\varepsilon\ot\id)f\Rightarrow (\varepsilon\ot\id)f'$ in the $2$-category $\mathsf{Bialg}^{2}_{\mathsf{cc}}$, we have that $g_1:(\id\ot\varepsilon,1\ot1)(f,1\ot1)\Rightarrow (\id\ot\varepsilon,1\ot1) (f',1\ot1)$ and $g_2:(\varepsilon\ot\id,1\ot1)(f,1\ot1)\Rightarrow (\varepsilon\ot\id,1\ot1)(f',1\ot1)$ are also $2$-cells in $\TwTrBialg$. Thus, by Proposition \ref{pro:Tw2cart2}, we get a unique $2$-cell $g:(f,1\ot1)\Rightarrow (f',1\ot1)$ in $\TwTrBialg$ such that $(\id\ot\varepsilon,1\ot1)g=g_1$ and $(\varepsilon\ot\id,1\ot1)g=g_2$, namely $g=g_1\ot g_2$. Then $g:f\Rightarrow f'$ is the unique $2$-cell in $\mathsf{Bialg}^{2}_{\mathsf{cc}}$ such that $(\id\ot\varepsilon)g=g_1$ and $(\varepsilon\ot\id)g=g_2$.

The terminal object in the $1$-category $\Bialgcc$ is given by the base field $\Bbbk$. Now, given two morphisms $f,f':H\to \Bbbk$ in $\Bialgcc$, we  have that $f=f'=\varepsilon_H$. Thus, a $2$-cell $k:f\Rightarrow f'$ is a grouplike element $k$ in $\Bbbk$ and hence $k=1$ is the identity $2$-cell.    
\end{proof}

We observe that $U$ preserves binary products and the terminal object. Indeed, given cocommutative bialgebras $H_1$ and $H_2$, by Theorem \ref{thm:Tw2cart} the binary product object of $U(H_{1})$ and $U(H_{2})$ in $\TwTrBialg$ is given by
\[U(H_1)\otimes U(H_2)
=(H_1, 1\ot 1)\ot (H_2, 1\ot 1)=(H_1\ot H_2,1\ot1\ot1\ot 1)=U(H_1\ot H_2),\]
while the projections are given by $(\id\ot\varepsilon,1\ot 1)=U(\id\ot\varepsilon)$ and $(\varepsilon\ot\id,1\ot 1)=U(\varepsilon\ot\id).$ We can compute the diagonal morphism  $\langle U(\id\ot\varepsilon),U(\varepsilon\ot\id)\rangle
=((\id\ot\varepsilon)\ot (\varepsilon\ot\id))\Delta_{H_1\ot H_2},1\ot1\ot1\ot 1)
=\id.$ Thus, the $2$-functor $U$ preserves binary products. 
\begin{invisible}
 Note that, for $\Cc\coloneqq\TwTrBialg$, we have that
 \[\Cc((H,\Ss),U(H_1\ot H_2))=\Cc((H,\Ss),U(H_1)\ot U(H_2))\cong\Cc((H,\Ss),U(H_1))\times \Cc((H,\Ss),U(H_2))\]
 is a category equivalence.
\end{invisible}
Moreover, $U(\Bbbk)=(\Bbbk,1\ot 1)$, which is the terminal object in $\TwTrBialg$, see Lemma \ref{lem:termobj}. Therefore, the terminal morphism $U(\Bbbk)\to (\Bbbk,1\ot 1)$ is necessarily the identity, so $U$ preserves the terminal object.
\begin{invisible}
 Also here, for $\Cc\coloneqq\TwTrBialg$, we have that the functor
 \[\Cc((H,\Ss),U(\Bbbk))
 =\Cc((H,\Ss),(\Bbbk,1\otimes 1))\to *\] is a category isomorphism. 
\end{invisible}

If we compose $U$ and the $2$-functor $F:\TwTrBialg\to\TwBialg$ of \eqref{def:funcF}, we obtain the $2$-functor 
\[FU:\mathsf{Bialg}^{2}_{\mathsf{cc}}\to \TwBialg,\; H\mapsto H,\; f\mapsto (f,1\ot 1),\; a\mapsto a.\]

\subsection{The classifying category of \texorpdfstring{$\TwTrBialg$}{TwTr}}

Let $\Cc$ be a (locally small) $2$-category. Then, the isomorphism classes of $1$-cells form a set, i.e.\ $\Cc$ is \emph{locally essentially small}. Therefore, following \cite[Example 2.1.27]{Johnson-Yau-book}, we can consider the \emph{classifying category} $\mathsf{Cla}(\Cc)$ of $\Cc$. Its objects are the objects in $\Cc$ while for objects $X,Y$ in $\mathsf{Cla}(\Cc)$, the set $\mathsf{Cla}(\Cc)(X,Y)$ consists of isomorphism classes $[f]$ of $1$-cells $f$ in $\Cc(X,Y)$. The identity morphism of an object $X$ in $\mathsf{Cla}(\Cc)$ is $[\id_X]$. The composition of isomorphism classes is given by $[f]\circ [g]=[f\circ g]$.

It is easy to check that, if the $2$-category $\Cc$ has binary products and the terminal object, then the $1$-category $\mathsf{Cla}(\Cc)$ has them too and that the $1$-functor \[Q:\Cc\to \mathsf{Cla}(\Cc),\;X\mapsto X,\;f\mapsto [f],\] preserves them. Therefore, this construction allows one to pass to a simpler 1-categorical setting in which finite products can be transported and studied more conveniently. We apply it to obtain a genuine 1‑category of triangular bialgebras with finite products.

\begin{invisible} Let us consider the binary product in $\Cc$.
Let $[f]:X\rightarrow A$ and $[g]:X\rightarrow B$ be morphisms in $\mathsf{Cla}(\Cc)$. Since they are represented by morphisms $f:X\rightarrow A$ and $g:X\rightarrow B$ in $\Cc$, we know that there exists a morphism $h:X\rightarrow A\times B$ and
isomorphisms $ph\cong f$ and $qh\cong g$, where $p:A\times B\to A$ and $q:A\times B\to B$. Thus, $[p]\circ[h]=[ph]=[f]$ and $[q]\circ [h]=[qh]=[g]$.
\begin{equation*}
\xymatrix{&X\ar[ld]_{[f]}\ar[rd]^{[g]}\ar@{.>}[d]|{[h]}\\ A&
A\times B\ar[l]_{[p]}\ar[r]^{[q]}& B}
\end{equation*}%
Assume now there is another morphism $[k]:X\to A\times B$ such that $[p]\circ[k]=[f]$ and $[q]\circ [k]=[g]$. Then $pk\cong f\cong ph$ and $qk\cong g\cong qh $ so that 
there are invertible $2$-cells 
$\alpha:ph\Rightarrow pk$ and $\beta:qh\Rightarrow qk$. Thus, there is a unique $\gamma:h\Rightarrow k$ such that $p\gamma=\alpha$ and $q\gamma=\beta.$ Similarly, there is a unique $\gamma':k\Rightarrow h$ such that $p\gamma'=\alpha^{-1}$ and $q\gamma'=\beta^{-1}.$ Thus 
\[
p(\gamma\vc\gamma')=\id_p\hc(\gamma\vc\gamma')=(\id_p\vc\id_p)\hc(\gamma\vc\gamma')=(\id_p\hc\gamma)\vc(\id_p\hc\gamma')=p\gamma\vc p\gamma'=\id 
\]
and hence $\gamma\circ \gamma'=\id$. Similarly $\gamma'\circ \gamma=\id$. Thus $\gamma:h\Rightarrow k$ is invertible and hence $[h]=[k]$. 

Let us now consider the terminal object. Given an object $X$ in $\mathsf{Cla}(\Cc)$, then $X$ is an object in $\Cc$ so that there is a 1-cell $f:X\rightarrow \mathsf{1}$ and hence a 1-cell $[f]:X\rightarrow \mathsf{1}$.  If there is another 1-cell $[g]:X\rightarrow \mathsf{1}$, then it is represented by the 1-cell $g:X\rightarrow \mathsf{1}$. Therefore, there is a unique necessarily invertible $2$-cell $f\Rightarrow g$ so that $[f]=[g]$.

Explicitly, this means that for any 0-cell $X$ there is a 1-cell $X\rightarrow \mathsf{1}$ and for any two 1-cells $f,g:X\rightarrow \mathsf{1}$ there is a unique $2$-cell $f\Rightarrow g$. Note that this $2$-cell is necessarily  an isomorphism.

We deal with the functor $Q$. The diagonal morphism $<Q(p),Q(q)>:Q(A\times B)\to Q(A)\times Q(B)$ is uniquely determined by the equalities $[p]\circ <Q(p),Q(q)>=Q(p)$ and $[q]\circ <Q(p),Q(q)>=Q(p)$
and hence it is necessarily $[\id_{A\times B}]$. This shows that $Q$ preserves binary products. The terminal morphism $Q(\mathsf{1})\to \mathsf{1}$ is necessarily $\id_{\mathsf{1}}$. Thus, $Q$ preserves the terminal object.
\end{invisible}

In view of Theorem \ref{thm:Tw2cart}, the classifying category 
$\mathsf{Cla}(\TwTrBialg)$ of $\TwTrBialg$ has binary products and terminal object, hence it is a cartesian monoidal category. Explicitly, its objects are triangular bialgebras and its morphisms are equivalence classes $[f,\Ff]:(H,\Rr)\to (H',\Rr')$ of morphisms $(f,\Ff):(H,\Rr)\to (H',\Rr')$ in $\TwTrBialg$. Note that $[f,\Ff]=[f',\Ff']$ if, and only if, there is an invertible gauge transformation $a:(f,\Ff)\Rightarrow (f',\Ff')$. Thus, $[f,\Ff]$ can be called the \emph{gauge class} of $(f,\Ff)$. 

Given morphisms $(f_1,\Ff_1):(H,\Ss)\to (H_1,\Rr_1)$ and $(f_2,\Ff_2):(H,\Ss)\to (H_2,\Rr_2)$ in $\TwTrBialg$, the diagonal morphism  $\langle[f_1,\Ff_1],[f_2,\Ff_2]\rangle$ in $\mathsf{Cla}(\TwTrBialg)$ is given by the gauge class $[f,\Ff]$ where $(f,\Ff)=\langle(f_1,\Ff_1),(f_2,\Ff_2)\rangle$ is the diagonal morphism in $\TwTrBialg$ given in Proposition \ref{pro:Tw2cart1}.

\section{Applications to twisted tensor product}\label{sec:applications}

As recalled in \cite{Wu-Zhou}, the categorical problem of determining whether a braided finite tensor category is braided equivalent to the Deligne tensor product of two braided finite tensor categories motivates the analogous question for Hopf algebras: whether a quasitriangular Hopf algebra can be expressed as a tensor product of two non-trivial quasitriangular Hopf algebras. Two finite-dimensional quasitriangular Hopf algebras have braided equivalent representation categories precisely when one is isomorphic to a twist of the other. Consequently, every finite-dimensional quasitriangular Hopf algebra admitting a factorizable quotient Hopf algebra decomposes as a twisted tensor product of two quasitriangular Hopf algebras. In \cite{Wu-Zhou}, the authors established sufficient conditions for such a factorization to occur for an arbitrary quasitriangular Hopf algebra.

Following \cite[Definition 3.1]{Wu-Zhou} we recall the following definition where we drop out the hypothesis of the existence of an antipode. 

\begin{definition}
\label{def:twtensprod}
Let $(H,\Rr)$, $(H_1,\Rr_1)$ and $(H_2,\Rr_2)$ be quasitriangular bialgebras. Then $(H,\Rr)$ is called a \emph{twisted tensor product} of  $(H_1,\Rr_1)$ and $(H_2,\Rr_2)$ if there is a twist $\Ff$ on the tensor product bialgebra $H_{1}\otimes H_2$ and an isomorphism of quasitriangular bialgebras  $f:(H,\Rr)\to (( H_{1}\otimes H_2)_\Ff,\widetilde{\Rr}_\Ff)$ where $\widetilde{\Rr}=(\id\ot\tau\ot\id)(\Rr_1\ot\Rr_2)$.
\end{definition}

\begin{remark}
\label{rmk:twistedasinv1cell}
    By Lemma \ref{lem:isoTwTr}, if $(H,\Rr)$, $(H_1,\Rr_1)$ and $(H_2,\Rr_2)$ are triangular bialgebras, then $(H,\Rr)$ is a twisted tensor product of $(H_{1},\Rr_{1})$ and $(H_{2},\Rr_{2})$ if there is a twist $\Ff$ on the tensor product bialgebra $H_{1}\otimes H_2$ and an invertible 1-cell $(f,\Ff):(H,\Rr)\to(H_{1}\ot H_{2},\widetilde{\Rr})$ in $\TwTrBialg$.
\end{remark}


In \cite[Theorem
4.2]{Wu-Zhou} there is a characterization of quasitriangular Hopf algebras $(H,\Rr) $ that are a twisted tensor product of two quotient Hopf algebras $H_{1}$ and
$H_{2}$ with twist $\mathcal{J}=1\otimes \left(
f_{2}S\otimes f_{1}\right) \left( \Rr \right) \otimes 1$, where $f_{1}:H\to H_{1}$ and $f_{2}:H\to H_{2}$ denote the projections. In the triangular case, we are able to provide a further characterization in terms of the binary product in $\TwTrBialg$. We prove the result for bialgebras noting that, if $H$ has an antipode $S$, then one has $ 1\otimes \left( f_{2}\otimes
f_{1}\right) \left( \Rr^{-1}\right) \otimes 1 =1\otimes \left(
f_{2}S\otimes f_{1}\right) \left( \Rr \right) \otimes 1.$ 

\begin{corollary}
\label{cor:twtnsprdweak}
The following are equivalent for a triangular bialgebra $\left( H,\Rr \right) $.
\begin{enumerate}
    \item There are triangular bialgebras $(H_{1},\Rr_1)$ and $(H_{2},\Rr_2)$ and an
invertible $1$-cell $\left( f,\Ff \right) :\left( H,\Rr \right)
\rightarrow \left( H_{1},\Rr _{1}\right) \otimes \left( H_{2},%
\Rr _{2}\right) $ in $\TwTrBialg$, where $\Ff =1\otimes \Ww^{-1}%
\otimes 1$ for a central weak $\Rr $-matrix $\Ww\in
H_{2}\otimes H_{1}$.

\item There are triangular bialgebras $(H_{1},\Rr_{1})$ and $(H_{2},\Rr_{2})$ and surjective morphisms $f_i:(H,\Rr)\to (H_i,\Rr_i)$ in $\TrBialg$, for $i=1,2$, such that the diagonal morphism $\langle(f_1,1\otimes 1),(f_2,1\otimes 1)\rangle:(H,\Rr)\to \left( H_{1},\Rr _{1}\right) \otimes \left( H_{2},%
\Rr _{2}\right) $ is an
invertible $1$-cell in $\TwTrBialg$.

\item  There are triangular bialgebras $(H_1,\Rr_1)$ and $(H_2,\Rr_2)$, and an isomorphism $f:(H,\Rr)\to (( H_{1}\otimes H_2)_\Ff,\widetilde{\Rr}_\Ff)$ in $\TrBialg$, where $\widetilde{\Rr}=(\id\ot\tau\ot\id)(\Rr_1\ot\Rr_2)$ and $\Ff =1\otimes \Ww^{-1}%
\otimes 1$ for a central weak $\Rr $-matrix $\Ww\in
H_{2}\otimes H_{1}$,
i.e.\ $(H,\Rr)$ is a twisted tensor product of  $(H_1,\Rr_1)$ and $(H_2,\Rr_2)$ in the sense of Definition \ref{def:twtensprod},  where  $\Ff =1\otimes \Ww^{-1}%
\otimes 1$.
\item 
There are bialgebras $H_{1}$ and $H_{2}$ and surjective bialgebra maps $f_{i}:H\rightarrow  H_{i} ,i=1,2$, such that $(f_1\ot f_2)\Delta:H\to (H_1\otimes H_2)_\Ff$ is a bialgebra isomorphism, where $\Ff \coloneqq  1\otimes \left( f_{2}\otimes
f_{1}\right) \left( \Rr^{-1}\right) \otimes 1 .$
\end{enumerate}
\end{corollary}

\begin{proof}
First recall that $\left( H_{1},\Rr _{1}\right) \otimes \left( H_{2},%
\Rr _{2}\right) =\left( H_{1}\otimes H_{2},\widetilde{\Rr }%
\right)$, where $\widetilde{\Rr }=\left( \id \otimes \tau
\otimes \id \right) \left( \Rr _{1}\otimes \Rr %
_{2}\right) .$

\noindent $\left( 1\right) \Rightarrow \left( 2\right) $. Define the following 1-cells in $\TwTrBialg$:
\begin{align*}
\left( f_{1},\Ff _{1}\right)  &\coloneqq\left( \id \otimes
\varepsilon ,1\otimes 1\right) \left( f,\Ff \right)
=\left( \left( \id \otimes \varepsilon \right) f,\left( \id %
\otimes \varepsilon \otimes \id \otimes \varepsilon \right) \left(
1\otimes \Ww^{-1}\otimes 1\right) \right)  =\left( \left( \id \otimes \varepsilon \right) f,1\otimes 1\right),\\
\left( f_{2},\Ff _{2}\right)  &\coloneqq\left( \varepsilon \otimes \mathrm{%
Id},1\otimes 1\right)\left( f,\Ff \right)=\left( \left( \varepsilon \otimes \id \right) f,\left( \varepsilon
\otimes \mathrm{\mathrm{\id }}\otimes \varepsilon \otimes \id %
\right) \left( 1\otimes \Ww^{-1}\otimes 1\right) \right)  =\left( \left( \varepsilon \otimes \id \right) f,1\otimes 1\right) .
\end{align*}%
In particular, $f_i:(H,\Rr)\to (H_i,\Rr_i)$ is a morphism in $\TrBialg$ for $i=1,2.$
By Lemma \ref{lem:isoTwTr}, $\left( f,\Ff \right) $ is invertible in $\TwTrBialg$ if and only if $%
f:\left( H,\Rr \right) \rightarrow \left( \left( H_{1}\otimes
H_{2}\right) _{\Ff },\widetilde{\Rr }_{\Ff }\right) $
is an isomorphism in $\TrBialg$. Therefore, $f:H\rightarrow
H_{1}\otimes H_{2}$ is bijective and hence $f_{1}=\left( \id \otimes
\varepsilon \right) f$ and $f_{2}=\left( \varepsilon \otimes \id %
\right) f$ are surjective. 

By construction, the diagonal morphism $\left\langle \left( f_{1},%
\Ff _{1}\right) ,\left( f_{2},\Ff _{2}\right) \right\rangle$ satisfies 
\begin{align*}
\left( \id \otimes
\varepsilon ,1\otimes 1\right) \circ \left\langle \left( f_{1},%
\Ff _{1}\right) ,\left( f_{2},\Ff _{2}\right) \right\rangle &=\left( f_{1},%
\Ff _{1}\right)\\
\left(\varepsilon \ot \id  ,1\otimes 1\right) \circ \left\langle \left( f_{1},%
\Ff _{1}\right) ,\left( f_{2},\Ff _{2}\right) \right\rangle &=\left( f_{2},%
\Ff _{2}\right).
\end{align*}
Hence, we can
take 
\begin{align*}
 g_{1}&=1:\left( \id \otimes
\varepsilon ,1\otimes 1\right) \circ \left\langle \left( f_{1},%
\Ff _{1}\right) ,\left( f_{2},\Ff _{2}\right) \right\rangle
\Rightarrow \left( \id \otimes \varepsilon ,1\otimes 1\right) \circ
\left( f,\Ff \right) \\
g_{2}&=1:\left(\varepsilon \ot \id  ,1\otimes 1\right) \circ \left\langle \left( f_{1},%
\Ff _{1}\right) ,\left( f_{2},\Ff _{2}\right) \right\rangle \Rightarrow \left( \varepsilon \otimes \id ,1\otimes
1\right) \circ \left( f,\Ff \right).
\end{align*} 
By the universal property of
the binary product there is a unique $2$-cell $g:\left\langle \left( f_{1},%
\Ff _{1}\right) ,\left( f_{2},\Ff _{2}\right) \right\rangle
\Rightarrow \left( f,\Ff \right) $ such that $\left( \id %
\otimes \varepsilon ,1\otimes 1\right) g=g_{1}=1$ and $\left(
\varepsilon \otimes \id ,1\otimes 1\right) g=g_{2}=1.$ By
construction,
\begin{align*}
g &=\left( \id \otimes \varepsilon \otimes \varepsilon \otimes
\id \right) \left( \Ff ^{-1}\right) \left( g_{1}\otimes
g_{2}\right)(\id \otimes \varepsilon \otimes \varepsilon \otimes
\id )((1\ot(f_{2}\ot f_{1})(\Rr^{-1})\ot1)(\id \ot\tau\ot\id )(\Ff_{1}\ot\Ff_{2}))\\&= \left( 1\otimes \left(  \varepsilon \otimes \varepsilon \right)(\Ww)\otimes 1\right) \left( 1\otimes
1\right)(1\ot(\varepsilon\ot\varepsilon)(\Rr^{-1})\ot1)(\id \ot\varepsilon\ot\varepsilon\ot\id )(\Ff_{1}\ot\Ff_{2}) =1\otimes 1.
\end{align*}%
This means that $\left\langle \left( f_{1},\Ff _{1}\right) ,\left(
f_{2},\Ff _{2}\right) \right\rangle =\left( f,\Ff \right) $. Therefore, $\langle(f_1,1\otimes 1),(f_2,1\otimes 1)\rangle$ is an invertible 1-cell in $\TwTrBialg$. The implication $(2)\Rightarrow(1)$ is obvious once observed the specific form assumed by the diagonal morphism $(f,\Ff)$ in Proposition \ref{pro:Tw2cart1} when $\Ff_1=1\otimes 1$ and $\Ff_2=1\otimes 1$.

\noindent $\left( 1\right) \Leftrightarrow \left( 3\right) $.  It follows by Lemma \ref{lem:isoTwTr}, as mentioned in Remark \ref{rmk:twistedasinv1cell}.

\noindent $\left( 2\right) \Leftrightarrow \left( 4\right) $. If $f_i:(H,\Rr)\to (H_i,\Rr_i)$ is a surjective morphism in $\TrBialg$, for $i=1,2$, then $f_i:H\to H_i$ is a surjective bialgebra map, for $i=1,2$. Conversely, if $f_i:H\to H_{i}$ is a surjective bialgebra map, for $i=1,2$, then $f_i:(H,\Rr)\to (H_i,\Rr_i)$ is a surjective morphism in $\TrBialg$ where $\Rr_i\coloneqq (f_i\otimes f_i)(\Rr)$, for $i=1,2$. 
By Proposition \ref{pro:Tw2cart1}, the diagonal morphism of $\left(
f_{1},1\otimes 1\right) $ and $\left( f_{2},1\otimes 1\right) $ is given by%
\begin{eqnarray*}
\left\langle \left( f_{1},1\otimes 1\right) ,\left( f_{2},1\otimes 1\right) \right\rangle  &=&\left( \left( f_{1}\otimes f_{2}\right) \Delta
_{H},\left( 1\otimes \left( f_{2}\otimes f_{1}\right) \left( \Rr %
^{-1}\right) \otimes 1\right) \right)  .
\end{eqnarray*}
Thus, by Lemma \ref{lem:isoTwTr}, $ \langle(f_1,1\otimes 1),(f_2,1\otimes 1)\rangle$ is invertible in $\TwTrBialg$ if, and only if, $\left( f_{1}\otimes f_{2}\right) \Delta
_{H}:(H,\Rr)\to ((H_1\ot H_2)_\Ff,\widetilde{\Rr}_\Ff)$ is an isomorphism in $\TrBialg$, where $\Ff= 1\otimes \left( f_{2}\otimes f_{1}\right) \left( \Rr %
^{-1}\right) \otimes 1$, i.e.\ $(f_{1}\ot f_{2})\Delta_{H}:H\to(H_{1}\ot H_{2})_{\Ff}$ is a bialgebra isomorphism and $((f_{1}\ot f_{2})\Delta_{H}\ot (f_{1}\ot f_{2})\Delta_{H})(\Rr)=\widetilde{\Rr}_{\Ff}$. To conclude, we observe that, given $\Rr_{i}:=(f_{i}\ot f_{i})(\Rr)$, for $i=1,2$, the latter equality is satisfied:
\begin{align*}
\widetilde{\Rr}_{\Ff}&=(1\otimes \left( f_{2}\otimes f_{1}\right) \left( \Rr %
^{-1}\right) \otimes 1)^{\mathrm{op}}(f_{1}(\Rr^{i})\ot f_{2}(\Rr^{j})\ot f_{1}(\Rr_{i})\ot f_{2}(\Rr_{j}))(1\otimes \left( f_{2}\otimes f_{1}\right) \left( \Rr %
^{-1}\right) \otimes 1)^{-1}\\&=(f_{1}(\overline{\Rr}_{s})\ot1\ot1\ot f_{2}(\overline{\Rr}^{s}))(f_{1}(\Rr^{i})\ot f_{2}(\Rr^{j})\ot f_{1}(\Rr_{i})\ot f_{2}(\Rr_{j}))(1\ot f_{2}(\Rr^{l})\ot f_{1}(\Rr_{l})\ot 1)\\
&\overset{\mathclap{\eqref{tr}}}{=}(f_1(\Rr^{s})\ot1\ot1\ot f_{2}(\Rr_{s}))(f_{1}(\Rr^{i})\ot f_{2}(\Rr^{j})\ot f_{1}(\Rr_{i})\ot f_{2}(\Rr_{j}))(1\ot f_{2}(\Rr^{l})\ot f_{1}(\Rr_{l})\ot 1)\\&=(f_{1}\ot f_{2}\ot f_{1}\ot f_{2})(\Rr^{s}\Rr^{i}\ot\Rr^{j}\Rr^{l}\ot \Rr_{i}\Rr_{l}\ot\Rr_{s}\Rr_{j})\\&
\overset{\mathclap{\eqref{eq:DDR}}}=(f_{1}\ot f_{2}\ot f_{1}\ot f_{2})(\Delta_H\ot\Delta_H)(\Rr).
\end{align*}
This computation completes the proof.
\end{proof}

Still in the triangular setting, we now provide a characterization of the twisted tensor product (Definition \ref{def:twtensprod}) for an arbitrary twist, thus not necessarily one of the form $\Ff=1\ot \Ww^{-1}\ot 1$ as in Corollary \ref{cor:twtnsprdweak}, by employing binary products in $\TwTrBialg$. To this aim, we first need Lemma \ref{lem:iso&2cell} which is based on the following remark.

\begin{remark} 
\label{rmk:partial}
Consider two morphisms $(f',\Ff'),(f,\Ff):(H',\Rr')\to (H,\Rr)$ in $\TwTrBialg$ and let $a:(f,\Ff)\Rightarrow (f',\Ff')$ be a gauge transformation, with $a\in H$. If it is invertible, then $a\in H ^\times_{\varepsilon}\coloneqq\{a\in H^\times\mid \varepsilon(a)=1\}$, see Remark \ref{rmk:gaugeps}.
As in \cite[\S4.2]{Davydov}, we note that any $a\in H^\times_\varepsilon$ defines an automorphism $\partial(a):(H,\Rr)\to (H,\Rr)$, by setting $\partial(a)\coloneqq (\hat{a},(1\otimes 1)^a)$,
where $(1\otimes 1)^a$ is given as in \eqref{def:F^h}, for $\Ff=1\otimes 1$. Hence one obtains 
the map \[\partial:H^\times_\varepsilon\to \mathrm{Aut}_{\TwTrBialg}(H,\Rr)\] which  is a group homomorphism. 
By definition of composition, one gets $\partial(a)\circ (f,\Ff)=(f',\Ff').$ 
\end{remark}

\begin{lemma}
\label{lem:iso&2cell}
   Let $a:(f,\Ff)\Rightarrow (f',\Ff')$ be an invertible $2$-cell in $\TwTrBialg$. Then, the $1$-cell $(f,\Ff)$ is invertible if, and only if, so is $(f',\Ff').$
\end{lemma}

\begin{proof}
By Remark \ref{rmk:partial}, since $a$ is invertible, then $\partial(a)\circ (f,\Ff)=(f',\Ff')$, with $\partial(a)$  invertible.
\end{proof}

We are now ready to state the aforementioned characterization.

\begin{corollary}
The following are equivalent for a triangular bialgebra $\left( H,%
\Rr \right) $.

\begin{enumerate}
\item There are triangular bialgebras $(H_{1},\Rr_{1})$ and $(H_{2},%
\Rr_{2})$, and an invertible $1$-cell $\left( f,\Ff \right) :\left( H,\Rr\right) \rightarrow \left( H_{1},\Rr _{1}\right) \otimes \left(
H_{2},\Rr _{2}\right) $ in $\TwTrBialg$.

\item There are triangular bialgebras $(H_{1},\Rr_{1})$ and $(H_{2},\Rr_{2})$, twists $\Ff_{i}$ of $H_{i}$ and surjective morphisms $f_{i}:(H,\Rr%
)\rightarrow (\left( H_{i}\right) _{\Ff _{i}},\left( \Rr_{i}\right) _{%
\Ff _{i}})$ in $\TrBialg$, for $i=1,2$, such that the diagonal morphism $\langle (f_{1},%
\Ff _{1}),(f_{2},\Ff _{2})\rangle :(H,\Rr)\rightarrow \left(
H_{1},\Rr _{1}\right) \otimes \left( H_{2},\Rr _{2}\right) $
is an invertible $1$-cell in $\TwTrBialg$.
\item There are triangular bialgebras $(H_1,\Rr_1)$ and $(H_2,\Rr_2)$, a twist $\Ff$ on $H_{1}\ot H_{2}$ and an isomorphism $f:(H,\Rr)\to (( H_{1}\otimes H_2)_\Ff,\widetilde{\Rr}_\Ff)$ in $\TrBialg$ where $\widetilde{\Rr}=(\id\ot\tau\ot\id)(\Rr_1\ot\Rr_2)$,
i.e.\ $(H,\Rr)$ is a twisted tensor product of $(H_1,\Rr_1)$ and $(H_2,\Rr_2)$ in the sense of Definition \ref{def:twtensprod}.
\item 
There are bialgebras $H_{1}$ and $H_{2}$ and surjective bialgebra maps $f_{i}:H\rightarrow \left( H_{i}\right) _{\Ff _{i}},i=1,2$, such that $(f_1\ot f_2)\Delta:H\to (H_1\otimes H_2)_\Ff$ is a bialgebra isomorphism, where $\Ff \coloneqq \left( 1\otimes \left( f_{2}\otimes
f_{1}\right) \left( \Rr^{-1}\right) \otimes 1\right) \big(\left(
\id \otimes \tau \otimes \id \right) \left( \Ff %
_{1}\otimes \Ff _{2}\right)\big) .$
\end{enumerate}
\end{corollary}

\begin{proof}
First recall that $\left( H_{1},\Rr _{1}\right) \otimes \left( H_{2},%
\Rr _{2}\right) =\left( H_{1}\otimes H_{2},\widetilde{\Rr }%
\right),$ where $\widetilde{\Rr }=\left( \id \otimes \tau
\otimes \id \right) \left( \Rr _{1}\otimes \Rr %
_{2}\right) .$

\noindent $\left( 1\right) \Rightarrow \left( 2\right) $. Define the following 1-cells in $\TwTrBialg$:
\begin{gather*}
\left( f_{1},\Ff _{1}\right) \coloneqq\left( \id \otimes
\varepsilon ,1\otimes 1\right) \circ \left( f,\Ff \right) =\left(
\left( \id \otimes \varepsilon \right) f,\left( \id \otimes
\varepsilon \otimes \id \otimes \varepsilon \right) \left( \Ff %
\right) \right) , \\
\left( f_{2},\Ff _{2}\right) \coloneqq\left( \varepsilon \otimes
\id ,1\otimes 1\right) \circ \left( f,\Ff \right) =\left(
\left( \varepsilon \otimes \id \right) f,\left( \varepsilon \otimes
\mathrm{\mathrm{\id }}\otimes \varepsilon \otimes \id \right)
\left( \Ff \right) \right) .
\end{gather*}%
In particular, $f_{i}:(H,\Rr)\rightarrow (\left( H_{i}\right) _{\Ff %
_{i}},\left( \Rr_{i}\right) _{\Ff _{i}})$ is a morphism in $\TrBialg$  for $%
i=1,2.$ By Lemma \ref{lem:isoTwTr}, $\left( f,\Ff \right) $ is an
isomorphism in $\TwTrBialg$ if and only if $f:\left( H,\Rr \right)
\rightarrow \left( \left( H_{1}\otimes H_{2}\right) _{\Ff },%
\widetilde{\Rr }_{\Ff }\right) $ is an isomorphism in $\TrBialg
$. Therefore, $f:H\rightarrow H_{1}\otimes H_{2}$ is bijective and hence $%
f_{1}=\left( \id \otimes \varepsilon \right) f$ and $f_{2}=\left(
\varepsilon \otimes \id \right) f$ are surjective.
By construction, the diagonal morphism $\left\langle \left( f_{1},%
\Ff _{1}\right) ,\left( f_{2},\Ff _{2}\right) \right\rangle$ satisfies 
\begin{align*}
\left( \id \otimes
\varepsilon ,1\otimes 1\right) \circ \left\langle \left( f_{1},%
\Ff _{1}\right) ,\left( f_{2},\Ff _{2}\right) \right\rangle &=\left( f_{1},%
\Ff _{1}\right)\\
\left(\varepsilon \ot \id  ,1\otimes 1\right) \circ \left\langle \left( f_{1},%
\Ff _{1}\right) ,\left( f_{2},\Ff _{2}\right) \right\rangle &=\left( f_{2},%
\Ff _{2}\right).
\end{align*}
Hence, we can
take 
\begin{align*}
 g_{1}&=1:\left( \id \otimes
\varepsilon ,1\otimes 1\right) \circ \left\langle \left( f_{1},%
\Ff _{1}\right) ,\left( f_{2},\Ff _{2}\right) \right\rangle
\Rightarrow \left( \id \otimes \varepsilon ,1\otimes 1\right) \circ
\left( f,\Ff \right) \\
g_{2}&=1:\left(\varepsilon \ot \id  ,1\otimes 1\right) \circ \left\langle \left( f_{1},%
\Ff _{1}\right) ,\left( f_{2},\Ff _{2}\right) \right\rangle \Rightarrow \left( \varepsilon \otimes \id ,1\otimes
1\right) \circ \left( f,\Ff \right).
\end{align*} 
The universal property of
the binary product yields a unique $2$-cell $g:\left\langle \left( f_{1},%
\Ff _{1}\right) ,\left( f_{2},\Ff _{2}\right) \right\rangle
\Rightarrow \left( f,\Ff \right) $ such that $\left( \id %
\otimes \varepsilon ,1\otimes 1\right) g=g_{1}=1$ and $\left(
\varepsilon \otimes \id ,1\otimes 1\right) g=g_{2}=1.$ By
construction, setting 
\begin{align*}
\Gg&\coloneqq (\id \ot\varepsilon\ot\varepsilon\ot\id )((1\ot(f_{2}\ot f_{1})(\Rr^{-1})\ot1)(\id \ot\tau\ot\id )(\Ff_{1}\ot\Ff_{2}))=1\ot1,\\ \Gg'&\coloneqq\left( \id \otimes
\varepsilon \otimes \varepsilon \otimes \id \right) \left( \Ff %
\right) ,    
\end{align*}
we have
$g=\mathcal{G'}^{-1}\left( g_{1}\otimes g_{2}\right)\Gg =\mathcal{G'}^{-1}$. Hence, we have the $2$-cell $\mathcal{G'}^{-1}:\left\langle \left( f_{1},%
\Ff _{1}\right) ,\left( f_{2},\Ff _{2}\right) \right\rangle
\Rightarrow \left( f,\Ff \right) ,$ which is invertible as so is the element $\Gg'$. Thus, we can apply Lemma \ref{lem:iso&2cell}, to conclude that, since $(f,\Ff)$ is an invertible $1$-cell, then so is $\left\langle \left( f_{1},%
\Ff _{1}\right) ,\left( f_{2},\Ff _{2}\right) \right\rangle.$

\noindent $\left( 2\right) \Rightarrow \left( 1\right). $ It is clear.

\noindent $\left( 1\right) \Leftrightarrow \left( 3\right) $. It follows by Lemma \ref{lem:isoTwTr}, as mentioned in Remark \ref{rmk:twistedasinv1cell}.

\noindent $\left( 2\right) \Leftrightarrow \left( 4\right) $. If $f_i:(H,\Rr)\to((H_i)_{\Ff_i},(\Rr_i)_{\Ff_i})$ is a surjective morphism in $\TrBialg$, for $i=1,2$, then $f_i:H\to (H_i)_{\Ff_i}$ is a surjective bialgebra map, for $i=1,2$. Conversely, if $f_i:H\to (H_i)_{\Ff_i}$ is a surjective bialgebra map, for $i=1,2$, then $(H_{i})_{\Ff_i}$ becomes triangular in a unique way such that $f_i$ becomes a morphism in $\TrBialg$, i.e.\ via $(f_i\otimes f_i)(\Rr)$. If we set $\Rr_i\coloneqq (f_i\otimes f_i)(\Rr)_{\Ff_i^{-1}}$, we then get $(\Rr_i)_{\Ff_i}=(f_i\otimes f_i)(\Rr)$ so that $f_i:(H,\Rr)\to((H_i)_{\Ff_i},(\Rr_i)_{\Ff_i})$ is a surjective morphism in $\TrBialg$, for $i=1,2$. By Proposition \ref{pro:Tw2cart1}, the diagonal morphism of $\left(
f_{1},\Ff_{1}\right) $ and $\left( f_{2},\Ff_{2}\right) $ is given by%
\begin{eqnarray*}
\left\langle \left( f_{1},\Ff_{1}\right) ,\left( f_{2},\Ff_{2}\right) \right\rangle  &=&\left( \left( f_{1}\otimes f_{2}\right) \Delta
_{H},\left( 1\otimes \left( f_{2}\otimes f_{1}\right) \left( \Rr %
^{-1}\right) \otimes 1\right)(\id \ot\tau\ot\id )(\Ff_{1}\ot\Ff_{2}) \right)  .
\end{eqnarray*}
By Lemma \ref{lem:isoTwTr}, the $1$-cell 
$\langle (f_{1},%
\Ff _{1}),(f_{2},\Ff _{2})\rangle :(H,\Rr)\rightarrow \left(
H_{1},\Rr _{1}\right) \otimes \left( H_{2},\Rr _{2}\right) $
is invertible in $\TwTrBialg$ if, and only if, 
$\left( f_{1}\otimes f_{2}\right) \Delta
_{H}:(H,\Rr)\to ((H_1\ot H_2)_\Ff,\widetilde{\Rr}_\Ff)$ is an isomorphism in $\TrBialg$, where $\Ff= \left( 1_{H_{1}}\otimes \left( f_{2}\otimes
f_{1}\right) \left( \Rr^{-1}\right) \otimes 1_{H_{2}}\right) \big(\left(
\id \otimes \tau \otimes \id \right) \left( \Ff %
_{1}\otimes \Ff _{2}\right)\big)$, i.e.\ $(f_{1}\ot f_{2})\Delta_{H}:H\to(H_{1}\ot H_{2})_{\Ff}$ is a bialgebra isomorphism and $((f_{1}\ot f_{2})\Delta_{H}\ot (f_{1}\ot f_{2})\Delta_{H})(\Rr)=\widetilde{\Rr}_{\Ff}$. To conclude, we observe that, given $\Rr_{i}:=(f_{i}\ot f_{i})(\Rr)_{\Ff_{i}^{-1}}$, for $i=1,2$, we have
\[
\Rr_{1}\ot\Rr_{2}=((\Ff_{1}^{-1})^{\mathrm{op}}\ot(\Ff_{2}^{-1})^{\mathrm{op}})(f_{1}\ot f_{1}\ot f_{2}\ot f_{2})(\Rr\ot\Rr)(\Ff_{1}\ot\Ff_{2})
\]
and then the following equality is satisfied:
\[
\begin{split}
\widetilde{\Rr}_{\Ff}&=(1\otimes \left( f_{2}\otimes f_{1}\right) \left( \Rr %
^{-1}\right) \otimes 1)^{\mathrm{op}}((\id \ot\tau\ot\id )(\Ff_{1}\ot\Ff_{2}))^{\mathrm{op}}\\&\hspace{0.5cm}(\id \ot\tau\ot\id )((\Ff_{1}^{-1})^{\mathrm{op}}\ot(\Ff_{2}^{-1})^{\mathrm{op}})
(f_{1}(\Rr^{i})\ot f_{2}(\Rr^{j})\ot f_{1}(\Rr_{i})\ot f_{2}(\Rr_{j}))(\id \ot\tau\ot\id )(\Ff_{1}\ot\Ff_{2})\\&\hspace{0.5cm}((\id \ot\tau\ot\id )(\Ff_{1}\ot\Ff_{2}))^{-1}(1\otimes \left( f_{2}\otimes f_{1}\right) \left( \Rr %
^{-1}\right) \otimes 1)^{-1}\\&=(f_{1}(\overline{\Rr}_{s})\ot1\ot1\ot f_{2}(\overline{\Rr}^{s}))(f_{1}(\Rr^{i})\ot f_{2}(\Rr^{j})\ot f_{1}(\Rr_{i})\ot f_{2}(\Rr_{j}))(1\ot f_{2}(\Rr^{l})\ot f_{1}(\Rr_{l})\ot 1)\\&\overset{\eqref{tr}}{=}(f_1(\Rr^{s})\ot1\ot1\ot f_{2}(\Rr_{s}))(f_{1}(\Rr^{i})\ot f_{2}(\Rr^{j})\ot f_{1}(\Rr_{i})\ot f_{2}(\Rr_{j}))(1\ot f_{2}(\Rr^{l})\ot f_{1}(\Rr_{l})\ot 1)\\&=(f_{1}\ot f_{2}\ot f_{1}\ot f_{2})((\Rr^{s}\Rr^{i}\ot\Rr^{j}\Rr^{l}\ot \Rr_{i}\Rr_{l}\ot\Rr_{s}\Rr_{j}))\\&\overset{\eqref{eq:DDR}}=(f_{1}\ot f_{2}\ot f_{1}\ot f_{2})(\Delta\ot\Delta)(\Rr).
\end{split}
\]
This computation completes the proof. 
%
%
\end{proof}

As noticed in Remark \ref{rmk:twistedasinv1cell}, a triangular bialgebra $(H,\Rr)$ is a twisted tensor product of triangular bialgebras $(H_{1},\Rr_{1})$ and $(H_{2},\Rr_{2})$ if and only if there is an invertible 1-cell $(f,\Ff):(H,\Rr)\to(H_{1}\ot H_{2},\widetilde{\Rr})$ in $\TwTrBialg$. Using this, we end this section providing some simple examples of twisted tensor products.

\begin{example}
In view of Remark \ref{rmk:tensorproductmorphism} and Lemma \ref{lem:isoTwTr}, we obtain the following basic examples.

1). Given the Sweedler Hopf algebra $(H,\Rr_{\lambda})$, we can define invertible 1-cells $(f_{s},\Ff_{d}):(H,\Rr_{\lambda})\to(H,\Rr_{\lambda s^{2}-2d})$ in $\TwTrBialg$, for $0\not=s\in\Bbbk$ and $d\in\Bbbk$, as in Example \ref{exa:Sw1}. Thus
\[ (f_{s},\Ff_{d})\otimes (\id,1\ot 1)\overset{\eqref{eq:tns1cell}}=   (f_{s}\ot\id ,(\Ff_{d})^i\ot1\ot(\Ff_{d})_{i}\ot1):(H,\Rr_{\lambda})\ot(K,1\ot1)\to(H,\Rr_{\lambda s^{2}-2d})\ot(K,1\ot1)
\]   
    is an invertible 1-cell in $\TwTrBialg$, for any cocommutative bialgebra $K$. Therefore, $(H,\Rr_{\lambda})\ot(K,1\ot1)$ is a twisted tensor product of $(H,\Rr_{\lambda s^{2}-2d})$ and $(K,1\ot1)$. 
    
    2). Given the abelian group $\Gamma=\langle x,y\ |\ xy=yx,\ x^{n}=1,\ y^{n}=1\rangle$, for $n>1$, we can define invertible 1-cells $(f,\Ff):(\Bbbk\Gamma,\Rr)\to(\Bbbk\Gamma,\Rr')$ as in Example \ref{ex:groupalgebra}. Then
\[
(f_s\ot\Ff_d)\ot(f,\Ff)\overset{\eqref{eq:tns1cell}}=(f_{s}\ot f,(\id \ot\tau\ot\id )(\Ff_{d}\ot\Ff)):(H,\Rr_{\lambda})\ot(\Bbbk\Gamma,\Rr)\to(H,\Rr_{\lambda s^{2}-2d})\ot(\Bbbk\Gamma,\Rr')
\]
is an invertible 1-cell in $\TwTrBialg$. Therefore, $(H,\Rr_{\lambda})\ot(\Bbbk\Gamma,\Rr)$ is a twisted tensor product of $(H,\Rr_{\lambda s^{2}-2d})$ and $(\Bbbk\Gamma,\Rr')$.
\end{example}

\medskip

\noindent\emph{Acknowledgements.} The authors would like to thank Paolo Aschieri, Lucrezia Bottegoni and Alan Cigoli for insightful discussions and useful comments that contributed to the development of our study. The authors also thank the referee for several useful suggestions. This paper was written while the authors were members of the “National Group for Algebraic and Geometric Structures and their Applications” (GNSAGA-INdAM). It is partially based upon work from COST Action CaLISTA CA21109 supported by COST (European Cooperation in Science and Technology). 
The authors were also partially supported by the project funded by the European Union -NextGenerationEU under NRRP, Mission 4 Component 2 CUP D53D23005960006 - Call PRIN 2022 No. 104 of February 2, 2022 of Italian Ministry of University and Research; Project 2022S97PMY \emph{Structures for Quivers, Algebras and Representations} (SQUARE). \medskip



\bibliography{semqtri}

@incollection {Barr,
    AUTHOR = {Barr, Michael},
     TITLE = {Exact categories},
 BOOKTITLE = {Exact categories and categories of sheaves},
    SERIES = {Lecture Notes in Math.},
    VOLUME = {236},
     PAGES = {1--120},
 PUBLISHER = {Springer, Berlin},
      YEAR = {1971},
      ISBN = {3-540-05678-5; 0-387-05678-5},
   MRCLASS = {18E10},
  MRNUMBER = {3727442},
       DOI = {10.1007/BFb0058579},
       URL = {https://doi-org.bibliopass.unito.it/10.1007/BFb0058579},
}

@inproceedings {Dr87,
    AUTHOR = {Drinfel'd, V. G.},
     TITLE = {Quantum groups},
 BOOKTITLE = {Proceedings of the {I}nternational {C}ongress of
              {M}athematicians, {V}ol. 1, 2 ({B}erkeley, {C}alif., 1986)},
     PAGES = {798--820},
 PUBLISHER = {Amer. Math. Soc., Providence, RI},
      YEAR = {1987},
      ISBN = {0-8218-0110-4},
   MRCLASS = {17B50 (16A24 17B65 57T05 58F07 82A05 82A15)},
  MRNUMBER = {934283},
}

@book {Majid-book,
    AUTHOR = {S. Majid},
     TITLE = {Foundations of quantum group theory},
 PUBLISHER = {Cambridge University Press, Cambridge},
      YEAR = {1995},
     PAGES = {x+607},
      ISBN = {0-521-46032-8},
   MRCLASS = {17B37 (18D99 81R50)},
  MRNUMBER = {1381692},
MRREVIEWER = {Dmitri\u{\i} I. Gurevich},
       DOI = {10.1017/CBO9780511613104},
       URL = {https://doi.org/10.1017/CBO9780511613104},
}

@incollection {Kassel,
    AUTHOR = {Kassel, Ch.},
     TITLE = {Quantum groups},
 BOOKTITLE = {Algebra and operator theory ({T}ashkent, 1997)},
     PAGES = {213--236},
 PUBLISHER = {Kluwer Acad. Publ., Dordrecht},
      YEAR = {1998},
      ISBN = {0-7923-5094-4},
   MRCLASS = {17B37 (18D10 81R50)},
  MRNUMBER = {1643398},
MRREVIEWER = {Alejandro\ Tiraboschi},
}

@article {Chen-quasi,
    AUTHOR = {Chen, Hui-Xiang},
     TITLE = {Quasitriangular structures of bicrossed coproducts},
   JOURNAL = {J. Algebra},
  FJOURNAL = {Journal of Algebra},
    VOLUME = {204},
      YEAR = {1998},
    NUMBER = {2},
     PAGES = {504--531},
      ISSN = {0021-8693,1090-266X},
   MRCLASS = {16W30},
  MRNUMBER = {1624475},
MRREVIEWER = {Horia\ Pop},
       DOI = {10.1006/jabr.1997.7381},
       URL = {https://doi.org/10.1006/jabr.1997.7381},
}

@incollection {Davydov,
    AUTHOR = {Davydov, A.},
     TITLE = {Twisted automorphisms of {H}opf algebras},
 BOOKTITLE = {Noncommutative structures in mathematics and physics},
     PAGES = {103--130},
 PUBLISHER = {K. Vlaam. Acad. Belgie Wet. Kunsten (KVAB), Brussels},
      YEAR = {2010},
      ISBN = {978-90-6569-061-6},
   MRCLASS = {16T05 (16T10 16T25 17B35)},
  MRNUMBER = {2742734},
MRREVIEWER = {Gabriella\ B\"ohm},
}

@book {BorI94,
    AUTHOR = {Borceux, Francis},
     TITLE = {Handbook of categorical algebra. 1},
    SERIES = {Encyclopedia of Mathematics and its Applications},
    VOLUME = {50},
      NOTE = {Basic category theory},
 PUBLISHER = {Cambridge University Press, Cambridge},
      YEAR = {1994},
     PAGES = {xvi+345},
      ISBN = {0-521-44178-1},
   MRCLASS = {18-02 (18Axx)},
  MRNUMBER = {1291599},
MRREVIEWER = {Martin\ Hyland},
}

@book {Johnson-Yau-book,
    AUTHOR = {Johnson, Niles and Yau, Donald},
     TITLE = {2-dimensional categories},
 PUBLISHER = {Oxford University Press, Oxford},
      YEAR = {2021},
     PAGES = {xix+615},
      ISBN = {978-0-19-887138-5; 978-0-19-887137-8},
   MRCLASS = {18-02 (18N10 18N15 18N20)},
  MRNUMBER = {4261588},
MRREVIEWER = {Robert\ Laugwitz},
       DOI = {10.1093/oso/9780198871378.001.0001},
       URL = {https://doi.org/10.1093/oso/9780198871378.001.0001},
}

@article {CKWW2,
    AUTHOR = {Carboni, A. and Kelly, G. M. and Walters, R. F. C. and Wood,
              R. J.},
     TITLE = {Cartesian bicategories {II}},
   JOURNAL = {Theory Appl. Categ.},
  FJOURNAL = {Theory and Applications of Categories},
    VOLUME = {19},
      YEAR = {2007},
     PAGES = {93--124},
      ISSN = {1201-561X},
   MRCLASS = {18D05 (18D20)},
  MRNUMBER = {3656673},
}

@article {KM,
    AUTHOR = {Krähmer, U. and Mahaman, M.},
     TITLE = {Clones from comonoids},
   JOURNAL = {Revista de la
Unión Matemática Argentina},
  FJOURNAL = {},
    VOLUME = {},
      YEAR = {2024},
     PAGES = {},
      ISSN = {},
   MRCLASS = {},
  MRNUMBER = {},
}

@article {Agore,
    AUTHOR = {Agore, A. L.},
     TITLE = {Limits of coalgebras, bialgebras and {H}opf algebras},
   JOURNAL = {Proc. Amer. Math. Soc.},
  FJOURNAL = {Proceedings of the American Mathematical Society},
    VOLUME = {139},
      YEAR = {2011},
    NUMBER = {3},
     PAGES = {855--863},
      ISSN = {0002-9939,1088-6826},
   MRCLASS = {16T15 (16T05 16T10)},
  MRNUMBER = {2745638},
MRREVIEWER = {Serge\ M.\ Skryabin},
       DOI = {10.1090/S0002-9939-2010-10542-7},
       URL = {https://doi.org/10.1090/S0002-9939-2010-10542-7},
}

@article {Pfeiffer,
    AUTHOR = {Pfeiffer, Hendryk},
     TITLE = {2-groups, trialgebras and their {H}opf categories of
              representations},
   JOURNAL = {Adv. Math.},
  FJOURNAL = {Advances in Mathematics},
    VOLUME = {212},
      YEAR = {2007},
    NUMBER = {1},
     PAGES = {62--108},
      ISSN = {0001-8708,1090-2082},
   MRCLASS = {16W30},
  MRNUMBER = {2319763},
MRREVIEWER = {Alessandro\ Ardizzoni},
       DOI = {10.1016/j.aim.2006.09.014},
       URL = {https://doi.org/10.1016/j.aim.2006.09.014},
}

@article {Davydov-TwDer,
    AUTHOR = {Davydov, A.},
     TITLE = {Twisted derivations of {H}opf algebras},
   JOURNAL = {J. Pure Appl. Algebra},
  FJOURNAL = {Journal of Pure and Applied Algebra},
    VOLUME = {217},
      YEAR = {2013},
    NUMBER = {3},
     PAGES = {567--582},
      ISSN = {0022-4049,1873-1376},
   MRCLASS = {16W25 (16T05 17B99)},
  MRNUMBER = {2974231},
MRREVIEWER = {Ki-Bong\ Nam},
       DOI = {10.1016/j.jpaa.2012.08.004},
       URL = {https://doi.org/10.1016/j.jpaa.2012.08.004},
}

@misc{Lomp,
      title={GENERALIZED {K}AC-{P}ALJUTKIN ALGEBRAS}, 
      author={Christian Lomp},
      year={2025},
      eprint={????},
 archivePrefix={arXiv},
  howpublished = {\href{https://arxiv.org/pdf/2505.00645}{arXiv:2505.00645}},
}

@misc{Wu-Zhou,
      title={Splitting property of quasitriangular {H}opf
algebras}, 
      author={J. Wu and K. Zhou},
      year={2025},
      eprint={????},
 archivePrefix={arXiv},
  howpublished = {\href{https://arxiv.org/abs/2501.15401v2}{arXiv:2501.15401v2}},
}

@article {Schneider,
    AUTHOR = {Schneider, H.-J.},
     TITLE = {Some properties of factorizable {H}opf algebras},
   JOURNAL = {Proc. Amer. Math. Soc.},
  FJOURNAL = {Proceedings of the American Mathematical Society},
    VOLUME = {129},
      YEAR = {2001},
    NUMBER = {7},
     PAGES = {1891--1898},
      ISSN = {0002-9939,1088-6826},
   MRCLASS = {16W30},
  MRNUMBER = {1825894},
MRREVIEWER = {Mitsuhiro\ Takeuchi},
       DOI = {10.1090/S0002-9939-01-05787-2},
       URL = {https://doi.org/10.1090/S0002-9939-01-05787-2},
}

@article{ShilinZhang,
    AUTHOR = {Yang, Shilin and Zhang, Yongfeng},
     TITLE = {Ore extensions of automorphism type for {H}opf algebras},
   JOURNAL = {Bull. Iranian Math. Soc.},
  FJOURNAL = {Bulletin of the Iranian Mathematical Society},
    VOLUME = {46},
      YEAR = {2020},
    NUMBER = {2},
     PAGES = {487--501},
      ISSN = {1017-060X,1735-8515},
   MRCLASS = {16S36 (16T05 16T10)},
  MRNUMBER = {4074085},
MRREVIEWER = {YanHua\ Wang},
       DOI = {10.1007/s41980-019-00271-x},
       URL = {https://doi.org/10.1007/s41980-019-00271-x},
}

@incollection {Pansera,
    AUTHOR = {Pansera, Deividi},
     TITLE = {A class of semisimple {H}opf algebras acting on quantum
              polynomial algebras},
 BOOKTITLE = {Rings, modules and codes},
    SERIES = {Contemp. Math.},
    VOLUME = {727},
     PAGES = {303--316},
 PUBLISHER = {Amer. Math. Soc., Providence, RI},
      YEAR = {2019},
      ISBN = {978-1-4704-4104-3},
   MRCLASS = {16T05},
  MRNUMBER = {3938158},
MRREVIEWER = {Julia\ Yael\ Plavnik},
       DOI = {10.1090/conm/727/14643},
       URL = {https://doi.org/10.1090/conm/727/14643},
}

@incollection {Bourn,
    AUTHOR = {Bourn, Dominique},
     TITLE = {Normalization equivalence, kernel equivalence and affine
              categories},
 BOOKTITLE = {Category theory ({C}omo, 1990)},
    SERIES = {Lecture Notes in Math.},
    VOLUME = {1488},
     PAGES = {43--62},
 PUBLISHER = {Springer, Berlin},
      YEAR = {1991},
      ISBN = {3-540-54706-1},
   MRCLASS = {18E05 (18A20 18A30 18D30)},
  MRNUMBER = {1173004},
MRREVIEWER = {Walter\ Tholen},
       DOI = {10.1007/BFb0084212},
       URL = {https://doi.org/10.1007/BFb0084212},
}

@article {BJK,
    AUTHOR = {Borceux, F. and Janelidze, G. and Kelly, G. M.},
     TITLE = {Internal object actions},
   JOURNAL = {Comment. Math. Univ. Carolin.},
  FJOURNAL = {Commentationes Mathematicae Universitatis Carolinae},
    VOLUME = {46},
      YEAR = {2005},
    NUMBER = {2},
     PAGES = {235--255},
      ISSN = {0010-2628,1213-7243},
   MRCLASS = {18C15 (18C20 18D10 18D15 18G50)},
  MRNUMBER = {2176890},
MRREVIEWER = {Steve\ Awodey},
}

@article {JMT,
    AUTHOR = {Janelidze, George and M\'arki, L\'aszl\'o{} and Tholen,
              Walter},
     TITLE = {Semi-abelian categories},
      NOTE = {Category theory 1999 (Coimbra)},
   JOURNAL = {J. Pure Appl. Algebra},
  FJOURNAL = {Journal of Pure and Applied Algebra},
    VOLUME = {168},
      YEAR = {2002},
    NUMBER = {2-3},
     PAGES = {367--386},
      ISSN = {0022-4049,1873-1376},
   MRCLASS = {18E10 (18B99)},
  MRNUMBER = {1887164},
MRREVIEWER = {Ana\ Jerem\'ias L\'opez},
       DOI = {10.1016/S0022-4049(01)00103-7},
       URL = {https://doi.org/10.1016/S0022-4049(01)00103-7},
}

@article {GSV,
    AUTHOR = {Gran, Marino and Sterck, Florence and Vercruysse, Joost},
     TITLE = {A semi-abelian extension of a theorem by {T}akeuchi},
   JOURNAL = {J. Pure Appl. Algebra},
  FJOURNAL = {Journal of Pure and Applied Algebra},
    VOLUME = {223},
      YEAR = {2019},
    NUMBER = {10},
     PAGES = {4171--4190},
      ISSN = {0022-4049,1873-1376},
   MRCLASS = {16T05 (16B50 18B40 18D35 18E10 18G50)},
  MRNUMBER = {3958087},
MRREVIEWER = {Luz\ Adriana\ Mej\'ia Casta\~no},
       DOI = {10.1016/j.jpaa.2019.01.004},
       URL = {https://doi.org/10.1016/j.jpaa.2019.01.004},
}
\bibliographystyle{acm}

\end{document}